\documentclass[reqno, 12pt, reqno]{amsart}
\usepackage[romanian,english]{babel}
\usepackage{combelow}

\usepackage{amsmath,amssymb,amsthm,bm,graphicx,mathrsfs,mathtools, amsfonts, dsfont,multicol}

\usepackage[pagebackref=true]{hyperref}

\usepackage{bbm}
\usepackage{enumitem}

\usepackage[margin=3.2cm]{geometry}

\RequirePackage{mathrsfs} \let\mathcal\mathscr
\numberwithin{equation}{section}

\renewcommand{\d}{\mathrm{d}}
\renewcommand{\phi}{\varphi}
\renewcommand{\rho}{\varrho}

\newcommand{\0}{\mathbf{0}}
\newcommand{\PP}{\mathbb{P}}
\renewcommand{\AA}{\mathbb{A}}

\newcommand{\FF}{\mathbb{F}}
\newcommand{\ZZ}{\mathbb{Z}}

\newcommand{\NN}{\mathbb{N}}
\newcommand{\QQ}{\mathbb{Q}}
\newcommand{\RR}{\mathbb{R}}
\newcommand{\CC}{\mathbb{C}}

\newcommand{\TT}{\mathbb{T}}

\DeclareMathOperator{\tr}{Tr}

\def \bfg { {\bf g}}

\renewcommand{\leq}{\leqslant}

\renewcommand{\geq}{\geqslant}

\newcommand{\h}{\mathbf{h}}

\newcommand{\x}{\mathbf{x}}
\newcommand{\y}{\mathbf{y}}

\renewcommand{\v}{\mathbf{v}}
\newcommand{\uu}{\mathbf{u}}

\newcommand{\z}{\mathbf{z}}
\newcommand{\w}{\mathbf{w}}
\renewcommand{\b}{\mathbf{b}}
\renewcommand{\a}{\mathbf{a}}
\renewcommand{\k}{\mathbf{k}}

\newcommand{\ve}{\varepsilon}

\newcommand{\bal}{\boldsymbol{\alpha}}
\newcommand{\bbe}{\boldsymbol{\beta}}

\newcommand{\bga}{\boldsymbol{\gamma}}

\DeclareMathOperator{\rank}{rank}

\DeclareMathOperator{\meas}{meas}

\DeclareMathOperator{\chara}{char}
\def \Mor {\operatorname{Mor}}

\newcommand{\vp}{\varpi}

\renewcommand{\hat}{\widehat}
\newcommand{\Ki}{K_\infty}

\newcommand{\1}{\mathbf{1}}

\newcommand{\g}{\mathbf{g}}

\newtheorem{theorem}{Theorem}[section]
\newtheorem{lemma}[theorem]{Lemma}

\newtheorem{prop}[theorem]{Proposition}
\newtheorem{corollary}[theorem]{Corollary}

\theoremstyle{remark}
\newtheorem{remark}[theorem]{Remark}

\newtheorem{hyp}[theorem]{Hypothesis}

\theoremstyle{definition}

\newtheorem{defi}[theorem]{Definition}

\newtheorem*{ack}{Acknowledgements}

\newcommand{\f}{\mathbf{f}}

\DeclareMathOperator{\ord}{ord}

\usepackage[usenames,dvipsnames]{color}

\title[Rational curves on complete intersections]{Rational curves on complete intersections and the circle method}

\author{Tim Browning}
\address{IST Austria\\
Am Campus 1\\
3400 Klosterneuburg\\
Austria}
\email{tdb@ist.ac.at}

\author{Pankaj Vishe}
\address{Department of Mathematical Sciences\\ Durham University\\ Durham\\ DH1 3LE\\ United Kingdom}
\email{pankaj.vishe@durham.ac.uk}

\author{Shuntaro Yamagishi}
\address{IST Austria\\
Am Campus 1\\
3400 Klosterneuburg\\
Austria}
\email{shuntaro.yamagishi@ist.ac.at}


\date{\today}

\begin{document}

\begin{abstract}
We study the geometry of the
space of rational curves on  smooth complete intersections  of low degree, which
pass through a given set of points on the variety.  The argument uses spreading out to a  finite field, together with an adaptation
to function fields of positive characteristic
of   work by Rydin Myerson
 on the circle method.  Our work also allows us to handle weak approximation for such varieties.
\end{abstract}

\maketitle
\thispagestyle{empty}

\setcounter{tocdepth}{1}
\tableofcontents

\section{Introduction}

\subsection*{Geometry of rational curves}
Let $d\geq 2$ be an integer and let
 $k$ be a field
whose characteristic exceeds $d$ if it is positive.
Let $X\subset \PP^{n-1}$ be a smooth complete intersection over $k$,
cut out by $R$ hypersurfaces of the same degree $d$.  We will always assume that
$n\geq Rd$, so that  $X$ is Fano.
In this paper we shall be interested in the geometry of the moduli space of degree $e$ rational curves on $X$ that pass through a given set of points on $X$.

To begin with, let $\mathcal M_{0,0}(X,e)$ be the moduli space
of degree $e$ rational curves on $X$.
As explained by Harris, Roth and Starr \cite[Lemma~4.2]{HRS},
basic deformation theory shows that every irreducible component of
$\mathcal M_{0,0}(X,e)$ has dimension at least
$$
\mu(e,R)= e(n-Rd)+n-R-4,
$$
which we refer to as the expected dimension.
In the special case $R=1$ of smooth degree $d$ hypersurfaces $X\subset \PP^{n-1}$, the space
$\mathcal M_{0,0}(X,e)$ has been the focus of a great deal of study.
For  $d=2$,
Kim and Pandharipande  \cite[Cor.~1]{kim} have addressed the irreducibility and  dimension of
 $\mathcal{M}_{0,0}(X,e)$. For $d=3$, similar results hold thanks to work of Coskun and Starr \cite{CS} if
 $n\geq 5$.  Using
a version of Bend-and-Break, the best result for generic degree $d$ hypersurfaces $X$ is due to
 Riedl and Yang \cite{RY}, who  establish that
 $\mathcal M_{0,0}(X,e)$ is irreducible and of the  expected dimension $\mu(e,1)$, provided that
 $n\geq d+3$.
Recent work of  Browning and Sawin \cite[Thm.~1.1]{BS2} achieves the same conclusion for any  smooth hypersurface $X\subset \PP^{n-1}$ of degree $d$, provided that
$n> (2d - 1)2^{d-1}$.
The latter result is proved using analytic number theory and builds on an approach
employed by Browning and Vishe \cite{BV'}.
The idea is to study
the moduli space of maps
$\Mor_e(\PP^1,X)$, whose expected dimension is $\mu(e,R)+3$, since
$\mathcal{M}_{0,0}(X,e)$ is obtained from $\Mor_e(\PP^1,X)$ on
taking into account an action by $\mathrm{PGL}_2$.
To calculate the dimension of $\Mor_e(\PP^1,X)$,
one  begins by working over a finite field $k=\FF_q$ of  characteristic $>d$ and counts
the  number of  points on it that are defined over a finite extension of $\FF_q$.
Finally, one combines this count with the
Lang--Weil estimate to control the irreducibility and the dimension.
In this paper we shall use the analytic number theory point of view to treat
 complete intersections with $R\geq 1$, beginning with the  following result.

\begin{theorem}\label{t:BV}
Let $d\geq 2$ and let
 $k$ be a field whose characteristic exceeds $d$ if it is positive.
Let $X\subset \PP^{n-1}$ be a smooth complete intersection over $k$,
cut out by $R$ hypersurfaces of  degree $d$.
Assume that $e\geq 1$ and
\begin{equation}\label{eq:n-bound}
n\geq
\begin{cases}
d(d-1)2^{d+1}R+R  &\text{ if $d\geq 3$,}\\
33R & \text{ if $d=2$.}
\end{cases}
\end{equation}
Then
$\mathcal{M}_{0,0}(X,e)$
is an irreducible locally complete intersection of dimension $\mu(e,R)$.
\end{theorem}

For comparison, it follows from work of Beheshti and Kumar \cite[Thm.~7.3]{BK} that a similar statement holds for {\em generic} complete intersections under the assumption $dR<(2n+R-1)/3$, which is equivalent to
$n>\frac{1}{2}(3dR-R+1)$.

Let $X$ be a smooth complete intersection over a field $k$, as in
the statement of Theorem \ref{t:BV}.
Let $B\subset \PP^1$ be the closed subscheme formed from
 $p_1,\dots,p_b\in \PP^1$. Consider the  map
\begin{equation}\label{eq:red}
\mathrm{red}: \Mor_e(\PP^1,X)\to \Mor(B,X)
\end{equation}
that is obtained by reducing to the subscheme $B$, where
$\Mor_e(\PP^1,X)$ is the
moduli space of maps and $\Mor(B,X)$ is isomorphic to $X^b$.
For any
 $y_1,\dots,y_b\in X$, we define
\begin{equation}\label{eq:def_M}
\mathrm{M}_{e,b}=\Mor_e(\PP^1,X;p_1,\dots,p_b;y_1,\dots,y_b)
\end{equation}
to be the moduli space of degree $e$ morphisms
$g:\PP^1\to X$ such that
$g(p_j)=y_j$ for $1\leq j\leq b$.  Thus
$\mathrm{M}_{e,b}$ is the  fibre of the map \eqref{eq:red} over
$(y_1,\dots,y_b)\in X^b$.
(Note that  $\mathrm{M}_{e,0}=\Mor_e(\PP^1,X)$
when  $b=0$.)
As explained by Debarre \cite[\S~2.3]{Debarre},
viewed as a fibre of the reduction map,
the expected dimension of $\mathrm{M}_{e,b}$ is
\begin{align*}
\dim \Mor_e(\PP^1,X)-b\dim X&=\mu(e,R)+3-b(n-1-R)\\
&= e(n-Rd) +(n-1-R)(1-b).
\end{align*}
In his lecture at the Banff workshop
``Geometry via Arithmetic'' in July 2021, Will Sawin  raised the question of
tackling the irreducibility of  the space $\mathrm{M}_{e,b}$ and calculating its dimension, pointing out that methods from
algebraic geometry haven't yet been made to yield this information, even  for generic complete intersections $X$.
This is the object of the following result.

\begin{theorem}\label{t:BV2}
Let $d\geq 2$ and let
 $k$ be a field whose characteristic exceeds $d$ if it is positive.
Let $X\subset \PP^{n-1}$ be a smooth complete intersection over $k$,
cut out by $R$ hypersurfaces of degree $d$.
Let $p_1,\dots,p_b\in \PP^1$ and let $y_1,\dots,y_b\in X$, for $b\geq 1$.
Assume that
$
e\geq (d+1+R)b
$
and \eqref{eq:n-bound} holds.
Then
$\mathrm{M}_{e,b}$
is  irreducible and has the expected  dimension.
\end{theorem}

When $d=3$ and $n\geq 10$, work of
M\^{a}nz\u{a}\cb{t}eanu \cite[Cor.~1.4]{adelina} addresses the
geometry of the space
$\mathrm{M}_{e,2}$ in the special case  $R=1$ of hypersurfaces.
Since $X$ is smooth, the Hessian covariant is non-vanishing.
Assuming that
$y_1,y_2\in X$ don't both vanish on the Hessian, and  that
$e \geq 19-\frac{2}{n-9}$, it is shown that
$\mathrm{M}_{e,2}$
is  irreducible and of the expected dimension.
Theorem \ref{t:BV2} arrives at the same conclusion for any $e\geq 10$,
provided that $n\geq 97$.

\subsection*{Arithmetic over function fields}

Our work on the geometry of $\mathrm{M}_{e,b}$ goes via a spreading out argument, as in \cite{BS2, BV'}, which will lead us to study the size of
$\#\mathrm{M}_{e,b}(\FF_q)$ for a suitable finite field $\FF_q$.
This in turn will be accessed through
a version of the
Hardy--Littlewood circle method
over the function field $\FF_q(t)$.
Let  $X\subset \PP^{n-1}$
 be a smooth complete intersection
cut out by $R$ hypersurfaces of equal degree $d\geq 2$,
all of which are defined over $\FF_q$. In suitable circumstances we shall be able to count $\FF_q(t)$-rational points on $X$ of bounded height via the circle method.
Over $\QQ$, this is carried out in a  classic paper of Birch \cite{birch}, which
is capable of proving that the Hasse principle and weak approximation hold, provided that $n\geq R+R(R+1)2^{d-1}(d-1)$, a result that has already been extended to $\FF_q(t)$ by
Lee \cite{lee}. In fact,
as described in \cite[Thms.~3.4 and 3.6]{green},
it follows from
the Lang--Nagata--Tsen theorems that $X(\FF_q(t))\neq \emptyset$ whenever $n > Rd^2$.

The  particular version of the circle method we shall use  is a function field analogue of the work recently carried out over $\QQ$ by Rydin Myerson \cite{simon, simon'}. It has the  key feature that it only depends linearly on the number $R$ of equations, rather than quadratically as in the case of Birch's result.
To be precise, if  $n\geq d 2^d R+R$ and the system of forms $f_1,\dots,f_R$ is  {\em suitably generic}, then one can conclude the Hasse principle and weak approximation over $\QQ$.
(Note that the genericity condition on the system of forms can be removed when $d=2$ \cite{simon} and when $d=3$ \cite{simon'}.)
When working in positive characteristic, it turns out that we can
remove the genericity condition for every $d$, provided that the finite field is large enough and we make a stronger assumption on the number of variables.

\begin{theorem}\label{t:simon}
Let $d,n,R\in \NN$ such that $d\geq 2$ and
$$
n>
\begin{cases}
d(d-1)2^{d}R+R  &\text{ if $d\geq 3$,}\\
17R & \text{ if $d=2$.}
\end{cases}
$$
 Assume that $\FF_q$ is a finite field such that
  $\chara(\FF_q)>d$
and $q\geq (d-1)^n$.
Let $f_1,\dots,f_R\in \FF_q[x_1,\dots,x_n]$ be forms of degree $d$ cutting out a
 smooth complete intersection  $X\subset \PP^{n-1}$.
Then $X$ satisfies  weak approximation over $\FF_q(t)$.
\end{theorem}

Our assumption on the number of variables is more stringent than the
bound $n\geq d 2^d R+R$ obtained by Rydin Myerson for generic forms, although it retains the feature that it is  linear in the number $R$ of equations.    This can be traced to the nature of the saving in the main new technical ingredient, provided in Theorem
\ref{t:Nbf}, which allows us to handle {\em all} forms defining a smooth complete intersection in the function field setting.

\subsection*{Low degree rational curves}

Let us return to the setting of Theorem \ref{t:BV}, specialised to the case $k=\FF_q$,  a finite field whose characteristic exceeds $d$.
Under the assumptions of the theorem our work ensures that
$\mathcal{M}_{0,0}(X,e)$
is an irreducible locally complete intersection of dimension $\mu(e,R)$, for each $e\geq 1$.
In the setting $d=3$ and $R=1$ of smooth cubic hypersurfaces,
Koll\'ar \cite{kollar} ask about an explicit lower bound on $e$,  sufficient to ensure that $X$ contains a degree $e$ rational curve defined over  $\FF_q$.
Building on the proof of Theorem \ref{t:simon},
our final result does exactly this in the broader setting of complete intersections, provided that  $q$ is sufficiently large.

\begin{theorem}\label{thm:low}
There exists a constant $c_{d,n}>0$ depending only on $d$ and $n$ such that the following holds.
Let $d\geq 2$ and let $\FF_q$ be a finite field such that
$\mathrm{char}(\FF_q)>d$ and $q\geq c_{d,n}$.
Let $X\subset \PP^{n-1}$ be a smooth complete intersection over $\FF_q$,
cut out by $R$ hypersurfaces of  degree $d$.
Assume that  \eqref{eq:n-bound} holds and
$e\geq 2(d-1)R+3d
$.
Then
$\mathcal{M}_{0,0}(X,e)(\FF_q)\neq \emptyset$.
\end{theorem}

Let $X\subset\PP^{n-1}$ be a smooth cubic hypersurface over $\FF_q$.
 In \cite[Example~7.6]{kollar}, Koll\'ar proves that for any  $q \geq  c_{3,n}$ and any point $y \in X(\FF_q)$, there exists a $\FF_q$-rational curve of degree at most $216$ on $X$ passing through $y$.
Although it will not be pursued here, it would be possible to
adapt the proof of Theorem \ref{thm:low} to yield explicit conditions on $e$ under which
$\mathrm{M}_{e,b}(\FF_q)\neq \emptyset$, for any given $b\geq 1$ and any smooth complete intersection $X\subset \PP^{n-1}$ for which
 \eqref{eq:n-bound} holds.

\begin{ack}
The authors are very grateful to  Lo\"is Faisant  and Jakob Glas for several useful comments.
The first and third authors were supported  by
a FWF grant (DOI 10.55776/P32428). Part of this work was supported by the Swedish Research
Council under grant no. 2016-06596, while the first and second authors were in residence at the Mittag-Leffler Institute in 2024.
\end{ack}

\section{Background on finite fields and  function fields}

\subsection{Points on varieties over finite fields}

The Lang--Weil estimate can be used to give an upper bound for the number of
$\FF_q$-points on a variety over a finite field $\FF_q$.
We will need a general version of this that features  convenient explicit constants. The following bound is proved in  \cite[Lemma 2.1]{CG}, but we include a full proof here for the sake of completeness.

\begin{lemma}
\label{LWbdd}
Let $X \subset \mathbb{A}^n$ be an affine variety of dimension $r$ and degree $\delta$. Then
$
\# X(\FF_q) \leq \delta q^r.
$
\end{lemma}
\begin{proof}
We argue by induction on $r$. When $r=0$ then $X$ consists of at most $\delta$ points and the lemma is trivial.
Assume now that we have $r \geq 1$.
Let $Z$ be an irreducible component of $X$, with  $\dim Z \geq 1$.
We shall show that there is an index $1 \leq i \leq n$ such that
the intersection of $Z$ with  the hyperplane $H_{\alpha} = \{ x_i = \alpha \}$ satisfies
$$
\dim (Z \cap H_{\alpha} ) < \dim Z,
$$
for any $\alpha \in \FF_q$. Suppose otherwise, so that there exists $\alpha_1,\dots,\alpha_n \in \FF_q$,  such that
$
Z\subset \{ x_j = \alpha_j \}
$
for each $1 \leq j \leq n$.
However, this means that $Z = \{ (\alpha_1, \ldots, \alpha_n) \}$, which contradicts the assumption that $\dim Z \geq 1$.
Since $Z \cap H_{\alpha}$ has dimension at most $\dim Z -1$ and degree at most $\deg Z$, it follows by the induction hypothesis that
$$
\# Z(\FF_q) \leq \sum_{\alpha \in \FF_q} \#  (Z \cap H_{\alpha} ) (\FF_q) \leq \sum_{\alpha \in \FF_q} (\deg Z) q^{\dim Z - 1}  \leq  (\deg Z) q^{\dim Z}.
$$
Therefore, if we denote by $Z_1, \ldots, Z_u$ the irreducible   components of $X$, then we obtain
$$
\# X(\FF_q) \leq
\sum_{1 \leq i \leq u} \# Z_i(\FF_q) \leq \sum_{1 \leq i \leq u} (\deg Z_i) q^{\dim Z_i}  \leq  \sum_{1 \leq i \leq u} (\deg Z_i) q^{\dim X}.
$$
Since $\deg X = \sum_{1 \leq i \leq u} \deg Z_i$, the result follows.
\end{proof}

\subsection{Function field notation}
In this section we collect together some notation and basic facts concerning the function field $K=\FF_q(t)$.
Let
 $\Omega$ be the set of  places of $K$. These correspond to either monic irreducible polynomials $\varpi$ in $\FF_q[t]$, which we call the {\em finite primes},  or the {\em prime at infinity} $t^{-1}$ which we usually denote  by $\infty$.
The associated absolute value  $|\cdot|_v$ is either $|\cdot|_\vp$ for some prime $\vp\in \FF_q[t]$ or $|\cdot|_{\infty}$, according to whether $v$ is a finite or infinite place, respectively.
These  are given by
$$
|a/b|_\vp=\left(\frac{1}{q^{\deg \vp}}\right)^{\ord_\vp(a) - \ord_\vp(b)} \quad \text{ and }\quad
|a/b|_\infty=
q^{\deg a-\deg b},
$$
for any $a/b\in K^*$, where $\ord_\vp(a)$ is the non-negative integer $\ell$ satisfying
$\varpi^\ell \| a$. We extend these definitions to  $K$ by taking $|0|_\vp=|0|_\infty=0.$
We will usually just write $|\cdot|=|\cdot|_\infty$.

For $v\in \Omega$ we let $K_v$ denote the completion of $K$ at $v$ with respect to $|\cdot|_v$,
together with its ring of integers $\mathcal{O}_v$.
We can extend the absolute value at the infinite place to $K_\infty$ to get  a non-archimedean
absolute value
$|\cdot|:K_\infty\rightarrow \RR_{\geq 0}$
given by $|\alpha|=q^{\ord \alpha}$, where $\ord \alpha$ is the largest $i\in
\ZZ$ such that $a_i\neq 0$ in the
representation $\alpha=\sum_{i\leq N}a_it^i$.
In this context we adopt the convention $\ord 0=-\infty$ and $|0|=0$.
We extend this to vectors by setting
$
|\x|=\max_{1\leq i\leq n}|x_i|,
$
for any $\x\in K_\infty^n$. It satisfies  the ultrametric inequality
$
|\x + \y| \leq \max \{ |\x|, |\y| \},
$
for any $\x, \y \in K_\infty^n$. For given $\x,\b\in \FF_q[t]^n$  and $m \in \FF_q[t]$ we will  write $\x\equiv \b
\bmod{m}$ to mean that
$\x=\b+ m\y$ for some $\y\in \FF_q[t]^n$.

We may identify $K_\infty$ with the set
$$
\FF_q((t^{-1}))=\left\{\sum_{i\leq N}a_it^i: \mbox{for $a_i\in \FF_q$ and some $N\in\ZZ$} \right\}
$$
and put
$$
\TT=\{\alpha\in K_\infty: |\alpha|<1\}=\left\{\sum_{i\leq -1}a_it^i: \mbox{for $a_i\in \FF_q$}
\right\}.
$$
Since $\TT$ is a locally compact
additive subgroup of $K_\infty$ it possesses a unique Haar measure $\d
\alpha$, which is normalised so that
$
\int_\TT \d\alpha=1.
$
We can extend $\d\alpha$ to a (unique) translation-invariant measure on $\Ki$ in
such a way that
$$
\int_{\{\alpha\in\Ki :
|\alpha|< q^N
\}} \d \alpha= q^N,
$$
for any $N\in \ZZ$.
These measures also extend to $\TT^n$ and $\Ki^n$, for any $n\in \NN$.

Given $\alpha \in K_\infty$ we denote $\| \alpha  \| = |\{ \alpha \}|$, where 
$\{\alpha\} \in \TT$ is the {\em fractional part} of $\alpha$. The following result will prove useful.

\begin{lemma}\label{lem:measure}
For any $N\in \ZZ_{\geq 0}$ and any non-zero $h\in \FF_q[t]$, we
have
$$
\meas \{ \beta \in \TT: \| h \beta  \| < q^{-N} \} =  q^{- N}.
$$
\end{lemma}
\begin{proof}
We may suppose that 
 $h = c_s t^s + \cdots + c_0 \in \FF_q[t]$, with $c_s \neq 0$ and $s \geq 0$.
Note that
$$
\{ h \beta \} = \sum_{i \leq -1} (c_0 \beta_{i} + c_{1} \beta_{- 1 + i} + \cdots + c_s \beta_{- s + i}   ) t^{i},
$$
where we write  $\beta=\sum_{i \leq -1} \beta_{i} t^i$. Then  the restriction 
$$
 \| h \beta  \| = |\{ h \beta \}|  < q^{-N}
$$
is equivalent to $(\beta_{- 1}, \ldots, \beta_{- s - N})$ satisfying
$$
\begin{pmatrix}
c_0 & c_{1}&  \cdots  & \cdots  & c_s & 0  & \cdots  &  0 \\
  0 & c_{0}& c_{1} & \cdots & \cdots  &  c_s  & \cdots & 0 \\
 \vdots & \vdots & \ddots &  \ddots & \ddots&  \ddots & \ddots & \vdots \\
  0 & 0& \cdots & c_0 & c_{1} & \cdots & \cdots&  c_s 
\end{pmatrix}
\begin{pmatrix}
\beta_{-1} \\
 \beta_{ -2} \\
 \vdots \\
 \beta_{-s  - N} \\
\end{pmatrix}
=
\mathbf{0}.
$$
Since $c_s \neq 0$, the matrix on the left hand side has rank $N$. 
Therefore, there are precisely $q^{ s }$ choices for $(\beta_{- 1}, \ldots, \beta_{- s - N})$. 
It follows that 
$$
\meas \{ \beta \in \TT: \| h \beta  \| < q^{-N} \} = q^{s} \meas \{   \gamma \in \TT: |  \gamma  | < q^{-s - N} \} = q^{-N},
$$
as required.
\end{proof}

\subsection{Characters}\label{s:add-characters}
There is a non-trivial additive character $e_q:\FF_q\rightarrow \CC^*$ defined
for each $a\in \FF_q$ by taking
$e_q(a)=\exp(2\pi i \tr(a)/p)$, where $\tr: \FF_q\rightarrow \FF_p$ denotes the
trace map.
This character induces a non-trivial (unitary) additive character $\psi:
K_\infty\rightarrow \CC^*$ by defining $\psi(\alpha)=e_q(a_{-1})$ for any
$\alpha=\sum_{i\leq N}a_i t^i$ in $\Ki$.
We have the  basic orthogonality property
$$
\sum_{\substack{b\in \FF_q[t]\\ |b|< q^N}}\psi(\gamma b)=\begin{cases}
q^N & \mbox{if $|\gamma|< q^{-N}$,}\\
0 & \mbox{otherwise},
\end{cases}
$$
for any $\gamma \in \TT$ and any integer $N\geq 0$, as proved in  \cite[Lemma 7]{kubota}.
We also have
\begin{equation}\label{eq:ortho}
\int_{\{\alpha\in K_\infty: |\alpha|<q^N\}} \psi(\alpha \gamma) \d \alpha =\begin{cases}
q^N &\text{ if $|\gamma|<q^{-N}$,}\\
0 &\text{ otherwise,}
\end{cases}
\end{equation}
for any $\gamma\in K_\infty$ and $N\in \ZZ$, as proved in \cite[Lemma 1(f)]{kubota}.

\subsection{Covering a box by boxes}
\label{boxes}
Given $\z \in K_{\infty}^R$ and $N \in \ZZ$, we denote
$$
B_{N}(\z) = \{ \bal \in K_{\infty}^R:  | \bal - \z | < q^{N} \}.
$$
Let $M \in \ZZ$ with $M \leq  N$. We claim that we can always
 cover $B_{{N}}(\z)$ by
 at most $q^{R (N  - M)}$
boxes of the form $B_{{M}}(\y)$.  If $M=N$ there is nothing to prove.
Suppose $M<N$ and assume without loss of generality that $\z = \mathbf{0}$. Then
we have
$$
B_{ {N}}(\mathbf{0}) =
\bigcup_{1\leq i\leq R}
\bigcup_{ \substack{  a^{(i)}_{M}, \ldots, a^{(i)}_{N-1} \in \FF_q  } }
B_{ {M}} \left( \sum_{j = M}^{N-1} a^{(1)}_j t^j, \ldots, \sum_{j = M}^{N-1} a^{(R)}_j t^j  \right),
$$
and the number of boxes used in this cover
is precisely $q^{R(N-M)}.$

\subsection{An auxiliary estimate}
\label{aux}

Let $f_1,\dots,f_R\in \FF_q[x_1,\dots,x_n]$ be forms of degree $d$ defined over a finite field $\FF_q$.
We will  assume that $\chara(\FF_q)>d$ throughout this section. Consider the affine variety
\begin{equation}\label{eq:Jak}
V=\left\{\x\in \AA^n: \rank \left(\frac{\partial f_k}{\partial x_i} \right)_{\substack{1\leq k\leq R\\1\leq i\leq n }}<R\right\}.
\end{equation}
Note that $V$ cuts out  the {\em Birch singular locus}  occuring in the work of Birch \cite{birch}.
As shown by Lee \cite{lee}, the function field version of Birch's work allows one to count $\FF_q(t)$-points of bounded height in  the system $f_1=\dots =f_R=0$, provided that
$
n-\dim V>R(R+1)2^{d-1}(d-1).
$

\begin{remark}\label{rem:sigma}
Let $V$ be given by \eqref{eq:Jak} and
let  $\sigma=\dim V$. If we assume that
the system $f_1,\dots,f_R$ cuts out a smooth complete intersection $X$ in $\PP^{n-1}$, then
we claim that $\sigma\leq R-1$.
This follows from the argument in
\cite[Lemma~3.1]{BHB} for  {\em optimal} systems of forms, this being an argument that is valid over fields of arbitrary characteristic. In characteristic $0$, it is shown in
\cite[Lemma~3.1]{BDHB} that $X$ can be defined by an equivalent optimal system
of forms, but use is made of Bertini's theorem, which is not generally available in positive characteristic.  Fortunately, the same conclusion can be arrived at in positive
characteristic, as explained by Glas in his proof of  \cite[Lemma~3.1]{glas}.
\end{remark}

Let
\begin{equation}\label{eq:Vf}
V_{\h.\f} = \{ \x \in \AA^n: h_1\nabla f_1 (\x)+\dots +h_R\nabla f_R (\x) = \mathbf{0} \},
\end{equation}
for any $\h\in \FF_q[t]^n$.
Then it will be convenient to  define
\begin{equation}\label{1.10}
\sigma_{\f} = \max_{\substack{\h \in \FF_q[t]^R\\
\h\neq \0} } \dim V_{\h. \mathbf{f}}.
\end{equation}

\begin{remark}\label{rem:sigma-f}
Note that $V_{\h. \mathbf{f}}$ is contained in the variety $V$ defined in \eqref{eq:Jak}. Hence, if the system $f_1,\dots,f_R$ cuts out a smooth complete intersection in $\PP^{n-1}$, then
Remark~\ref{rem:sigma} implies that $\sigma_\f\leq \sigma\leq R-1$.
\end{remark}

Suppose that
$$
f_k(\x)
=\sum_{i_1,\dots,i_d=1}^n c_{i_1,\dots,i_d}^{(k)}
x_{i_1}\dots x_{i_d},
$$
for $1\leq k\leq R$,
with coefficients $c_{i_1,\dots,i_d}^{(k)} \in \FF_q$ which are symmetric  in the indices.
Associated to each $f_k$ are the multilinear forms
\begin{equation}\label{eq:multi}
\Psi_i^{(k)}(\x^{(1)},\dots,\x^{({d-1})})
=
d!
\sum_{i_1,\dots,i_{d-1}=1}^n c_{i_1,\dots,i_{d-1},i}^{(k)}
x_{i_1}^{(1)}\dots x_{i_{d-1}}^{(d-1)},
\end{equation}
for $1\leq i\leq n$.
Given $J\in \NN$, $\bbe=(\beta^{(1)},\dots,\beta^{(R)})\in K_\infty^R$,
we shall be  interested in the size of the counting function
\begin{equation}\label{eq:rabbit}
N^{\textnormal{aux}}(J;  \bbe) =\#\left\{\underline\uu\in \FF_q[t]^{(d-1)n}:
\begin{array}{l}
|\uu^{(1)}|,\dots,|\uu^{(d-1)}|< q^J\\
\left|
\sum_{k=1}^{R} \beta^{(k)} \Psi_i^{(k)}(\underline\uu)
\right|<  q^{(d-2) J}\\
\text{for  $1 \leq i \leq n$}
\end{array}
\right\},
\end{equation}
where
$\underline\uu=(\uu^{(1)},\dots,\uu^{(d-1)})$.
(This is a function field variant of the quantity considered in \cite[Def.~1.1]{simon},
in which  the upper bound $|\bbe| q^{(d-2)J}$ is replaced by
$q^{(d-2)J}$.)
The trivial bound for this quantity is
$$
N^{\textnormal{aux}}(J;  \bbe)\leq q^{Jn(d-1)}.
$$
Our main result in this section offers the following improvement for suitable  $\bbe\in K_\infty^R$.

\begin{theorem}\label{t:Nbf}
Let $J\in \NN$, $M\in \ZZ$ and let   $\sigma=\dim V$, where  $V$ is the variety  cut out by
\eqref{eq:Jak}. Let $\bbe\in K_\infty^R$ such  that $|\bbe|= q^M$. Assume that
$M\geq d-2$.
Then
$$
N^{\textnormal{aux}}(J;  \bbe) \leq
(d-1)^{\min\{\frac{J+M}{d-1}, J\}n+n}  q^{Jn(d-1) -
\min\{\frac{J+M}{d-1}, J\}
(n-\sigma)}.
$$
\end{theorem}

\begin{proof}
Let $M^*\geq 0$ and
define
$$
N^{\textnormal{aux}}(J,M^*;  \bbe) =\#\left\{\underline\uu\in \FF_q[t]^{(d-1)n}:
\begin{array}{l}
|\uu^{(1)}|,\dots,|\uu^{(d-1)}|< q^J\\
\left|
\sum_{k=1}^{R} \beta^{(k)} \Psi_i^{(k)}(\underline\uu)
\right|<  q^{(d-2) J-M^*}\\
\text{for  $1 \leq i \leq n$}
\end{array}
\right\},
$$
for any $\bbe\in K_\infty^R$ such that $|\bbe|=1$.
Assuming that $M^*$ lies in the range
\begin{equation}\label{eq:M*-range}
d-2\leq M^*\leq
J(d-2),
\end{equation}
we shall prove that
\begin{equation}\label{eq:mu2}
N^{\textnormal{aux}}(J,M^*;  \bbe)\leq
(d-1)^{\frac{(J+M^*)n}{d-1}+n}  q^{Jn(d-1) - \frac{(J+M^*)(n-\sigma)}{d-1}}.
\end{equation}
 This will suffice for the theorem, since
$$
N^{\textnormal{aux}}(J;  \bbe)\leq  N^{\textnormal{aux}}(J ,
\min\{M,
J(d-2)\}
;  \bbe t^{-M}),
$$
on recalling that  $|\bbe|=q^M$.

Assume henceforth that  $|\bbe|=1$ and
let
$$
\beta^{(k)}=\sum_{r\geq 0} b_r^{(k)}t^{-r},
$$
for $1\leq k\leq R$ and $b_r^{(k)}\in \FF_q$. We may assume without loss of generality that $b_0^{(1)}\neq 0$, since $|\bbe|=1$.
Since $|\uu^{(j)}|<  q^J$, for $1\leq j\leq d-1$, each component of $\uu^{(j)}$ can be written
$$
u_{i}^{(j)}=\sum_{\ell=0}^{J-1} z_{i,\ell}^{(j)} t^\ell,
$$
for $1\leq i\leq n$ and $z_{i,\ell}^{(j)}\in \FF_q$. Expanding everything out, it therefore follows that
$\sum_{k=1}^{R} \beta^{(k)}\Psi_i^{(k)}(\underline\uu)$ is equal to
$$
d! \sum_{k=1}^R
\sum_{r\geq 0} b_r^{(k)} t^{-r}
\sum_{i_1,\dots, i_{d-1}=1}^n c_{i_1,\dots,i_{d-1},i}^{(k)} \sum_{\ell_1,\dots,\ell_{d-1}=0}^{J-1}
z_{i_1,\ell_1}^{(1)}\dots
z_{i_{d-1},\ell_{d-1}}^{(d-1)}
 t^{\ell_1+\dots+\ell_{d-1}},
$$
for $1\leq i\leq n$. The components  $z_{i,\ell}^{(j)}$ are therefore counted by $N^{\textnormal{aux}}(J,M^*;  \bbe)$ if and only
if
the coefficient of $t^{(d-2)J-M^* + v}$ vanishes in this expression for all
$$
0\leq v \leq J+M^*-(d-1).
$$
Note that this interval is non-empty, since $J\geq 1$ and $M^*\geq d-2$, by the lower bound
in \eqref{eq:M*-range}, whence $J+M^*\geq d-1$.

By collecting together  the coefficients of $t^{(d-2)J  -M^*+ v}$,
for $0\leq v \leq J+M^*-(d-1)$,
we see  that  $N^{\textnormal{aux}}(J;  \bbe)$ is equal to the number of elements
$z_{i,\ell}^{(j)}\in \FF_q$, for $1\leq i\leq n$, $0\leq \ell\leq J-1$ and $1\leq j\leq d-1$, such that
$$
\sum_{k=1}^R
\sum_{i_1,\dots, i_{d-1}=1}^n
\hspace{-0.4cm}
c_{i_1,\dots,i_{d-1},i}^{(k)}
\hspace{-1cm}
\sum_{\substack{\ell_1,\dots,\ell_{d-1}=0\\
\ell_1+\dots+\ell_{d-1}\geq v +(d-2)J-M^*
}}^{J-1}
\hspace{-1cm}
b_{\ell_1+\dots+ \ell_{d-1}-(d-2)J-v+M^*}^{(k)}
z_{i_1,\ell_1}^{(1)}\dots
z_{i_{d-1},\ell_{d-1}}^{(d-1)}=0,
$$
for $1\leq i\leq n$ and
$0\leq v \leq J+M^*-(d-1).$
In total  we have $Jn(d-1)$ variables and $n(J+M^*-d+2)$ equations. We let $\widetilde{Z} \subset \AA^{J(d-1)n}$ be the algebraic variety cut out by this system.
For each choice of $v$, we let $X_v \subset \AA^{J(d-1)n}$ denote the algebraic variety defined by the system of $n$ equations associated to it.

Let $g=\lfloor (J+M^*)/(d-1)\rfloor\geq 1$.
We are unable to extract anything useful from many of the equations, but we will be able to exploit information when $v = J+M^* - h(d-1)$ for integers $1\leq h\leq g$, and also when
$v = J+M^* - g(d-1)-1$, if
$(J+M^*)/(d-1)\not\in \ZZ$.
For this purpose we define
$$
Z =X_{J+M^*-(d-1)} \cap  X_{J+M^*-2(d-1)}\cap \dots \cap X_{J+M^*-g(d-1)}\cap X_{J+M^*-g(d-1)-\epsilon_J},
$$
where
$$
\epsilon_J=
\begin{cases}
0 & \text{ if $d-1\mid J+M^*$,}\\
1 & \text{ otherwise.}
\end{cases}
$$
In order to establish \eqref{eq:mu2}, it will suffice to prove that 
\begin{equation}\label{eq:dimZ}
\dim Z \leq
Jn(d-1)   -(g+\epsilon_J)(n-\sigma).
\end{equation}
Note that our assumption \eqref{eq:M*-range} on $M^*$ implies that
$$
(g+\epsilon_J)(n-\sigma)\leq \left(\frac{J+M^*}{d-1}+1-\frac{1}{d-1}\right)(n-\sigma)
\leq Jn(d-1),
$$
so that the right hand side of \eqref{eq:dimZ} is indeed non-negative.

It follows from B\'ezout's theorem, in the form \cite[Example~8.4.7]{fulton}, that the degree of $Z$ satisfies
$
\deg Z \leq (d-1)^{(g+\epsilon_J)n}.
$
Noting that $\widetilde{Z} \subset Z$, \eqref{eq:mu2}
follows from
an application of Lemma \ref{LWbdd} and the observation that
$$
 \frac{J+M^*}{d-1}\leq
g+\epsilon_J\leq \frac{J+M^*}{d-1}+1.
$$
It therefore  remains to prove \eqref{eq:dimZ}.

\subsection*{The case $g=1$.}
In the definition of  $X_{J+M^*-(d-1)}$, only
the indices $\ell_1=\dots=\ell_{d-1}=J-1$ are possible. It follows that
$X_{J+M^*-(d-1)}$
is given by the system of  equations
$$
G_i(\z_{J-1}^{(1)}, \dots, \z_{J-1}^{(d-1)})=0,
$$
for $1\leq i\leq n$, where
\begin{equation}\label{eq:Gi}
G_i(\z_{J-1}^{(1)}, \dots, \z_{J-1}^{(d-1)})=
\sum_{k=1}^R
b_{0}^{(k)}
\sum_{i_1,\dots, i_{d-1}=1}^n c_{i_1,\dots,i_{d-1},i}^{(k)}
z_{i_1,J-1}^{(1)}\dots
z_{i_{d-1},J-1}^{(d-1)}.
\end{equation}
If $\ve_J=1$ then,
in the definition of
$X_{J+M^*-(d-1)-1}$, we see that either
$\ell_1=\dots=\ell_{d-1}=J-1$, or else
 precisely one of $\ell_1,\dots, \ell_{d-1}$ is equal to $J-2$, with all the others equal to $J-1$.
It follows that
$X_{J+M^*-(d-1)-1}$
is given by the system of  equations
$$
A_i(\z_{J-1}^{(1)}, \dots, \z_{J-1}^{(d-1)})+
B_i(\z_{J-2}^{(1)}, \dots, \z_{J-2}^{(d-1)}, \z_{J-1}^{(1)}, \dots, \z_{J-1}^{(d-1)})
=0,
$$
for $1\leq i\leq n$, where
\begin{equation}\label{eq:Hi}
A_i(\z_{J-1}^{(1)}, \dots, \z_{J-1}^{(d-1)})=
\sum_{k=1}^R
b_{1}^{(k)}
\sum_{i_1,\dots, i_{d-1}=1}^n c_{i_1,\dots,i_{d-1},i}^{(k)}
z_{i_1,J-1}^{(1)}\dots
z_{i_{d-1},J-1}^{(d-1)}
\end{equation}
and $B_i(\z_{J-2}^{(1)}, \dots, \z_{J-2}^{(d-1)}, \z_{J-1}^{(1)}, \dots, \z_{J-1}^{(d-1)})$ is given by
\begin{equation}\label{eq:Ki}
\sum_{m=1}^{d-1}
\sum_{k=1}^R
\sum_{i_1,\dots, i_{d-1}=1}^n
\hspace{-0.3cm}
c_{i_1,\dots,i_{d-1},i}^{(k)}
b_{0}^{(k)}
z_{i_1,J-1}^{(1)}\dots
z_{i_{m-1},J-1}^{(m-1)}
z_{i_m,J-2}^{(m)}
z_{i_{m+1},J-1}^{(m+1)}
\dots
z_{i_{d-1},J-1}^{(d-1)}.
\end{equation}

Let  $Y_{J-1} \subset \AA^{n(d-1)}$ be the variety that is obtained by projecting
$X_{J+M^*-(d-1)}$ to the $(\z_{J-1}^{(1)}, \dots, \z_{J-1}^{(d-1)})$ variables,
being defined by
$
G_i(\z_{J-1}^{(1)}, \dots, \z_{J-1}^{(d-1)})=0
$, for $1\leq i\leq n$.
We shall prove \eqref{eq:dimZ}, for $g=1$,  by projecting to $Y_{J-1}$, which we denote by $\pi_0: Z \to Y_{J-1}$, and considering the dimension of the fibres.
The particular version of the fibre dimension theorem we use is the Corollary to Theorem 2 in Mumford \cite[\S I.8]{mum}.
Thus
\begin{equation}\label{eq:dimZ-1}
\dim Z \leq \dim  \overline{\pi_0(Z)} + \Delta_0  \leq \dim Y_{J-1} +\Delta_0,
\end{equation}
where $\Delta_0$ is the dimension of the fibre above any chosen point in the image of $\pi_0$.
We claim that
\begin{equation}\label{eq:dimZ-2}
\dim Y_{J-1} \leq n(d-2)+\sigma.
\end{equation}
To see this, we follow Birch \cite[Lemma~3.3]{birch} and proceed by intersecting the system of equations defining $Y_{J-1}$ with the diagonal
$$
D=\{ (\z_{J-1}^{(1)},  \dots,  \z_{J-1}^{(d-1)}) \in \AA^{n(d-1)} : \z_{J-1}^{(1)}= \dots = \z_{J-1}^{(d-1)}\}.
$$
This produces an algebraic variety cut out by the system
$$
\sum_{k=1}^R
b_{0}^{(k)}  \frac{\partial f_k (\z)}{\partial z_i}=0,
$$
for $1\leq i\leq n$. Since $b_0^{(1)}\neq 0$ these relations imply that $\z$ must satisfy \eqref{eq:Jak}. Hence
the affine dimension theorem implies that
$$
\sigma=\dim V\geq  \dim (Y_{J-1} \cap D) \geq \dim Y_{J-1} +\dim D-n(d-1).
$$
The diagonal has dimension $n$ and so it follows that \eqref{eq:dimZ-2} holds, as claimed.

It remains to bound  $\Delta_0$.
Suppose first that $\ve_J=0$. Then we define the variety
$$
T_{J-1}= \{ (\z_{0}^{(1)},\dots,
\z_{J-1}^{(d-1)})  \in \AA^{Jn(d-1)}   :   \z_{J-1}^{(1)}=\dots = \z_{J-1}^{(d-1)} = \mathbf{0} \}.
$$
It follows that
$$
\Delta_0 \leq \dim (Z \cap T_{J-1})\leq  \dim T_{J-1}\leq (J-1)n(d-1).
$$
In this case, the desired bound \eqref{eq:dimZ} follows immediately
from
\eqref{eq:dimZ-1} and \eqref{eq:dimZ-2}.

Suppose next that $\ve_J=1$.
Fix a  point
$\w\in \AA^n$,
to be determined in due course.
Then we define the variety
$$
S_{J-1}(\w)= \{ (\z_{0}^{(1)},\dots,
\z_{J-1}^{(d-1)})  \in \AA^{Jn(d-1)}   :   \z_{J-1}^{(1)}=\dots = \z_{J-1}^{(d-2)} =\w, ~\z_{J-1}^{(d-1)}= \mathbf{0} \}.
$$
Note that
$G_i(\w,\dots,\w,\0)=
A_i(\w,\dots,\w,\0)=0$,  for $1\leq i\leq n$. It follows that
$$
\Delta_0 \leq \dim (Z \cap S_{J-1}(\w))=
\dim (X_{J+M^*-(d-1)-1} \cap S_{J-1}(\w)),
$$
and from  \eqref{eq:Hi} and \eqref{eq:Ki} that
$X_{J+M^*-(d-1)-1} \cap S_{J-1}(\w)$ is cut out by the system of equations
$$
\sum_{k=1}^R
\sum_{i_1,\dots, i_{d-1}=1}^n
c_{i_1,\dots,i_{d-1},i}^{(k)}
b_{0}^{(k)}
w_{i_1}\dots
w_{i_{d-2}}
z_{i_{d-1},J-2}^{(d-1)}=0,
$$
for $1\leq i \leq n$. But this is just the system of equations
$$
H_{\mathbf{b}.\f}(\w)
  \z_{J-2}^{(d-1)}=\mathbf 0,
$$
where
$ \mathbf{b}.\f=b_0^{(1)}f_1+\dots +
b_0^{(R)}f_R$ and
$H_f$ is the Hessian matrix of second derivatives associated to a polynomial $f\in \FF_q[x_1,\dots,x_n]$.
Suppose that $H_{\mathbf{b}.\f}(\w)$ has rank $\rho$.
Then, on returning to
 $X_{J+M^*-(d-1)-1} \cap S_{J-1}(\w)$, we see that there are no constraints on the vectors
$\z_{0}^{(1)},\dots, \z_{J-2}^{(d-2)}$, but $\z_{J-2}^{(d-1)}$ is constrained to lie in a linear space of dimension $n-\rho$.
Thus
$$
\Delta_0\leq (J-1)n(d-1)-\rho.
$$
We now want to choose $\w$ to make $\rho$ as large as possible.
Let $T_\rho$ be the variety of $\w\in \AA^n$ such that
$\rank H_{\mathbf{b}.\f}(\w)\leq \rho$. Then it follows from
precisely the same argument in
\cite[Lemma 2]{BHB-4}, together with  Remark~\ref{rem:sigma-f},
that
$$
\dim T_\rho\leq \rho+\sigma,
$$
since  $b_0^{(1)}\neq 0$. 
Let  $\rho$ be the least non-negative integer such that $\dim T_{\rho}=n$. We then choose 
$\w\in T_{\rho}\setminus T_{\rho-1}$, and therefore deduce that
$$
\Delta_0\leq (J-1)n(d-1)-(n-\sigma).
$$
The desired bound  \eqref{eq:dimZ} now follows
from
\eqref{eq:dimZ-1} and \eqref{eq:dimZ-2}.

\subsection*{The case $g\geq 2$}

It will be convenient to define the variety
$$
T_{i}= \{ (\z_{0}^{(1)},\dots,
\z_{J-1}^{(d-1)})  \in \AA^{Jn(d-1)}   :   \z_{i}^{(1)}=\dots = \z_{i}^{(d-1)} = \mathbf{0} \},
$$
for each $i\in \{0,\dots,J-1\}$.
Let
 $Y_{J-2}\subset \AA^{n(d-1)}$ be the variety obtained by projecting
$X_{J+M^* -2(d-1)}\cap T_{J-1}$ to the
$(\z_{J-2}^{(1)}, \dots, \z_{J-2}^{(d-1)})$ variables.
Then we may  define the projection
$$
\pi_1: (Z \cap T_{J-1}) \to Y_{J-2}.
$$
Note that
$X_{J+M^* -2(d-1)}\cap T_{J-1}$
is defined by the system of  equations
$$
 \sum_{k=1}^R
\sum_{i_1,\dots, i_{d-1}=1}^n
\hspace{-0.2cm}
c_{i_1,\dots,i_{d-1},i}^{(k)}
\hspace{-0.8cm}
\sum_{\substack{\ell_1,\dots,\ell_{d-1}=0\\
\ell_1+\dots+\ell_{d-1}\geq (J-2)(d-1)
}}^{J - 2}
\hspace{-0.8cm}
b_{\ell_1+\dots+ \ell_{d-1}-(J-2)(d-1)}^{(k)}
z_{i_1,\ell_1}^{(1)}\dots
z_{i_{d-1},\ell_{d-1}}^{(d-1)}=0,
$$
for $1\leq i\leq n$. The summation conditions
  force $\ell_1=\dots=\ell_{d-1}= J-2$ in this expression,  so that we have the system
$G_{i}(\z_{J-2}^{(1)},\dots,\z_{J-2}^{(d-1)})=0$,
for $1\leq i\leq n$, in the notation of \eqref{eq:Gi}.
Hence
$Y_{J-2}$ is defined by exactly the  same set of equations that was used to define  the variety $Y_{J-1}$, from which it follows that  $\dim Y_{J-2}=\dim Y_{J-1}$.
By the same fibre dimension theorem as above, we obtain
$$
\dim (Z \cap T_{J-1}) \leq \dim Y_{J-1} + \Delta_1,
$$
where $\Delta_1$ is the dimension of the fibre above any
point in  $\pi_1 (Z \cap T_{J-1})$. Taking the point
$$
(\z_{J-2}^{(1)}, \dots, \z_{J-2}^{(d-1)}) = (\mathbf{0}, \dots, \mathbf{0}) \in \pi_1 (Z \cap T_{J-1}),
$$
we see that
$
\Delta_1 \leq  \dim  (Z  \cap T_{J-2} \cap T_{J-1}).
$
Continuing inductively in this way, we obtain the  bound
\begin{align*}
\dim Z
&\leq \dim Y_{J-1} + \dim (T_{J-1} \cap Z)\\
&\leq 2 \dim Y_{J-1} +\dim (T_{J-2} \cap T_{J-1} \cap Z)\\
&\hspace{0.2cm}\vdots\\
&\leq (g-1) \dim Y_{J-1} + \dim ( T_{{J} - (g-1) } \cap \dots \cap T_{J-1}\cap Z).
\end{align*}

Since $M^*\leq J(d-2)$, it follows that
 $J-g\geq 0$.
If $\epsilon_J=0$ then we repeat this once more to deduce that
\begin{align*}
\dim Z
&\leq
 g \dim Y_{J-1}+ \dim ( T_{J - g } \cap \dots \cap T_{J-1} \cap \AA^{nJ(d-1)}) \\
&= g \dim Y_{J-1} +(J - g )n(d-1).
\end{align*}
But then
\begin{align*}
\dim Z
&\leq (J-g ) n(d-1)+ g \left(n(d-2)+\sigma\right)= Jn(d-1)   -g(n-\sigma),
\end{align*}
by   \eqref{eq:dimZ-2}, which    is satisfactory for \eqref{eq:dimZ}.

On the other hand, if $\epsilon_J=1$ then we
deduce that
$$
\dim ( T_{{J} - (g-1) } \cap \dots \cap T_{J-1}\cap Z)\leq
\dim Y_{J-1}+ \Delta_g,
$$
where $\Delta_g$ is the
 dimension of the fibre above any
point in the image of the  projection
$$
\pi_g: Z \cap T_{J-1}\cap \dots\cap T_{J-(g-1)}\to Y_{J-g},
$$
where $Y_{J-g}\subset \AA^{n(d-1)}$ is  obtained by projecting
$X_{J+M^*-g(d-1)}\cap T_{J-1}\cap \cdots \cap T_{J-(g-1)}$ to the $(\z_{J-g}^{(1)},\dots,\z_{J-g}^{(d-1)})$ variables.
We argue as in the case $g=1$, by picking a point $\w\in \AA^n$ similarly to before and deducing that
\begin{align*}
\Delta_g&\leq 
\dim\left( Z\cap T_{J-1}\cap \cdots \cap T_{J-(g-1)}\cap S_{J-g}(\w)\right)\\
&\leq \dim \left(X_{J+M^*-g(d-1)-1}\cap T_{J-1}\cap \cdots \cap T_{J-(g-1)} \cap S_{J-g}(\w)\right),
\end{align*}
where
$$
S_{J-g}(\w)=\left\{
(\z_0^{(1)},\dots, \z_{J-1}^{(d-1)}) \in \AA^{Jn(d-1)}: 
\z_{J-g}^{(1)}=\cdots =\z_{J-g}^{(d-2)}=\w, ~\z_{J-g}^{(d-1)}=\0
\right\}.
$$
Thus  $\Delta_g\leq (J-g)n(d-1)-(n-\sigma)$, as previously, from  which 
it follows that
$$
\dim Z\leq g \dim Y_{J-1} +(J-g)n(d-1)-(n-\sigma)\leq Jn(d-1)-(g+1)(n-\sigma),
$$
by   \eqref{eq:dimZ-2}.
This too is satisfactory for \eqref{eq:dimZ}, which thereby  completes
the proof of
 Theorem
\ref{t:Nbf}.
\end{proof}

\begin{corollary}\label{c:Nbf}
Assume that
the system $f_1,\dots,f_R$ cuts out a smooth complete intersection in $\PP^{n-1}$.
Let $J\in \NN$ and let  $\bbe\in K_\infty^R$ such that
$|\bbe|\geq q^{d-1}$.
Then
$$
N^{\textnormal{aux}}(J;  \bbe) \leq (d-1)^n  q^{Jn(d-1)}\left(q^{n\left(1-\log_q(d-1)\right)-R+1}\right)^{-
\theta_d(J)},
$$
where
$$
\theta_d(J)=
\begin{cases}
\frac{J}{d-1}+1 &\text{ if $d\geq 3$ and $J\geq 2$,}\\
1 &\text{ if $d\geq 3$ and $J=1$,}\\
J &\text{ if $d=2$.}
\end{cases}
$$
\end{corollary}

\begin{proof}
Recall from Remark \ref{rem:sigma} that $\sigma\leq R-1$.
Suppose that $|\bbe|=q^M$, where $M\geq d-1$.
Theorem \ref{t:Nbf} then implies that
$$
N^{\textnormal{aux}}(J;  \bbe) \leq
(d-1)^n q^{Jn(d-1)} \left(
\frac{(d-1)^n}{q^{n-R+1}}\right)^{\min\{\frac{J+M}{d-1},
J\}}.
$$
 The statement follows on  noting that $d-1=q^{\log_q(d-1)}$ and
 \begin{align*}
\min\left\{\frac{J+M}{d-1}, J\right\}
&\geq
\min\left\{\frac{J}{d-1}+1,J\right\}=\theta_d(J),
\end{align*}
under the assumption $M\geq d-1$.
\end{proof}

\subsection{A further technical estimate}

Now that Theorem \ref{t:Nbf} is established, our final result will also prove useful.
Let $\bbe=(\beta^{(1)},\dots,\beta^{(R)})\in K_\infty^R$.
For each $0 \leq v \leq d-1$, let
\begin{equation}
N^{(v)}(J; \bbe)
\label{def Nv}
=
\# \left\{
\underline{\uu} \in \FF_q[t]^{(d-1)n}:
\begin{array}{l}
 |\uu^{(1)}|,\dots,|\uu^{(v)}|<  q^J \\
  |\uu^{(v + 1)}|,\dots,|\uu^{(d-1)}|<  q^{P} \\
\left \| \sum_{k = 1}^R \beta^{(k)} \Psi^{(k)}_{i}(\underline{\uu}) \right\| < q^{- (v + 1)P + v J} \\
\text{for all $1 \leq i \leq n$}
\end{array}
\right\},
\end{equation}
where $\underline{\uu}=(\uu^{(1)},\dots,\uu^{(d-1)})$.

\begin{lemma}
\label{last1}
Let $J,P \in \NN$ such that $J\leq P$ and let  $\bbe \in K_{\infty}^R$.
Then precisely one of the following two alternatives must happen:
\begin{enumerate}
\item[(i)] we have
$$
q^{- J} \leq    \max\left\{ q^{-d P + d - 2}  |\bbe|^{-1},   | \bbe |^{\frac{1}{d-1}}
q^{-1}
\right\};
$$
\item[(ii)] we have
$$q^{J - dP + d - 1}   \leq  |\bbe|  \leq q^{- J(d-1)+d-2}$$
and
$
N^{(d-1)}(J;  \bbe) =
N^{\textnormal{aux}}( J ;  \bbe t^{dP-J} ),
$
in the notation of \eqref{eq:rabbit}.
\end{enumerate}
\end{lemma}

\begin{proof}
Suppose that alternative (i) fails, so that $|\bbe|$ satisfies the inequalities in alternative (ii).
Then
$$
\left| \sum_{k = 1}^R \beta^{(k)} \Psi^{(k)}_{i}(\underline{\uu}) \right| \leq | \bbe | q^{ (d-1)(J-1) }
\leq q^{d-2}\cdot q^{- (d-1)} < 1,
$$
for any $|\uu^{(1)}|,\dots,|\uu^{(d-1)}|<  q^J$.
But then the statement clearly follows.
\end{proof}

\section{Singular series and singular integral}

Let $f_1,\dots,f_R\in \FF_q[x_1,\dots,x_n]$ be forms of degree $d$ defined over a finite field $\FF_q$. We  will  assume that the system of equations
$f_1=\dots =f_R=0$ defines a smooth  complete intersection in $\PP^{n-1}$.
Let
$m\in \FF_q[t]$ be a monic polynomial and let $\b\in \FF_q[t]^n$ be such that $\gcd(\b,m)=1$
and $f_i(\b)\equiv 0 \bmod{m}$, for $1\leq i\leq R$.
We define
\begin{equation}\label{eq:defF}
\mathbf{F}(\x) = \mathbf{f}(m \x + \b),
\end{equation}
where we write $\mathbf{f}=(f_1,\dots,f_R)$ for the vector of polynomials, and similarly for $\mathbf F$.  In Section \ref{s:apply} we shall initiate the application of the circle method. The purpose of this section is to discuss, in isolation, the  singular series and singular integral that will arise in our analysis.
In traditional applications of the circle method it is customary to recycle estimates from the minor arcs in order to establish the convergence of the singular series and the singular integral.
In the setting of smooth complete intersections defined over the constant field $\FF_q$, however, a direct  treatment is more efficient.

\subsection{Singular series}

Given a non-zero polynomial $g \in \FF_q[t]$ and $\a \in \FF_q[t]^R$,
we define
\begin{equation}\label{eq:D1}
S_g(\a) = |g|^{-n} \sum_{|\y| < |g|  }
\psi \left(  \frac{\a.\mathbf{F}(\y)}{g} \right),
\end{equation}
where $\mathbf{F}$ is given in \eqref{eq:defF} and
 $\psi$ is defined in Section \ref{s:add-characters}.
In this section we study  the {\em singular series}
\begin{eqnarray}
\label{def ss}
\mathfrak{S} = \sum_{ \substack{ g \in \FF_q[t] \setminus \{ 0 \}  \\
\text{$g$ monic}}}
\sum_{ \substack{ \a\in \FF_q[t]^R\\
|\a| < |gm^d| \\ \gcd(\a, g) = 1 } } S_{gm^d}(\a),
\end{eqnarray}
with the aim of establishing its convergence under relatively mild hypotheses.
The singular  series is slightly awkward to work with, since the inner sum
sum over $\a$ runs modulo $gm^d$, but only over those $\a$ which are coprime to $g$.

Define
\begin{equation}\label{eq:def A}
A(g)=\sum_{ \substack{ \a\in \FF_q[t]^R\\
|\a| < |g| \\ \gcd(\a, g) = 1 } } S_{g}(\a),
\end{equation}
for any monic $g\in \FF_q[t]$. This is exactly the sum $\mathcal{A}(g)$ that appears in
the work of Lee  \cite[Eq.~(4.6.24)]{lee}, and
the multiplicativity of $A(g)$ follows from
\cite[Lemma~4.7.2]{lee}.
Moreover, it is easy to check that
\begin{equation}\label{eq:IML}
\sum_{ \substack{ \a\in \FF_q[t]^R\\
|\a| < |gm^d| \\ \gcd(\a, g) = 1 } } S_{gm^d}(\a)=
\sum_{\substack{h\mid m^d\\ \gcd(h,g)=1}} A(g m^d/h).
\end{equation}
For now, we proceed to collect together some facts  about $A(g)$ for an
arbitrary monic $g\in \FF_q[t]$.

\begin{prop}
\label{pro:sga}
Let $d\geq 2$ and
assume that $n\geq dR$. Suppose  that $\gcd(\b,m)=1$ and
$f_i(\b)\equiv 0 \bmod{m}$, for $1\leq i\leq R$.
Then
there exists a constant $C>0$ that
depends only on $d$ and $n$, such that
$$
|A(g)|\leq C^{\omega(g)}|g|^{-(1-\frac{1}{d})(\frac{n}{d}-R)}|\gcd(g,m)|^{(1-\frac{1}{d})\frac{n}{d}+\frac{R}{d}}.
$$
\end{prop}

By multiplicativity it suffices to bound $A(g)$ when $g=\pi^e$ is a prime power, with $e\geq 1$.
To begin with,  orthogonality of characters yields
\begin{equation}\label{eq:A-step1}
\begin{split}
A(\pi^e)&=
|\pi|^{-en}
\sum_{ \substack{ \a\in \FF_q[t]^R\\
|\a| < |\pi^e| \\ \gcd(\a,\pi)=1 } }
\sum_{|\y|<|\pi^e|} \psi\left(\frac{\a.\mathbf{F}(\y)}{\pi^e}\right)\\
&=
|\pi|^{-e(n-R)} \left(N(\pi^e) -|\pi|^{n-R} N(\pi^{e-1})\right),
\end{split}\end{equation}
where
\begin{equation}\label{eq:fan}
N(\pi^e)=\#\left\{\y\in \FF_q[t]^n: |\y|<|\pi^e|,~ F_i(\y)\equiv 0\bmod{\pi^e} \text{ for $1\leq i\leq R$}\right\}.
\end{equation}
Recall that $F_i$ is defined in terms of $f_i$ via \eqref{eq:defF}, for $1\leq i\leq R$.
We shall need to differentiate according to whether or not $\pi$ divides $m$.

\begin{lemma}\label{lem:111}
Let $d\geq 2$ and
assume that $n\geq dR$.
Assume that $e\geq 1$ and
$\pi\nmid m$. Then
$$
A(\pi^e) \ll |\pi|^{-(1-\frac{1}{d})(\frac{n}{d}-R)e},
$$
where the implied constant depends only on $d$ and $n$.
\end{lemma}

\begin{proof}
Since $\pi\nmid m$ we can make a non-singular change of variables $\z=m\y+\b$ modulo $\pi^e$, finding that
$$
N(\pi^e)=\#\left\{\z\in \FF_q[t]^n: |\z|<|\pi^e|,~ f_i(\z)\equiv 0\bmod{\pi^e} \text{ for $1\leq i\leq R$}\right\}.
$$
It will be convenient to also define
$$
N^*(\pi^e)=\#\left\{\z\in \FF_q[t]^n: |\z|<|\pi^e|,~ \pi\nmid \z, ~f_i(\z)\equiv 0\bmod{\pi^e} \text{ for $1\leq i\leq R$}\right\}.
$$
In view of the fact that  $f_1=\dots=f_R=0$ defines a smooth complete intersection in $\PP^{n-1}$,
Hensel lifting yields
\begin{equation}\label{eq:Hensel}
N^*(\pi^e)=|\pi|^{n-R}N^*(\pi^{e-1}) \quad \text{ if $e\geq 2$}.
\end{equation}

Our analysis of $N(\pi^e)$ will depend on the relative size of $e$ to $d$.
Suppose first that $e=1$.
As explained in the  appendix by Katz to  \cite{hooley},
Deligne's resolution of the Riemann hypothesis now yields
\begin{equation}\label{eq:case1}
N(\pi)=|\pi|^{n-R}+O(|\pi|^{\frac{n-R+1}{2}}),
\end{equation}
whence
$$
A(\pi)=|\pi|^{-(n-R)} \left(N(\pi) -|\pi|^{n-R}\right)\ll |\pi|^{-\frac{n-R-1}{2}},
$$
by \eqref{eq:A-step1},
which is satisfactory for the lemma.

Suppose now that $2\leq e\leq d$.
Then we have
$$
N(\pi^e)=N^*(\pi^e)+
|\pi|^{n(e-1)}.
$$
On appealing to \eqref{eq:A-step1} and \eqref{eq:Hensel}, it therefore follows that
$$
A(\pi^e)=
|\pi|^{eR-n} (1-|\pi|^{-R})=O(|\pi|^{eR-n}).
$$
It is easy to check that $eR-n\leq -(1-\frac{1}{d})(\frac{n}{d}-R)e$ if $e\leq d$
and $n\geq dR$.
Thus this too is satisfactory for the lemma.

Finally we must deal with the case $e>d$.
On sorting according to the $\pi$-adic valuation of $\z$, we obtain
\begin{align*}
N(\pi^e)
&=\sum_{0\leq j\leq e}\#\left\{\z\in \FF_q[t]^n :
|\z|<|\pi^e|, ~\gcd(\z,\pi^e)=\pi^j, \f(\z)\equiv \0\bmod{\pi^e}\right\}\\
&=\sum_{0\leq j\leq e}
\hspace{-0.1cm}
\#\left\{\z'\in \FF_q[t]^n :
|\z'|<|\pi^{e-j}|, ~\gcd(\z',\pi^{e-j})=1, \pi^{dj} \f(\z')\equiv \0\bmod{\pi^e}\right\}.
\end{align*}
When $dj\geq e$ we bound the inner cardinality by $|\pi|^{(e-j)n}$.
When $dj<e$, the congruence only depends on the value of $\z'$ modulo $\pi^{e-dj}$, whence
\begin{equation}\label{eq:spoon}
N(\pi^e)=
\sum_{\substack{0\leq j< e/d
}}
|\pi|^{(d-1)jn}
N^*(\pi^{e-dj})+O\left(|\pi|^{(1-\frac{1}{d})en}\right),
\end{equation}
since
$
\sum_{e/d\leq j\leq e}|\pi|^{(e-j)n}\ll |\pi|^{(1-\frac{1}{d})en}.
$
Note that $n-R+(1-\frac{1}{d})(e-1)n\geq (1-\frac{1}{d})en$.
Returning to
\eqref{eq:A-step1}, we see that
\begin{align*}
|\pi|^{e(n-R)}
A(\pi^e)
&=
N(\pi^e) -|\pi|^{n-R} N(\pi^{e-1})\\
&=
\sum_{\substack{0\leq j< e/d}}
|\pi|^{(d-1)jn}
N^*(\pi^{e-dj})\\
&\quad -
|\pi|^{n-R}
\sum_{\substack{0\leq j< (e-1)/d}}\hspace{-0.2cm}
|\pi|^{(d-1)jn}
N^*(\pi^{e-1-dj})
+O\left(|\pi|^{\frac{n}{d}-R+(1-\frac{1}{d})en}\right).
\end{align*}
The first sum runs over $j$ such that $e-dj\geq 1$ and the second sum only involves $j$ with
$e-dj\geq 2$. It follows from \eqref{eq:Hensel} that all terms cancel apart from a possible term with $e-dj=1$, provided that $d\mid e-1$.  It now follows from \eqref{eq:case1}  that
\begin{align*}
A(\pi^e)
&\ll |\pi|^{-e(n-R)}\left(
|\pi|^{(d-1)\frac{(e-1)n}{d}+n-R}
+|\pi|^{\frac{n}{d}-R+(1-\frac{1}{d})en}\right)\\
&\ll |\pi|^{-e(n-R)}\cdot  |\pi|^{\frac{n}{d}-R+(1-\frac{1}{d})en}\\
&=  |\pi|^{-(\frac{n}{d}-R)(e-1)}.
\end{align*}
But $e-1> e(1-\frac{1}{d})$ if $e>d$, whence
$
A(\pi^e) \ll  |\pi|^{-(1-\frac{1}{d})(\frac{n}{d}-R)e}.
$
This finally completes the proof of the lemma.
 \end{proof}

\begin{lemma}\label{lem:222}
Suppose that $\mu\in \ZZ$ is such that $\pi^\mu\| m$, with  $\mu\geq 1$.
Assume that $e\geq 1$ and
$f_i(\b)\equiv 0 \bmod{\pi^\mu}$ for $1\leq i\leq R$, with  $\gcd(\b,\pi)=1$.
 Then
$$
A(\pi^e)=0 \quad \text{ if $e>\mu$}
$$
and
$$
|A(\pi^e)|\leq 2|\pi|^{eR} \quad  \text{ if $e\leq \mu$}.
$$
In particular, we have
$|A(\pi^e)| \leq 2
|\pi|^{\min\{e,\mu\}R}.$
\end{lemma}

\begin{proof}
It follows from
\eqref{eq:defF} and
\eqref{eq:A-step1} that
\begin{align*}
A(\pi^e)&=
|\pi|^{-e(n-R)} \left(N(\pi^e) -|\pi|^{n-R} N(\pi^{e-1})\right),
\end{align*}
where
$$
N(\pi^e)=\#\left\{\y\in \FF_q[t]^n: |\y|<|\pi^e|,~ \f(m\y+\b)\equiv \0\bmod{\pi^e}\right\}.
$$
If $e\leq \mu$ we simply take
$$
|N(\pi^e) -|\pi|^{n-R} N(\pi^{e-1})|\leq
 N(\pi^e)+
 |\pi|^{n-R} N(\pi^{e-1})
 \leq 2 |\pi|^{en},
 $$
 whence
$
|A(\pi^e)|\leq 2 |\pi|^{eR}$.

Suppose now that $e>\mu$ and write $m=\pi^\mu m'$, where $m'\in \FF_q[t]$ is coprime to $\pi$.
Then, for each $1\leq i\leq R$, the
congruence
$f_i(\pi^\mu m'\y+\b)\equiv 0\bmod{\pi^e}$ only depends on the value of $\y$ modulo $\pi^{e-\mu}$.
Hence
\begin{align*}
N(\pi^e)
&=|\pi|^{\mu n}\#\left\{\y\in \FF_q[t]^n: |\y|<|\pi^{e-\mu}|,~ \f(\pi^\mu m'\y+\b)\equiv \0\bmod{\pi^e} \right\}\\
&=|\pi|^{\mu n}N^{\dagger}(\pi^e),
\end{align*}
where
\begin{equation}\label{eq:late}
N^\dagger(\pi^e)=
\#\left\{\z\in \FF_q[t]^n: |\z|<|\pi^{e}|,~ \f(\z)\equiv \0\bmod{\pi^e}, ~\z\equiv \b \bmod{\pi^\mu} \right\}.
\end{equation}

We note that $e\geq 2$, since we are assuming that $\mu\geq 1$.
Making the change of variables $\z=\mathbf v+\pi^{e-1}\mathbf w$, for $\mathbf{v}$ modulo $\pi^{e-1}$ and $\mathbf w$ modulo $\pi$, we deduce that
\begin{align*}
N^\dagger(\pi^e)
&=
\hspace{-0.2cm}
\sum_{\substack{
\mathbf v\in \FF_q[t]^n\\
|\mathbf v|<|\pi^{e-1}|\\
\f(\mathbf v)\equiv \0\bmod{\pi^{e-1}}\\
\mathbf v\equiv \b \bmod{\pi^\mu}}}
\hspace{-0.3cm}
\#\left\{\mathbf w\in \FF_q[t]^n:
\begin{array}{l}
|\mathbf w|<|\pi|\\
\mathbf w.\nabla f_i(\mathbf v)\equiv -f_i(\mathbf v)/\pi^{e-1}\bmod{\pi}\\
\text{for $1\leq i\leq R$}
\end{array}
\right\}.
\end{align*}
Since  $\gcd(\b,\pi)=1$, we have  $\gcd(\mathbf v,\pi)=1$ in the outer sum.  Thus
the $R\times n$ matrix $(\nabla f_1(\mathbf v),\dots, \nabla f_R(\mathbf v))$ has full rank,
since $f_1=\dots=f_R=0$ defines a smooth complete intersection in $\PP^{n-1}$.
We conclude that the inner quantity is precisely $|\pi|^{n-R}$, whence
\begin{equation}\label{eq:seminar}
N^\dagger(\pi^e)
=|\pi|^{n-R} N^\dagger(\pi^{e-1})
.
\end{equation}
We have therefore shown that
\begin{align*}
A(\pi^e)&=
|\pi|^{-e(n-R)+\mu n} \left(N^\dagger(\pi^e) -|\pi|^{n-R} N^\dagger(\pi^{e-1})\right)=0.
\end{align*}
The statement of the lemma follows.
 \end{proof}

\begin{proof}[Proof of  Proposition \ref{pro:sga}]
We  write $g=hg'$ where $h=\gcd(g,m)$ and $g'$ is coprime to $m$.
Once combined with Lemmas \ref{lem:111} and \ref{lem:222},
it follows from
the multiplicativity of $A(g)$ that
\begin{align*}
|A(g)|
&=
\prod_{\substack{\pi^e\| g'
}} |A(\pi^e)|
\prod_{\substack{\pi^e\| h
}} |A(\pi^e)| \\
&\leq C^{\omega(g)} |g'|^{-(1-\frac{1}{d})(\frac{n}{d}-R)} |h|^R\\
&= C^{\omega(g)} |g|^{-(1-\frac{1}{d})(\frac{n}{d}-R)} |h|^{
(1-\frac{1}{d})(\frac{n}{d}-R)+R},
\end{align*}
for some constant $C>0$ that depends only on $d$ and $n$.
Now clearly
$$
\left(1-\frac{1}{d}\right)\left(\frac{n}{d}-R\right)+R=
\left(1-\frac{1}{d}\right)\frac{n}{d}+\frac{R}{d},
$$
and so the statement of
Proposition \ref{pro:sga} follows.
\end{proof}

\medskip

We may now return to the treatment of the singular series
\eqref{def ss}. We begin by proving the following result. 

\begin{lemma}\label{lem:snorlax}
We have 
$$
\sum_{ \substack{ \a\in \FF_q[t]^R\\
|\a| < |gm^d| \\ \gcd(\a, g) = 1 } } S_{gm^d}(\a)=\1_{\gcd(g,m)=1}  |m|^R
A(g).
$$
\end{lemma}

\begin{proof}
It follows from \eqref{eq:IML} that 
$$
\sum_{ \substack{ \a\in \FF_q[t]^R\\
|\a| < |gm^d| \\ \gcd(\a, g) = 1 } } S_{gm^d}(\a)=
\sum_{\substack{h\mid m^d\\ \gcd(h,g)=1}} A(g \tilde m),
$$
where $\tilde m=m^d/h$. 
Next, we factorise $\tilde m=m_0m_1$, where
$$
m_0=\prod_{\substack{\pi^\nu\| \tilde m\\ \pi\mid g}} \pi^\nu, \quad
m_1=\prod_{\substack{\pi^\nu\| \tilde m\\ \pi\nmid g}} \pi^\nu.
$$
Then we may write $g \tilde m = g_0m_1$, where
$g_0=g m_0$. In particular, $\gcd(m_0,m_1)=1$ and
$\gcd(g_0, m_1) = 1$, whence
$
A(g \tilde m)=A(g_0)A(m_1),
$
by multiplicativity.
The first part of Lemma \ref{lem:222} implies that the first  product vanishes
as soon as $g$ shares a common prime factor with $m$.
Therefore $m_0=1$ and we deduce that 
$$
\sum_{ \substack{ \a\in \FF_q[t]^R\\
|\a| < |gm^d| \\ \gcd(\a, g) = 1 } } S_{gm^d}(\a)=\1_{\gcd(g,m)=1} 
A(g)
\sum_{\substack{h\mid m^d}} A(\tilde m),
$$
since  the condition $\gcd(h,g)=1$ is implied by the condition $\gcd(g,m)=1$ when $h\mid m^d$.
 But
$$
\sum_{\substack{h\mid m^d}} A(\tilde m)=
\sum_{ \substack{
\a\in \FF_q[t]^R\\
 |\a| <|m^d| }}
S_{m^d}(\a)=
\frac{N(m^d)}{|m|^{d(n-R)}},
$$
by orthogonality of characters.  Using the Chinese remainder theorem we see that
$$
N(m^d)=\prod_{\pi^\mu\| m} N(\pi^{d\mu})=\prod_{\pi^\mu\| m} |\pi|^{\mu n}N^\dagger(\pi^{d\mu}),
$$
in the notation of \eqref{eq:late}. Appealing to \eqref{eq:seminar}, it now follows that
$$
N(m^d)
= |m|^n\prod_{\pi^\mu\| m} |\pi|^{(n-R)(d-1)\mu} N^\dagger(\pi^{\mu})= |m|^{R+d(n-R)},
$$
since $N^\dagger(m)=1$.
We deduce that
$$
\sum_{\substack{h\mid m^d}} A(\tilde m)=|m|^R,
$$
from which the lemma follows. 
\end{proof}

We define the truncated singular series
\begin{equation}\label{eq:scale}
\mathfrak{S}(B)
= \sum_{ \substack{ g \in \FF_q[t]\\
\text{$g$ monic}
 \\ 0<  |g| \leq   q^B  } }
\sum_{ \substack{
\a\in \FF_q[t]^R\\
 |\a| <|gm^d|  \\ \gcd(\a, g) = 1 } }
S_{gm^d}(\a),
\end{equation}
for any $B \in \ZZ_{\geq 0}$. It follows from Lemma \ref{lem:snorlax} that 
\begin{equation}\label{eq:jerk}
\mathfrak{S}(B)=|m|^R
\sum_{ \substack{\text{$g$ monic} \\ 0 < |g| \leq q^B\\ \gcd(g,m)=1 } }
A(g).
\end{equation}
Formally extending the sum to infinity, the singular series
\eqref{def ss} can also be written
\begin{equation}\label{eq:jerky}
\mathfrak{S} = |m|^R
\sum_{ \substack{\text{$g$ monic} \\ \gcd(g,m)=1 } }
A(g).
\end{equation}

We shall need to establish the convergence of
$\mathfrak{S}$, and show how well we can approximate it by
 $\mathfrak{S}(B)$.
For this we are led to  study the quantity
\begin{equation}\label{eq:ST}
\mathcal{S}(T)
=
\sum_{\substack{\text{$g$ monic}\\ |g|=q^T\\ \gcd(g,m)=1}}
  \left| A(g)\right|,
\end{equation}
for any  $T \in \ZZ_{\geq 0}$.

\begin{lemma}\label{lem:toad}
Let $d\geq 2$ and
assume that $n> dR$.
Let $\ve>0$.
Suppose  that $\gcd(\b,m)=1$ and
$f_i(\b)\equiv 0 \bmod{m}$, for $1\leq i\leq R$.
Then,
for any $T\in \ZZ_{\geq 0}$, we have
$$
\mathcal{S}(T)
\ll_\ve
q^{T(1+\ve)-T(1-\frac{1}{d})(\frac{n}{d}-R)}.
 $$
\end{lemma}

\begin{proof}
We shall use  Proposition \ref{pro:sga} to prove this result. First,
for any $\ve>0$,
it follows from the  divisor sum bound in function fields that
$C^{\omega(g)}=O_\ve(|g|^\ve)$, where the implied constant depends only on $\ve$ and $C$.  (Note that,
as explained in \cite[Lemma~2.3]{adelina},
 the implied constant in this estimate is independent of $q$.)
We deduce from Proposition \ref{pro:sga} that
\begin{align*}
\mathcal{S}(T)
&\ll_\ve
\sum_{\substack{g\in \FF_q[t]\\  \text{$g$ monic}\\ |g|=q^T}}
|g|^{-(1-\frac{1}{d})(\frac{n}{d}-R)+\ve} \ll_\ve
q^{T(1+\ve)-T(1-\frac{1}{d})(\frac{n}{d}-R)}.
\end{align*}
The statement of the lemma follows.
\end{proof}

The following result summarises our treatment of the singular series $\mathfrak{S}$ and its approximation by the truncated singular series $\mathfrak{S}(B)$.

\begin{lemma}
\label{lem2.5iv}
Let $d\geq 2$ and
assume that $(1-\frac{1}{d})(\frac{n}{d}-R)> 1$. Let $\ve>0$.
Suppose  that $\gcd(\b,m)=1$ and
$f_i(\b)\equiv 0 \bmod{m}$, for $1\leq i\leq R$.
Then, for any  $B\in \ZZ_{\geq 0}$, we have
\begin{align*}
|\mathfrak{S} - \mathfrak{S}(B)|
\ll_\ve~&
q^{R\deg m+(B+1)(1+\ve-(1-\frac{1}{d})(\frac{n}{d}-R))}.
\end{align*}
Moreover, $\mathfrak{S}$ is absolutely convergent and satisfies
$0<\mathfrak{S}\ll q^{R\deg m}$.
 The implied  constants in these estimates depend only on $d,n,R$ and $\ve$.
\end{lemma}

\begin{proof}
It follows from \eqref{eq:jerk} and \eqref{eq:jerky} that
\begin{align*}
|\mathfrak{S} - \mathfrak{S}(B)|
&\leq q^{R\deg m}
\sum_{k=B+1}^\infty \mathcal{S}(k),
\end{align*}
in the notation of \eqref{eq:ST}.
Since $(1-\frac{1}{d})(\frac{n}{d}-R)> 1$,  Lemma \ref{lem:toad} yields
\begin{align*}
|\mathfrak{S} - \mathfrak{S}(B)|
&\ll_\ve q^{R\deg m}
\sum_{k = B+1}^{\infty}
q^{k(1+\ve)-k(1-\frac{1}{d})(\frac{n}{d}-R)}\\
&\ll_\ve
q^{R\deg m+(B+1)(1+\ve-(1-\frac{1}{d})(\frac{n}{d}-R))}.
\end{align*}
The first part of the lemma readily follows.
Next, we may reapply
Lemma \ref{lem:toad} to deduce the upper bound
$\mathfrak{S}\ll q^{R\deg m}$.

It remains to establish the positivity of $\mathfrak{S}$ under the assumptions of the lemma, having  already established  that $\mathfrak{S}$ is absolutely convergent.
The argument in  \cite[Corollary~4.7.7]{lee} ensures that
$\mathfrak{S}>0$
provided that  there exists a smooth $\mathcal{O}_\pi$-point on
the variety $f_1=\cdots=f_R=0$,
for every finite prime $\pi\nmid m$. But this follows from the assumptions of the lemma, since
the completion of $\FF_q(t)$ at $\pi$ is a $C_1$ field.
\end{proof}

\subsection{Singular integral}

For any  $\boldsymbol{\gamma} \in K_{\infty}^R$, we define
\begin{equation}\label{eq:D2}
S_{\infty} (\bga)
=
\int_{\TT^n} \psi (\bga. \mathbf{f}(\v)) \d \v.
\end{equation}
In analogy to \eqref{eq:ST}, let
\begin{equation}\label{eq:IT}
\mathcal{I}(T)
=
\int_{|\bga|= q^{T}} S_{\infty} (\bga) \d\bga,
\end{equation}
for any  $T \in \ZZ_{\geq 0}$.
We begin by proving the following upper bound.

\begin{lemma}\label{lem:est2}
Let $d\geq 2$ and assume that $n\geq dR$.  For any  $T \in \ZZ_{\geq 0}$ we have
$$
\mathcal{I}(T)\ll \min\left\{q^{R(T+1)}, q^{(1-\frac{2}{d})n+R-(1-\frac{1}{d})(\frac{n}{d}-R)T}\right\}.
 $$
\end{lemma}

\begin{proof}
Let $S\in \ZZ$ and  define
$$
m(S)=
\meas\left\{
\v\in \TT^n:  |f_i(\v)|<q^{-S} \text{ for $1\leq i\leq R$}\right\}.
$$
We claim  that
\begin{equation}\label{eq:claim-ms}
m(S)=
\begin{cases}
1 &\text{ if $S<d$,}\\
q^{(d-1-S)n}N(t^{S+1-d}) &\text{ if $S\geq d$,}
\end{cases}
\end{equation}
where 
$$
N(\pi^e)=\#\left\{\z\in \FF_q[t]^n: |\z|<|\pi^e|,~ f_i(\z)\equiv 0\bmod{\pi^e} \text{ for $1\leq i\leq R$}\right\},
$$
for any prime $\pi\in \FF_q[t]$.

Taking \eqref{eq:claim-ms} on faith for the moment, let us see how it suffices to complete the proof of the lemma.  Taking the trivial upper bound $|S_\infty(\bga)|\leq 1$, we see that $|\mathcal{I}(T)|\leq q^{R(T+1)}$. This is less than or equal to
$q^{n-(1-\frac{1}{d})(\frac{n}{d}-R)T}$ if
$T\leq  d-1$.
Assume henceforth that $T\geq d$.
It follows from \eqref{eq:ortho}  and \eqref{eq:claim-ms} that
\begin{align*}
\mathcal{I}(T)
&= q^{R(T+1)}\left(m(T+1)-q^{-R}m(T)\right)\\
&= q^{R(T+1)+(d-2-T)n}
\left(N(t^{T+2-d})- q^{n-R}
N(t^{T+1-d})
\right).
\end{align*}
Taking $\pi=t$ and $e=T+2-d$, it now follows from \eqref{eq:A-step1} and
Lemma \ref{lem:111} that
\begin{align*}
N(t^{T+2-d})- q^{n-R}
N(t^{T+1-d})
&= q^{(T+2-d)(n-R)}A(t^{T+2-d})\\
&\ll q^{(T+2-d)(n-R)}\cdot
q^{-(1-\frac{1}{d})(\frac{n}{d}-R)(T+2-d)}.
\end{align*}
Hence
\begin{align*}
\mathcal{I}(T)
&\ll  q^{R(T+1)+(d-2-T)n}\cdot
q^{(T+2-d)(n-R)}\cdot
q^{-(1-\frac{1}{d})(\frac{n}{d}-R)(T+2-d)}\\
&=  q^{(d-1)R-(1-\frac{1}{d})(\frac{n}{d}-R)(T+2-d)}\\
&=  q^{g-(1-\frac{1}{d})(\frac{n}{d}-R)T},
\end{align*}
where
$$
g=(d-1)R+(d-2)\left(1-\frac{1}{d}\right)\left(\frac{n}{d}-R\right)\leq \left(1-\frac{2}{d}\right)n+R,
$$
on taking  $1-\frac{1}{d}\leq 1$.

We now turn to the proof of  \eqref{eq:claim-ms}.
If $S<d$ then the condition
$|f_i(\v)|<q^{-S}$ is vacuous for $1\leq i\leq R$.
In this case  $m(S)=1$. Suppose next that  $S\geq d$. Then
\begin{align*}
m(S)
&=
\int_{\substack{\w\in \TT^n\\  |\w|<q^{-S}}}
\hspace{-0.2cm}
\#\left\{(\mathbf u_1,\dots,\mathbf u_S)\in \FF_q^{nS} : |f_i(t^{-S}\mathbf u_1+\dots +t^{-1} \mathbf u_{S}+\w)|<q^{-S}\right\}\d \w\\
&=q^{-Sn}
\#\left\{(\mathbf u_1,\dots,\mathbf u_S)\in \FF_q^{nS} : |f_i(t^{-S}\mathbf u_1+\dots +t^{-1} \mathbf u_{S})|<q^{-S}\right\},
\end{align*}
since the integrand  doesn't depend on $\w$.
Let $\mathbf u=t^{-S}\mathbf u_1+\dots +t^{-1} \mathbf u_{S}$ be
one of the vectors appearing in the latter cardinality and suppose that
$$
f_i(\uu)=A_0+A_1t^{-1}+\dots+A_Nt^{-N},
$$
for some $A_0,\dots,A_N\in \FF_q$ and  $N\geq 0$. Then the condition $|f_i(\uu)|<q^{-S}$ is equivalent to demanding that $A_0=\dots=A_S=0$.
We now consider the effect of writing  $t^{-1}=\tau$ and working over $\FF_q[\tau]$.
Then $\mathbf u$ becomes $\uu'=\tau^{S}\mathbf u_1+\dots +\tau \mathbf u_{S}$
and $f_i(\uu)$ becomes
$f_i(\uu')=A_0+A_1\tau+\dots+A_N\tau^N$, with the condition
$A_0=\dots=A_S=0$ now being equivalent to $f_i(\uu')\equiv 0\bmod{\tau^{S+1}}$. Hence, on recalling that $S\geq d$,  it follows that
\begin{align*}
m(S)
&=q^{-Sn}\#\left\{
\uu\in \FF_q[\tau]^n: |\uu|<q^S, ~f_i(\tau \uu)\equiv 0 \bmod{\tau^{S+1}} \text{ for $1\leq i\leq R$}
\right\}\\
&=q^{-Sn}\#\left\{
\uu\in \FF_q[\tau]^n: |\uu|<q^S, ~f_i(\uu)\equiv 0 \bmod{\tau^{S+1-d}} \text{ for $1\leq i\leq R$}
\right\}.
\end{align*}
Note that $S>S+1-d$, since $d\geq 2$.
Breaking the $\uu$ into residue classes modulo $\tau^{S+1-d}$,
the claim \eqref{eq:claim-ms} easily follows, which thereby completes the proof of the lemma.
\end{proof}

We define the truncated  singular integral
\begin{equation}\label{def si-t}
\mathfrak{I}( B )
=
\int_{|\bga| < q^{B}} S_{\infty} (\bga) \d\bga,
\end{equation}
for any $B \in \ZZ_{\geq 0}$, together with the completed {\em singular integral}
\begin{equation}
\label{def si}
\mathfrak{I} =  \int_{ K_{\infty}^R } S_{\infty}(\bga) \d \bga.
\end{equation}
We are now ready to compare these two quantities, as summarised in the following result.

\begin{lemma}
\label{si}
Let $d\geq 2$ and assume that $n\geq dR$.
Let $B\geq 0$.
Then
$$
| \mathfrak{I}
- \mathfrak{I}(B)|
\ll
q^{n-(1-\frac{1}{d})(\frac{n}{d}-R)B}.
$$
Moreover, $\mathfrak{I}$ is absolutely convergent and satisfies
 $
0< \mathfrak{I} \ll q^n
 $.
 The implied  constants in these estimates depend only on $d, n, R$.
\end{lemma}

\begin{proof}
We clearly have $(1-\frac{2}{d})n+R\leq n$ if $n\geq dR$.
It follows from Lemma \ref{lem:est2} that
\begin{align*}
| \mathfrak{I} - \mathfrak{I}(B) |
&\leq
\sum_{j \geq  B}
|\mathcal{I}(j)|\ll
q^{n-(1-\frac{1}{d})(\frac{n}{d}-R)B},
\end{align*}
as claimed in  the first part of the lemma.
The second part follows on taking
$$
|\mathfrak{I}|
\leq
\sum_{j \geq 0}
|\mathcal{I}(j)|
$$
and reapplying Lemma \ref{lem:est2}.
Finally, the last part follows from
\cite[Lemma~4.8.3]{lee}, using the fact that
$K_\infty$ is a $C_1$ field.
\end{proof}

\section{Application of the circle method}\label{s:apply}

Throughout this section, let $\FF_q$ be a finite field of characteristic $>d$ and let
$f_1,\dots,f_R\in \FF_q[x_1,\dots,x_n]$ be forms of degree $d$, cutting out a smooth complete intersection in $\PP^{n-1}$. In particular it follows from Remark \ref{rem:sigma}
that $\sigma\leq R-1$, in the notation of \eqref{eq:Jak}.

Let $m \in \FF_q[t] \setminus \{ 0 \}$ and $\b \in \FF_q[t]^{n}$ with $\deg b_i < \deg m$ and
$\gcd(\b,m)=1$.
Our primary goal in this section is to adapt the work of Rydin
Myerson \cite{simon} to
produce an estimate for the quantity
\begin{equation}\label{eq:defn-NP}
N(\mathbf{f};P,m,\b)=
\#\left\{ \g\in \FF_q[t]^n:
\begin{array}{l}
 |\g|<q^{P+\deg m} \text{ and } \g\equiv \b \bmod{m}\\
f_1(\g)=\dots=f_R(\g)=0
\end{array}
\right\},
\end{equation}
for any $P\in \NN$. It will be crucial to have an estimate that depends explicitly on $P$ and on $q$.

\subsection{Fourier analytic interpretation}

Recall the definition \eqref{eq:defF} of
$$
\mathbf{F}(\x) = \mathbf{f}(m \x + \b),
$$
where we write $\mathbf{f}=(f_1,\dots,f_R)$ for the vector of polynomials, and similarly for $\mathbf F$.
Letting
\begin{equation}\label{eq:sum-S}
S(\bal;P) = S(\bal; P , m, \b) = \sum_{  |\x| < q^{P}  } \psi (\bal. \mathbf{F}(\x)),
\end{equation}
for $\bal\in \TT^R$ and $P\in \NN$, we may now write
\begin{equation}\label{orthog}
N(\mathbf{f};P,m,\b)= \int_{\TT^R} S (\bal; {P}) \d \bal.
\end{equation}

Let $\mathcal{C}$ be an arbitrary real parameter that satisfies
$\mathcal{C}>d R$. (We will make an acceptable choice of  $\mathcal{C}$ in  \eqref{setcalC}.)
We will need to work under the following assumption. Given $\bal \in K_\infty^R$, we put $\| \bal  \| = |\{ \bal \}|$, where 
$\{\bal\} \in \TT^R$ is the fractional part of $\bal$. 

\begin{hyp}\label{hyp:2.1}
There exists $C\geq 1$ such that
$$
\min \left\{ \frac{|S(\bal; {P})|}{q^{nP}}  ,   \frac{|S(\bal + \bbe; {P})|}{q^{nP}}   \right\}
\leq C
\max \{  q^{-dP+d-1} \|m^d \bbe \|^{-1}, \|m^d\bbe\|^{\frac{1}{d-1}}q^{-1}  \}^{\mathcal{C}},
$$
for any $\bal, \bbe \in K_{\infty}^R$ and any $P \in \NN$.
\end{hyp}

The following result is a  function field analogue of \cite[Prop.~2.1]{simon}, with the added feature
that  the dependence on $q$ is made precise.

\begin{theorem}
\label{p2.1}
Let $f_1, \ldots, f_R \in \FF_q[x_1, \ldots, x_n]$ be forms of degree $d\geq 2$.
Let $\mathcal{C} > dR$ and suppose that Hypothesis \ref{hyp:2.1} holds for $C\geq 1$.
Assume that
$P>R(d-1) + d\deg m$ 
and $(1-\frac{1}{d})(\frac{n}{d}-R)>1$.
Then
$$
N(\mathbf{f};P,m,\b)= \mathfrak{S} \mathfrak{I} q^{(n - dR)P-dR\deg m}+
O_\ve\left(E\right),
$$
for any $\ve>0$,
with
\begin{align*}
E=~& q^{(n-\mathcal{C})P}
+
q^{1+dR(1 + \deg m) +\delta_0(1+d\deg m)(1-\frac{dR}{\mathcal{C}})+ (n - dR - \delta_1) P}\\
&\quad +
q^{(n - d R-\delta_2+\ve)P -( d-1) R \deg m+n 
+\delta_2(1+d\deg m+R(d-1))},
\end{align*}
and where
\begin{align}
\label{def del0}
\delta_0 &=
\frac{n - \sigma_\f}{ (d-1)2^{d-1}R  },\\
\label{def del1}
\delta_1& = \delta_0 \left( 1 - \frac{dR}{\mathcal{C}}  \right),
\end{align}
and  $\delta_2=
\left(1-\frac{1}{d}\right)\left(\frac{n}{d}-R\right)-1$.
Moreover, the singular series $\mathfrak{S}$ is defined in \eqref{def ss} and the singular integral $\mathfrak{I}$ in \eqref{def si}.
Finally,  the implicit constant depends only on
 $d, n, R, \mathcal{C},C$ and $\ve$.
\end{theorem}

For each integer $J$ in the range $0\leq J\leq P$, define
\begin{equation}\label{defn major arc}
\mathfrak{M}(J)=\mathfrak{M}_m(J;P)=
\hspace{-0.2cm}
\bigcup_{\substack{g\in \FF_q[t] \text{ monic}\\
0<|g|\leq q^J
}}
\bigcup_{\substack{\a\in \FF_q[t]^R\\
|\a|<|gm^d|\\ \gcd(\a,g)=1}} \left\{\bal\in \TT^R: \left|\bal-\frac{\a}{gm^d}\right|<\frac{q^{J-d
P+d}}{|gm^d|}\right\}
\end{equation}
to be the {\em major arcs
 of level $J$}.
The corresponding {\em minor arcs} are defined to be
$$
\mathfrak{m} (J) = \TT^R \setminus \mathfrak{M} ( J ).
$$
With this notation we may break the integral   \eqref{orthog} into major and minor arcs, obtaining
\begin{equation}
\label{NNN}
N(\mathbf{f};P,m,\b) =
\int_{\mathfrak{M}(\Delta(J))} S(\bal;P) \d \bal +
\int_{\mathfrak{m}(\Delta(J))} S(\bal;P) \d \bal.
\end{equation}
We claim that the major arcs of level $J$ are non-overlapping if
\begin{equation}\label{eq:overlap}
J\leq \frac{d(P-1)}{2}.
\end{equation}
To see this, if we have overlapping arcs associated to distinct  $\frac{\a}{gm^d}$ and
$\frac{\a'}{g'm^d}$, then we would be able to deduce that
$$
\frac{1}{|m|^d|gg'|}\leq \left|
\frac{\a}{gm^d}-\frac{\a'}{g'm^d}\right|<
\frac{q^{J-d
P+d}}{|m|^d\min\{|g|,|g'|\}},
$$
by the ultrametric inequality. Since $|g|,|g'|\leq q^J$, this leads to a contradiction under the inequality
\eqref{eq:overlap}.

The first  step in the treatment of the minor arcs will be  the following result, which is a function field analogue of \cite[Lemma 2.2]{simon}.

\begin{lemma}
\label{lem2.2}
Let  $\mathfrak{B}\subset \TT^R$
 be a measurable subset.
Let $\mathcal{C} > dR$ and assume that Hypothesis~\ref{hyp:2.1} holds with $C\geq 1$.
Suppose further that
$$
\sup_{\bal \in \mathfrak{B}} |S(\bal;P)| \leq  C_0 q^{ (n - \delta) P},
$$
for some $\delta >0$ and $C_0 \geq 1$. Then
\begin{align*}
\int_\mathfrak{B} |S(\bal;P)| \d \bal
\ll~&
C_0q^{
(n - \mathcal{C}) P}
+C_0 q^{1+dR(1 + \deg m)  +(n - dR - \delta(1 - \frac{ d R}{\mathcal{C}}) ) P},
\end{align*}
where the implicit constant depends only on $C, \mathcal{C}, d$ and $R$.

If $\delta$ is such that $\delta P\geq \mathcal{C}\left(1-\frac{d\deg m}{d-1}\right)$,  then
\begin{align*}
\int_\mathfrak{B} |S(\bal;P)| \d \bal
\ll~&
C_0q^{
(n - \mathcal{C}) P}
+C_0 q^{1+ d R \deg m +(n - dR - \delta(1 - \frac{ d R}{\mathcal{C}}) ) P}.
\end{align*}
\end{lemma}

\begin{proof}
Throughout this proof we allow any implied constant to depend on $C, \mathcal{C}, d$ and $R$.
For each $j \in \ZZ$, we set
$
D(j) = \{ \bal \in \mathfrak{B}: |S(\bal;P)| > q^j  \}
$
and we put $L(j)=\meas D(j)$.
Note that $L(j)\leq 1$, since $\TT^R$ has measure $1$.
Then, for any integers $J\leq K$, we have
\begin{align*}
\int_\mathfrak{B}
|S(\bal;P)| \d \bal
=~&
\int_{\mathfrak{B} \setminus D(J)} |S(\bal;P)| \d \bal
+
\sum_{j = J}^{K - 1}
\int_{D(j) \setminus D(j+1)} |S(\bal;P)| \d \bal
\\
&
+ \int_{ D(K)} |S(\bal;P)| \d \bal \\
\leq~&
q^{J} + \sum_{j = J}^{K -1}  L(j)q^{j+1}
+ L(K)
\sup_{\bal \in \mathfrak{B}} |S(\bal;P)|\\
\leq~&
q^{J} + \sum_{j = J}^{K -1}  L(j)q^{j+1}
+ C_0 q^{(n-\delta) P}L(K),
\end{align*}
by the hypothesis in the lemma.
We shall take
$$
J=nP-\lceil \mathcal{C}P\rceil \quad \text{ and } \quad
K=nP-\lceil \delta P\rceil.
$$
If $\delta P\geq \mathcal{C}P
$ then the upper bound in the lemma is a trivial consequence of
the hypothesis of the lemma, since $\mathfrak B$ has measure at most $1$.
 Hence we may proceed under the assumption that $J\leq K$.

We note that $q^J\leq q^{(n-\mathcal{C})P}$ and so the first term is satisfactory.
The second and third term combine to give
 \begin{align*}
\sum_{j = J}^{K -1}  L(j)q^{j+1}+
  C_0q^{(n-\delta)P}L(K)
&  \leq
  C_0\sum_{j = J}^{(n-\delta) P}  L(j)q^{j+1}.
\end{align*}
In particular, we have $j\leq nP-1$ in this sum.
We claim that
\begin{equation}\label{eq:measure-L}
L(j)\ll
\begin{cases}
q^{d R(\deg m -  P - \frac{(j-nP) }{ \mathcal{C} })}
& \text{ if $j\leq nP+\mathcal{C}\left(\frac{d\deg m}{d-1}-1\right)$,}\\
q^{dR(1-P -\frac{j-nP}{d\mathcal{C}})}
&
\text{ if $j> nP+\mathcal{C}\left(\frac{d\deg m}{d-1}-1\right)$,}
\end{cases}
\end{equation}
noting that the second case only occurs when $\deg m=0$.
Taking this on faith for the moment, let us see how it suffices to complete the proof of the lemma.

When
$\delta P\geq \mathcal{C}\left(1-\frac{d\deg m}{d-1}\right)$, we are always in the first case of
 \eqref{eq:measure-L}  and we deduce that
\begin{align*}
 \sum_{j = J}^{(n-\delta) P}  L(j)q^{j+1}&\ll
 q^{ d R(\deg m - P ) }
  \sum_{j = J}^{(n-\delta)P}
   q^{j+1-\frac{(j-nP)dR}{\mathcal{C}}}\\
    &=
q^{1+
dR(\deg m-  P)
+\frac{dnRP}{\mathcal{C}}} \sum_{j = J}^{(n-\delta)P}
   q^{j(1-\frac{dR}{\mathcal{C}})}\\
   &\ll
   q^{1+
d R(\deg m - P)
+\frac{dnRP}{\mathcal{C}}}
   q^{(nP-\delta P)(1-\frac{dR}{\mathcal{C}})}\\
&\leq   q^{1+ dR \deg m +(n-dR)P- \delta(1 - \frac{dR}{\mathcal{C}}) P},
\end{align*}
 since $\mathcal{C}>dR$, by assumption.  This is satisfactory for the second part of the lemma.
Alternatively, when the assumption
$\delta P\geq \mathcal{C}\left(1-\frac{d\deg m}{d-1}\right)$ is not made, we can combine the estimates in
 \eqref{eq:measure-L} to get the general upper bound
$L(j)\ll
q^{dR(1 + \deg m -P-\frac{(j-nP)}{\mathcal{C}})}$, 
whence
\begin{align*}
 \sum_{j = J}^{(n-\delta) P}  L(j)q^{j+1}&\ll
 q^{dR(1 + \deg m -P)}
  \sum_{j = J}^{(n-\delta)P}
   q^{j+1-\frac{(j-nP)dR}{\mathcal{C}}}\\
   &\ll
   q^{1+
dR(1 + \deg m  -P)
+\frac{dnRP}{\mathcal{C}}}
   q^{(nP-\delta P)(1-\frac{dR}{\mathcal{C}})}\\
&\leq   q^{1+dR (1 + \deg m )  +  (n-dR)P- \delta(1 - \frac{dR}{\mathcal{C}}) P},
\end{align*}
which completes the claim in the first part of the lemma.

  It remains to establish \eqref{eq:measure-L}  for any $j\leq nP-1$.
For any  $\bal, \bal + \bbe \in D(j)$, it follows from  Hypothesis~\ref{hyp:2.1} that
$$ \| m^d  \bbe \|< C^{1/\mathcal{C}}
q^{-dP+d-1  - \frac{j-nP}{\mathcal{C}}} \quad
\text{or}\quad  C^{-(d-1)/\mathcal{C}}
q^{(j-nP)\frac{d-1}{\mathcal{C}} +d-1}  <  \| m^d \bbe\|.
$$
We set
\begin{align*}
r_1(j)&=
-dP+d-1-d\deg m -\frac{j-nP}{\mathcal{C}} +\frac{\log_q C}{\mathcal{C}}
,\\
r_2(j)&=
(j-nP)\frac{d-1}{\mathcal{C}}-d\deg m+d-1-\frac{(d-1)\log_q C}{\mathcal{C}},
\end{align*}
in which notation the above conclusion can be written as
\begin{equation}
\label{ineq1}
\| m^d  \bbe\|< q^{ \min\{0,  r_1(j) + d \deg m \} } 
\quad \textnormal{ or } \quad
q^{r_2(j) + d \deg m}
<
 \|m^d  \bbe\|.
\end{equation}

Recall the definition of the boxes
from Section \ref{boxes}.
As explained there, since  $\mathfrak{B}\subset \TT^R$,
we can cover $D(j)$ by
at most
$$
\max\left\{1,q^{R(-\lceil r_2(j)\rceil)  }\right\}
\leq \frac{1}{q^{R\min\{r_2(j),0\}}}
$$
boxes of the form $B_{\lceil r_2(j)\rceil}(\z')$. Suppose that we are given any such box, with $\z' \in K_{\infty}^R$. Then we  claim that
\begin{equation}\label{eq:bomb}
\meas \left(  B_{\lceil r_2(j)\rceil}(\z') \cap D( j )  \right)  \leq   q^{ R (r_1(j) + d \deg m)},
\end{equation}
from which
it follows that
$$
L(j)\ll q^{R(r_1(j)+ d \deg m  - \min\{0,r_2(j)\})}.
$$
Suppose first that
$j\leq nP+\mathcal{C}\left(\frac{d\deg m}{d-1}-1\right)$.
Then $\min\{0,r_2(j)\} =  r_2(j)$ and it follows that
$
L(j)\ll
q^{R(d \deg m - dP - \frac{(j-nP)d}{\mathcal{C}})}$,
as claimed. In the opposite case we have
$\min\{0,r_2(j)\}\geq -\frac{(d-1)\log_qC}{\mathcal{C}}$ and it follows that
\begin{align*}
L(j)&\ll
q^{dR(1-P)-\frac{(j-nP)R}{\mathcal{C}}}.
\end{align*}
The  bound
claimed in \eqref{eq:measure-L} is now an easy consequence.

It remains to prove \eqref{eq:bomb}.
Clearly there is nothing to prove if the intersection is empty.
Thus, let us fix $\bal \in B_{\lceil r_2( j )\rceil }(\z') \cap D( j )$
and we suppose that  $\bal + \bbe \in B_{\lceil r_2( j )\rceil}(\z') \cap D( j )$,
for some $\bbe \in \TT^R$. 
Then
$$
\| m^d \bbe \| = \| m^d (\bal + \bbe - \z') -  m^d (\bal - \z') \|  < q^{\lceil r_2( j )\rceil + d \deg m }.
$$
Therefore,  it follows from  \eqref{ineq1} that
$ \| m^d \bbe \| <  q^{ \min\{0,  r_1(j) + d \deg m \}  } 
$. 
The proof of   \eqref{eq:bomb} now follows on appealing to Lemma \ref{lem:measure}.
\end{proof}

\subsection{Treatment of the  minor arcs}

Recall from (\ref{eq:multi}) the definition of the multilinear form associated to each form $f_k$, $1\leq k\leq R$.
Then
the multilinear form associated to $F_k$ is
$m^d \Psi^{(k)}_{i}(\x^{(1)}, \ldots, \x^{(d-1)})$, for $1\leq k\leq R$.
Recall the definition of $N^{(v)}(J; \bbe)$ from (\ref{def Nv}).
The following result is a version of  \cite[Lemma 4.3.4]{lee} in which the implied constant has been made explicit.

\begin{lemma}
\label{lem3.5}
Let $J,P \in \NN$ such that $J\leq P$ and let $\bal \in K_{\infty}^R$.
Then we have
$$
|S (\bal;P)  |^{2^{d-1}} \leq q^{n (2^{d-1} P -  (d-1) J ) } N^{(d-1)}(J; m^d \bal).
$$
\end{lemma}

\begin{proof}
We begin with an application of \cite[Cor.~4.3.2]{lee}, which leads to the inequality
$$
|S(\bal;P)|^{2^{d-1}}\leq q^{(2^{d-1}-d+1)nP   } N^{(0)}(J; m^d  \bal),
$$
for any $\bal \in K_{\infty}^R$.
Therefore, it suffices to prove
\begin{equation}\label{N0 bound}
N^{(0)}(J; m^d \bal) \leq
q^{-n (J - P) (d-1) }
N^{(d-1)}(J; m^d \bal),
\end{equation}
which in turn follows by showing that
$$
N^{(v-1)}(J; m^d \bal) \leq q^{-n (J - P)}  N^{(v)}(J; m^d \bal),
$$
for each $1 \leq v \leq d-1$.

This will be a straightforward consequence of  the version of the shrinking lemma
found in work of Browning and Sawin \cite[Lemma 6.4]{BSa}. To see this,
fix $1 \leq v \leq d-1$ and $\x^{(i)} \in \FF_q[t]^n$, for each $i \neq v$, satisfying
\begin{eqnarray}
\label{cond+++}
|\x^{(1)}|, \ldots, |\x^{(v - 1)}| < q^{J}
\quad
\textnormal{and}
\quad
|\x^{(v + 1)}|, \ldots, |\x^{(d - 1)}| < q^{P}.
\end{eqnarray}
Let
$$
L_i(\x)
=
\sum_{k = 1}^R \alpha^{ (k) } m^d  \Psi^{(k)}_i (\x^{(1)}, \ldots, \x^{(v-1)}, \x , \x^{(v+ 1)} \ldots, \x^{(d-1)}),
$$
for $1 \leq i \leq n$.
We write
$$
M(b; Z)  =
\# \{
(\x, \y)  \in \FF_q[t]^{2n}:
|\x|<  q^{b + Z}, ~|L_i(\x) + y_i|<  q^{- b + Z}, 1\leq i\leq n \},
$$
for any  $b, Z \in \ZZ$.
Now choose $b, Z_1, Z_2 \in \ZZ$ such that
\begin{eqnarray}
\label{cond lattice}
Z_2 \leq 0, \quad b + Z_2 \in \ZZ,  \quad b - Z_2 \in \ZZ_{>0} \quad \textnormal{and} \quad Z_2 - Z_1 \in \ZZ_{\geq 0}.
\end{eqnarray}
Taking $a=b+Z_2\in \ZZ$, $c=b-Z_2 \in \ZZ_{>0}$ and $s=Z_2-Z_1 \in \ZZ_{\geq 0}$ in
\cite[Lemma~6.4]{BSa}, it follows
that
$$
\frac{M(b; Z_1)}{M(b; Z_2)} \geq q^{n(Z_1 - Z_2)},
$$
for each $\x^{(i)}$ with $i \neq v$, such that
\eqref{cond+++} holds.
Setting
$$
b = P \frac{v+ 1}{2}  -  J \frac{ v- 1 }{2},
\quad
Z_1 = - P \frac{v + 1}{2} +  J \frac {v + 1}{2}
$$
and
$$
Z_2  =  - P \frac{v - 1}{2} +  J \frac{v - 1}{2},
$$
we clearly have
$$
b + Z_2 = P, \quad b - Z_2 = P v - J (v-1)
\quad
\textnormal{and}
\quad
Z_2 - Z_1 = P  - J,
$$
so that \eqref{cond lattice} holds.
Summing the above inequality over
all $\x^{(i)}$ with $i \neq v$, we arrive at the
inequality
$
N^{(v)}(J; m^d \bal) \geq q^{n (J - P)} N^{(v-1)}(J; m^d \bal),
$
which is what we were supposed to prove.
\end{proof}

Given a form $f \in \FF_q[t][x_1, \ldots, x_n]$ of degree $d$, we put
$$
\mathcal{M}_f
=
\left\{ (\x^{(1)}, \ldots, \x^{(d-1)}) \in \mathbb{A}^{(d-1)n }:
\Psi_j(\x^{(1)}, \ldots, \x^{(d-1)}) = 0,
~1 \leq j \leq n\right\},
$$
where $\Psi_{j}$ are the multilinear forms associated to $f$, and we let
$$
\mathcal{M}_f(J) = \mathcal{M}_f \cap \{ (\x^{(1)}, \ldots, \x^{(d-1)}) \in \FF_q[t]^{(d-1)n}:
|\x^{(1)}|, \ldots, |\x^{(d-1)}|  <   q^{J} \}.
$$
Let $\h \in \FF_q[t]^R$.
Since
\begin{align*}
V_{\h. \mathbf{f}} 
&=  \mathcal{M}_{\h. \mathbf{f}} \cap
 \{ (\x^{(1)}, \ldots, \x^{(d-1)}) \in \AA^{(d-1)n}
:  \x^{(1)} =  \cdots = \x^{(d-1)} \},
\end{align*}
in the notation of \eqref{eq:Vf},
it follows that
$$
\dim V_{\h. \mathbf{f}} + (d - 2)  n \geq \dim \mathcal{M}_{\h. \mathbf{f}},
$$
by the affine dimension theorem.  Hence
\begin{eqnarray}
\label{g inv}
\# \mathcal{M}_{\h. \mathbf{f}}(J) \ll q^{  ((d - 2) n + \dim V_{ \h. \mathbf{f}}) J  },
\end{eqnarray}
where the implicit constant depends only on $d$ and $n$.

\begin{lemma}
\label{lem2.5}
Let $J,P \in \NN$ such that $J\leq P$ and  let $\bal=(\alpha^{(1)},\dots,\alpha^{(R)}) \in K_{\infty}^R$.  Then one of the following two alternatives holds:
\begin{enumerate}[label= $(\textnormal{\roman*})$]
\item We have
$$
S(\bal;P)  \ll q^{n P -  \frac{(n  - \sigma_\f )}{2^{d-1}} J },
$$
where the implicit constant depends only on $d,n,R$, and $\sigma_\f$ is defined in \eqref{1.10}.

\item There exist $g, a_1, \ldots, a_R \in \FF_q[t]$, with $g$ monic,  such that $\gcd(g, \a) = 1$,
$$
0<  |g| \leq  q^{R(d-1) (J-1)}$$
and
$$
|g m^d\bal- \a| <  q^{- d P + R(d-1) J-(R-1)(d-1)}.
$$
\end{enumerate}
\end{lemma}

\begin{proof}
For $1\leq j\leq n$,
let $\mathbf{M}_j$ be the matrix whose columns are
$$
\begin{pmatrix}
  \Psi^{(1)}_{j}(\x^{(1)}, \ldots, \x^{(d-1)}) \\
  \vdots  \\
  \Psi^{(R)}_j (\x^{(1)}, \ldots, \x^{(d-1)})
\end{pmatrix},
$$
for each $(\x^{(1)}, \ldots, \x^{(d-1)})$ counted by $N^{(d-1)}(J; m^d \bal)$,
in the notation of
\eqref{def Nv}.
We then put these matrices together and define $\mathbf{M}=(\mathbf{M}_1,\dots,\mathbf{M}_n)$.
We consider two cases depending on the rank of $\mathbf{M}$.

Suppose first that   $\rank \mathbf{M} = R$. Then there exists a non-singular $R \times R$ submatrix
$$
\mathbf{M}_0
=
(\Psi^{(i)}_{j_\ell}( \x_{\ell}^{(1)}, \ldots, \x_{\ell}^{(d-1)} ))_{1 \leq i, \ell \leq R}.
$$
By the definition of $N^{(d-1)}(J; m^d \bal)$
there exist  $s^{(\ell)} \in \FF_q[t]$ and $|\gamma^{(\ell)}| < q^{-dP+(d-1)J}$ satisfying
\begin{equation}\label{eq:bird}
\sum_{i = 1}^R m^d  \alpha^{(i)} \Psi^{(i)}_{j_\ell}(\x_\ell^{(1)}, \ldots, \x_\ell^{(d-1)} )
= s^{(\ell)} + \gamma^{(\ell)},
\end{equation}
for each $1 \leq \ell \leq R$.
Let $\mathbf{C} = ( {c}_{i, \ell})_{1 \leq i, \ell \leq R}$ be the cofactor matrix of $\mathbf{M}_0$, so that
$\mathbf{C}$ has entries in $\FF_q[t]$ and satisfies
$$
\mathbf{M}_0 \mathbf{C}^T =  (\det \mathbf{M}_0)  \mathbf{I}_{R} = \mathbf{C}\mathbf{M}_0^T,
$$
where $\mathbf{I}_{R }$ is the $R \times R$ identity matrix.
Any entry in $\mathbf{C}$ has absolute value at most $q^{(R-1)(d-1)(J-1)}$,
since each entry of $\mathbf{M}_0$ has absolute value
$\leq  q^{(d-1)(J-1)}$.
On multiplying both sides of
\eqref{eq:bird} by $\mathbf{C}$ on the left, we  obtain
\begin{align*}
\begin{pmatrix}
(\det \mathbf{M}_0) m^d \alpha^{(1)} - \sum_{\ell = 1}^R  c_{1, \ell} s^{(\ell)}
\\
\vdots
\\
(\det \mathbf{M}_0) m^d \alpha^{(R)} - \sum_{\ell = 1}^R  c_{R, \ell} s^{(\ell)}
\end{pmatrix}
=
\mathbf{C}  \begin{pmatrix}
\gamma^{(1)}
\\
\vdots
\\
\gamma^{(R)}
\end{pmatrix}.
\end{align*}
We therefore obtain
\begin{align*}
\left| (\det \mathbf{M}_0) m^d  \alpha^{(i)} - \sum_{\ell = 1}^R  c_{i, \ell} s^{(\ell)} \right|
&<  q^{(R-1) (d-1) (J-1)} q^{- d P  + (d-1) J }\\
&=   q^{- d P + R (d-1) J-(R-1)(d-1) },
\end{align*}
for $1 \leq i \leq R$.
Set
$$
g = \frac{ \det \mathbf{M}_0 }{D} \quad \textnormal{and} \quad a_i = \frac{1}{D }
\sum_{\ell = 1}^R  c_{i, \ell} s^{(\ell)},
$$
for $1 \leq i \leq R$,
where
$$
D= \gcd \left(  \det \mathbf{M}_0,  \sum_{\ell = 1}^R  c_{1, \ell} s^{(\ell)}, \ldots, \sum_{\ell = 1}^R  c_{R, \ell} s^{(\ell)} \right).
$$
On dividing through by an element of $\FF_q^*$ we can further assume that $g$ is monic.
Noting that
$0< |\det \mathbf{M}_0| \leq  q^{R (d-1) (J-1)}$,
we therefore establish alternative (ii) of the lemma.

Next we suppose that $\rank \mathbf{M} < R$.
Then there exists $\h \in \FF_q[t]^R \setminus \{ \mathbf{0} \}$ such that
$$
\sum_{i = 1}^R h_i \Psi^{(i)}_j(\x^{(1)}, \ldots, \x^{(d-1)}) =0,
$$
for $1\leq j\leq n$ and  every
 $(\x^{(1)}, \ldots, \x^{(d-1)})$ counted by $N^{(d-1)}(J; m^d \bal)$.
Since
$$
\sum_{i = 1}^R h_i \Psi^{(i)}_j(\x^{(1)}, \ldots, \x^{(d-1)}) = \Psi_j^{(\h)}(\x^{(1)}, \ldots, \x^{(d-1)}) ,
$$
for $1 \leq j \leq n$,
where $\Psi_j^{(\h)}$ are the multilinear forms associated to the form $\h.\mathbf{f}$,
it then follows from \eqref{g inv} that
$$
N^{(d-1)}(J; m^d \bal) \leq \# \mathcal{M}_{\h . \mathbf{f}} (J) \ll
q^{J( (d-2) n + \dim V_{\h. \mathbf{f}} ) } \leq q^{J( (d - 2) n + \sigma_\f) },
$$
where the implicit constant depends only on $d$ and $n$, and $\sigma_\f$ is given in \eqref{1.10}.
Thus, by Lemma \ref{lem3.5},  it follows that
$$
S(\bal;P)  \ll q^{n P -  \frac{(n - \sigma_\f ) }{2^{d-1}}  J },
$$
which is precisely the bound in alternative (i).
\end{proof}

Let us set
\begin{equation}
\label{def Del}
\Delta(J) =  R(d-1)  (J-1)
\end{equation}
and recall the definition \eqref{def del0} of $\delta_0$.
Since the set of $\bal$ satisfying alternative (ii) of Lemma \ref{lem2.5} is contained
in the set of major arcs $\mathfrak{M}(\Delta(J))$, in the notation of \eqref{defn major arc}, the following result is an  immediate consequence of Lemma \ref{lem2.5}.

\begin{corollary}\label{cor:minor est1}
Let $J\leq P$ be integers. Then
$$
\sup_{\bal \in \mathfrak{m}(\Delta(J) ) } |S(\bal;P)| \ll
q^{n P -
(d-1)R\delta_0 J},
$$
where the implicit constant depends only on $d$, $n$ and $R$, and where $\delta_0$ and $\Delta(J)$ are given by \eqref{def del0} and \eqref{def Del}, respectively.
\end{corollary}

\subsection{A major arc approximation}

Recall the notation \eqref{eq:D1} for $S_g(\a)$ and
\eqref{eq:D2} for $S_\infty(\bbe)$.
We proceed by recording the following result, which refines work of Lee
\cite[Lemma~4.6.2]{lee}. (To be precise, Lee's version operates under the slightly more stringent assumption that
$|\bbe| < q^{-(d-1) P} |g|^{-1} |m|^{-d}$.)

\begin{lemma}
\label{lem2.4}
Let $P \in \NN$, let $g \in \FF_q[t]$ and $\a \in \FF_q[t]^R$. Put $\bal = \a/g + \bbe$ and assume
that $0 < |g| \leq  q^P$ and
$
|\bbe| <  q^{ -(d-1)(P-1) } |g|^{-1} |m|^{-d}.
$
Then
$$
S \left( \frac{\a}{g} + \bbe;P  \right)
 =
q^{n P  }
S_{g}(\a) S_\infty (m^d t^{dP} \bbe ).
$$
\end{lemma}

\begin{proof}
We begin by deducing from \eqref{eq:sum-S} that
$$
S \left( \frac{\a}{g} + \bbe ;P \right)
=
\sum_{|\x| < q^P} \psi \left( \left( \frac{\a}{g} + \bbe  \right). \f (m \x + \b)  \right).
$$
Let us write
$\x = g \y + \z$, where
$|\y| < q^P |g|^{-1}$ and $|\z| < |g|.$
Then it follows that
\begin{align*}
S \left( \frac{\a}{g} + \bbe;P  \right)
&=
\sum_{|\y| < q^P |g|^{-1}}  \sum_{  |\z| < |g| } \psi \left( \left( \frac{\a}{g} + \bbe  \right). \f (m (g \y + \z) + \b)  \right)
\\
&=
\sum_{  |\z| < |g| } \psi \left(  \frac{\a . \f (m \z + \b) }{g}   \right)
\sum_{|\y| < q^P |g|^{-1}}   \psi (  \bbe  . \f (m (g \y + \z) + \b)  ).
\end{align*}
We claim that
$
\psi (  \bbe  . \f (m (g \y + \z) + \b)  ) =  \psi (  \bbe  . \f (m g \y  + \b)  )
$
in the inner sum,
under the assumptions of the lemma,
from which it will follow that
\begin{equation}\label{eq:pen-pot}
S \left( \frac{\a}{g} + \bbe ;P \right)
=
|g|^n S_g(\a) \sum_{|\y| < q^P |g|^{-1}}   \psi (  \bbe  . \f (m g \y  + \b)  ).
\end{equation}
To check the claim it will suffice to show that
\begin{equation}\label{eq:toe-cut}
\ord (  \bbe  .   (\f (m (g \y + \z) + \b)   -  \f (m g \y  + \b))  ) < -1,
\end{equation}
for any $|\z|<|g|$.
Note that $|m g \y  + \b|\leq q^{\deg m+ P-1}$. Hence it follows from Taylor expansion that
the left hand side of this expression is
\begin{align*}
\ord (  \bbe  .  & (\f (m (g \y + \z) + \b)   -  \f (m g \y  + \b))  ) \\
&\leq  \ord \bbe + \max \left(  \ord (m \z) + (d-1) \ord ( m g \y + \b ), d \ord (m \z)     \right)
\\
& \leq    \ord \bbe + \max \left(  \deg m^d + \deg g - 1
+ (d-1)   (P-1)  , \deg m^d  + d(\deg g - 1)     \right)
\\
& =    \ord \bbe + (d-1)   (P-1)  +  \deg g +     \deg m^d - 1
\\
&< -1,
\end{align*}
as required.
Now that we have verified \eqref{eq:pen-pot}, a further application of \eqref{eq:toe-cut} reveals that
\begin{align*}
\sum_{|\y| < q^P |g|^{-1}}   \psi (  \bbe  . \f (m g \y  + \b)  )
&=
|g|^{- n} \sum_{  |\z| < |g| } \sum_{|\y| < q^P |g|^{-1}}   \psi (  \bbe  . \f ( m ( g \y + \z ) + \b )  )
\\
&=
|g|^{- n} \sum_{|\x| < q^P}   \psi (  \bbe  . \f ( m \x + \b )  ).
\end{align*}
Hence
\begin{eqnarray}
\label{last eq}
S \left( \frac{\a}{g} + \bbe;P  \right) = S_g(\a) S(\bbe;P).
\end{eqnarray}

Arguing as in \eqref{eq:toe-cut}, it is easy to check that
$$
\ord (  \bbe  .   (\f (m (\x + \boldsymbol{\sigma}) + \b)   -  \f (m \x + \b))  ) < -1,
$$
for any  $|\x| < q^P$ and   $\boldsymbol{\sigma} \in \TT^n$.
Thus
$
\psi (  \bbe  . \f ( m \x  + \b )  )  = \psi (  \bbe  . \f ( m \v  + \b )  ),
$
for any $|\x - \v| < 1$. It now follows that
\begin{align*}
S(\bbe;P) &=
\sum_{|\x| < q^P}  \int_{|\x - \v| < 1}  \psi (  \bbe  . \f ( m \x  + \b )  )  \d \v
\\
&=
\sum_{|\x| < q^P}   \int_{|\x - \v| < 1}    \psi ( \bbe  . \f ( m \v + \b )  )  \d \v
\\
&=
\int_{|\v| < q^P}   \psi (  \bbe  . \f ( m \v  + \b )  )  \d \v\\
&=q^{nP}
\int_{\TT^n}   \psi (  \bbe  . \f ( t^P m \v  + \b )  )  \d \v,
\end{align*}
on making a change of variables.
Next, we claim that
$$
\ord (  t^{d P} \bbe  .   ( \f (m \v + \b t^{-P})   -  \f (m\v)  ) ) < -1,
$$
for any  $\v \in \TT^n$. Taking this on faith for the moment, this will imply that
\begin{align*}
S(\bbe;P)
=  q^{n P} \int_{\TT^n}   \psi ( m^d  t^{d P}  \bbe  . \f ( \v  )  )  \d\v
= q^{n P} S_\infty ( m^d t^{d P}  \bbe ),
\end{align*}
which will complete the proof of the lemma, on  substituting back  into \eqref{last eq}.
To check the claim, we note that the left hand side is
\begin{align*}
&\leq  dP + \ord \bbe + \max \left( \ord (\b t^{-P}) + (d-1) \ord ( m \v ), d \ord (\b t^{-P})     \right)
\\
&\leq  dP + \ord \bbe + \max \left(\deg m - 1 - P  + (d-1) ( \deg m - 1 ), d ( \deg m - 1 - P )     \right)
\\
&=  dP +   \ord \bbe  - P  + d ( \deg m - 1 )
\\
& =   \ord \bbe + (d-1)(P-1) + \deg m^d - 1
\\
&< -1,
\end{align*}
as required.
\end{proof}

\subsection{Deduction  of Theorem  \ref{p2.1}}

Let $\mathcal{C}>dR$.
We now have everything in place to complete the proof of
Theorem  \ref{p2.1}, which we tackle under the assumption that
Hypothesis \ref{hyp:2.1} holds and 
$P>R(d-1) + d\deg m$.
Recall the notation for  $\Delta(J)$  in \eqref{def Del}
and the definition of the major arcs
 $\mathfrak{M}(\Delta(J))=
\mathfrak{M}_m(\Delta(J);P)$
in
\eqref{defn major arc}.
We require the major arcs to be non-overlapping, which in view of \eqref{eq:overlap}, means that we need
\begin{equation}\label{eq:overlap'}
\Delta(J)\leq \frac{d(P-1)}{2}.
\end{equation}
We shall apply Lemma
\ref{lem2.4}  with $gm^d$ in place of $g$.
We also need to choose $J $ in order  that
this lemma is applicable on the major arcs $\mathfrak{M}(\Delta(J))$.
Thus we need $|gm^d|\leq  q^P$ and $|gm^d\bbe|<q^{-(d-1)(P-1)}|m|^{-d}$; in other words, we need
$$
|g|\leq  q^{P-d\deg m} \quad \text{ and }\quad |g\bbe|<q^{-(d-1)(P-1)-2d\deg m}.
$$
But
$\bal\in \mathfrak{M}(\Delta(J))$ implies that $|g|\leq q^{\Delta(J)}$ and
$|g\bbe|<q^{\Delta(J)-dP+d}|m|^{-d}$. Thus we need
$\Delta(J)\leq  P-d\deg m$ and $\Delta(J)-dP+d\leq -(d-1)(P-1)-d\deg m.$
These two inequalities are satisfied if $\Delta(J)\leq P-1-d\deg m$,
which in the light of \eqref{def Del},
we can achieve by setting
\begin{equation}\label{eq:lower-mogo}
J= \left\lfloor \frac{P-1-d\deg m}{R(d-1)} \right\rfloor +1 \geq \frac{P-1-d\deg m}{R(d-1)}.
\end{equation}
This is also enough to ensure that \eqref{eq:overlap'} holds.

We break the integral into major and minor arcs, as in \eqref{NNN}.
For the minor arcs, we apply the first part of Lemma \ref{lem2.2} with $\mathfrak{B} = \mathfrak{m}(\Delta(J))$
and
$\delta = (d-1)R
\delta_0 J/ P$, and where $C_0$ is the implicit constant from
Corollary \ref{cor:minor est1}.
Hence it follows that
\begin{align*}
\int_{\mathfrak{m}(\Delta(J))} |S(\bal;P)| \d \bal
\ll~&
q^{ (n - \mathcal{C}) P}
+q^{1+dR(1 + \deg m)+(n - dR)P -
(d-1)R\delta_0J(1 - \frac{ d R}{\mathcal{C}})}.
\end{align*}
Recalling that $\mathcal{C}>dR$, together with the lower bound
\eqref{eq:lower-mogo},
we see that
\begin{align*}
1+dR&(1 + \deg m)+(n - dR)P -
(d-1)R\delta_0 J\left(1 - \frac{ d R}{\mathcal{C}}\right)
\\
&\hspace{-0.5cm}
\leq
1+dR(1 + \deg m) + (n - dR-\delta_1)P +
\delta_0 (1+d\deg m)\left(1 - \frac{ d R}{\mathcal{C}}\right),
\end{align*}
where
$\delta_1$ is given by
\eqref{def del1}.
Hence
\begin{equation}\label{MIN}
\begin{split}
\int_{\mathfrak{m}(\Delta(J))} |S(\bal;P)| \d \bal
&\ll
q^{(n-\mathcal{C})P}+
q^{1+dR(1 + \deg m)+\delta_0(1+d\deg m)(1-\frac{dR}{\mathcal{C}})+ (n - dR - \delta_1) P}.
\end{split}
\end{equation}

Turning to
the major arcs, it follows from  Lemma \ref{lem2.4}
that
\begin{align*}
\int_{\mathfrak{M}(\Delta(J))} S(\bal;P) \d \bal  &=
\sum_{\substack{
g\in \FF_q[t] \text{ monic}\\
|g| \leq  q^{\Delta(J)}}} \sum_{ \substack{
\a\in \FF_q[t]^R
\\
|\a| < |gm^d| \\ \gcd(\a, g) = 1 } }
\int_{ |\bbe| < \frac{q^{\Delta(J)  - d P+d}}{|gm^d|} }
S \left( \frac{\a}{g m^d} + \bbe;P  \right)  \d \bbe
\\
&=
q^{(n - d R)P - d R \deg m}
\sum_{\substack{
g\in \FF_q[t] \text{ monic}\\
|g| \leq  q^{\Delta(J)}}} \sum_{ \substack{
\a\in \FF_q[t]^R
\\
|\a| < |gm^d| \\ \gcd(\a, g) = 1 } }S_{gm^d}(\a)
\mathfrak{I}(B),
\end{align*}
where 
$\mathfrak{I}(B)$ is given by 
\eqref{def si-t}, with 
$B=\Delta(J) +d-\deg g$.
According to the first part of  Lemma~\ref{si}, we have
\begin{align*}
| \mathfrak{I}
- \mathfrak{I}(B)|
&\ll  q^{n-(1-\frac{1}{d})(\frac{n}{d}-R)B}\\
&\ll  q^{\frac{n}{d}+dR-R-(1-\frac{1}{d})(\frac{n}{d}-R)\Delta(J)}|g|^{(1-\frac{1}{d})(\frac{n}{d}-R)}.
\end{align*}
Appealing to Lemma \ref{lem:snorlax} and 
recalling the 
notation \eqref{eq:ST}, 
 this error term contributes
\begin{align*}
&\ll
q^{(n - d R)P - d R \deg m+\frac{n}{d}+(d-1)R-(1-\frac{1}{d})(\frac{n}{d}-R)\Delta(J)}|m|^R
\sum_{T=0}^{ \Delta(J)}
q^{(1-\frac{1}{d})(\frac{n}{d}-R)T}
\mathcal{S}(T).
\end{align*}
Next, it follows from Lemma \ref{lem:toad} that
this contribution is
\begin{align*}
&\ll_\ve
q^{(n - d R)P - (d-1) R \deg m+\frac{n}{d}+(d-1)R-(1-\frac{1}{d})(\frac{n}{d}-R)\Delta(J)+P\ve}
\sum_{T=0}^{ \Delta(J)}q^T\\
&\ll_\ve
q^{(n - d R)P - (d-1) R \deg m+\frac{n}{d}
+(d-1)R-\left((1-\frac{1}{d})(\frac{n}{d}-R)-1\right)\Delta(J)+P\ve},
\end{align*}
for any $\ve>0$.
Finally,
we may  apply  Lemmas \ref{lem2.5iv}  and \ref{si} to deduce that
\begin{align*}
\mathfrak{S}(\Delta(J))
\mathfrak{I}
=
\mathfrak{S}\mathfrak{I} +O_\ve\left(
q^{n+R\deg m+(\Delta(J)+1)(1-(1-\frac{1}{d})(\frac{n}{d}-R))+P\ve}
\right).
\end{align*}

Let
$$
E_{\text{major}}=\left| \int_{\mathfrak{M}(\Delta(J))} S(\bal;P) \d \bal -  \mathfrak{S} \mathfrak{I} q^{(n - dR) P - d R \deg m} \right|,
$$
where $\mathfrak{S}$ and $\mathfrak{I}$ are given by \eqref{def ss} and \eqref{def si}, respectively.
Then we have shown that
\begin{align*}
E_{\text{major}}
  \ll_\ve~&
q^{(n - d R)P - (d-1) R \deg m+P\ve}
\left(
q^{\frac{n}{d}
+(d-1)R}
+
q^{n-\left((1-\frac{1}{d})(\frac{n}{d}-R)-1\right)}
\right)\\
&\times q^{-\left((1-\frac{1}{d})(\frac{n}{d}-R)-1\right)\Delta(J)},
\end{align*}
for any $\ve>0$,
on recalling our assumption that $(1-\frac{1}{d})(\frac{n}{d}-R)>1$.  Under this assumption one easily checks that
the middle term in brackets is at most $2q^{n}$. Moreover,
it follows from  \eqref{def Del}  and \eqref{eq:lower-mogo} that
$$
\Delta(J)=R(d-1)(J-1)\geq P-1-d\deg m-R(d-1).
$$
Note that  $\Delta(J)\geq 0$ under the assumptions of
Theorem  \ref{p2.1}.
Hence we obtain
\begin{align*}
E_{\text{major}}
&  \ll_\ve
q^{(n - d R-\delta_2+\ve)P -( d-1) R \deg m+n +
\delta_2(1+d\deg m+R(d-1))},
\end{align*}
where $\delta_2=
\left(1-\frac{1}{d}\right)\left(\frac{n}{d}-R\right)-1$.
Finally, we complete the proof of
Theorem  \ref{p2.1} on combining this with
 \eqref{MIN} in \eqref{NNN}.

\section{Arithmetic deductions}\label{sec4}

The main goal of this section is to prove  Theorem \ref{t:simon}.
We begin by proving Hypothesis \ref{hyp:2.1} for a suitable $C\geq 1$, for which we will need to assume  that $q\geq (d-1)^n$.
Recalling the definition of $N^{(d-1)}(J; \bbe)$ from \eqref{def Nv},
and the definition of the sum
$S(\bal;P) = S(\bal; P , m, \b)$ from
\eqref{eq:sum-S},
the first step is to prove the following result.

\begin{lemma}
\label{last2}
Let $J,P \in \NN$ such that  $J \leq P$.
Then
$$
\min \left \{ \frac{|S(\bal;P)|}{q^{n P}} ,  \frac{|S(\bal + \bbe;P)|}{q^{n P}}  \right\}^{2^d}
\leq
q^{- n (d-1) J  } N^{(d-1)}(J; m^d \bbe),
$$
for any $\bal, \bbe \in K_{\infty}^R$.
\end{lemma}
\begin{proof}
Let us write $S(\bal;P)=S(\bal)$ for simplicity.
We start by noting that
\begin{align*}
\min \{ |S(\bal)|, &|  S(\bal + \bbe)|  \}^2\\
&\leq |S(\bal)| |  S(\bal + \bbe)|
\\
&=
|S(\bal + \bbe) S(-\bal)|
\\
&=
 \sum_{|\x| < q^{P}} \sum_{|\y| < q^{P}}
\psi \left((\bal +  \bbe). \mathbf{f}(m  \x + \b) - \bal. \mathbf{f}(m \x + m \y + \b)\right)
\\
&\leq
 \sum_{|\y| < q^{P}}
|\widetilde{S}_{\bal, \y}(\bbe)|,
\end{align*}
where
$$
\widetilde{S}_{\bal, \y}(\bbe) = \sum_{|\x| < q^{P}}
\psi (m^d \bbe. \mathbf{f}(\x) + G_{\bal, \bbe, \y, m, \b}(\x)),
$$
for a suitable polynomial  $G_{\bal, \bbe, \y, m, \b}$ of degree $< d$.
Precisely the same argument used in  \cite[Cor.~4.3.2]{lee} now yields
$$
|\widetilde{S}_{\bal, \y}(\bbe)|^{2^{d-1}} \leq q^{(2^{d-1}-d+1)n P}N^{(0)}(J; m^d \bbe),
$$
where $N^{(0)}(J; \bbe)$ is defined in \eqref{def Nv}. (Note that the term  $G_{\bal, \bbe, \y, m, \b}$ gets eliminated during the Weyl differencing process, since it has degree $<d$.)
An application of
\eqref{N0 bound} implies
$$
|\widetilde{S}_{\bal, \y}(\bbe)|^{2^{d-1}} \leq q^{ n 2^{d-1}  P  -  n (d-1) J  } N^{(d-1)}(J; m^d \bbe).
$$
With this estimate, we may deduce that
\begin{align*}
\min \{ |S(\bal)|, |  S(\bal + \bbe)|  \}^2
&\leq
 \sum_{|\y| < q^{P}}
q^{ n P -  \frac{n(d-1)}{2^{d-1}} J } N^{(d-1)}(J; m^d \bbe)^{\frac{1}{2^{d-1}}}
\\
&=
q^{2 n P -  \frac{n (d-1)}{2^{d-1}} J } N^{(d-1)}(J; m^d \bbe)^{\frac{1}{2^{d-1}}}.
\end{align*}
The result follows by rearranging this inequality.
\end{proof}

With this lemma to hand, we may now invoke the estimates from Section~\ref{aux} to prove
Hypothesis \ref{hyp:2.1}.
Let
\begin{equation}
\label{setcalC}
\mathcal{C} =  \frac{1}{2^{d+\nu(d)} (d-1)}  \left(  \left( 1 -  \log_q (d-1) \right) n  - R+1 \right) ,
\end{equation}
where
\begin{equation}\label{eq:nu-d}
\nu(d)=\begin{cases}
1 & \text{ if $d=2$,}\\
0 & \text{ if $d\geq 3$}.
\end{cases}
\end{equation}
We shall prove that
Hypothesis \ref{hyp:2.1} holds with
$$
C=(d-1)^{\frac{n}{2^d}},
$$
under
 the
 assumptions $q\geq (d-1)^n$ and
   $\mathcal{C}\geq 1$. (Later we shall need to assume that $\mathcal{C}>dR$.)
Let $P \in \NN$ and let $\bal, \bbe \in K_{\infty}^R$. 
Let
$C_0=(d-1)^{n}> 1
$
and write
$$
\gamma=C_0^{\frac{1}{2^d}}\min \left\{  \frac{|S( \bal;P)| }{q^{n P}}, \frac{|S( \bal + \bbe;P)| }{q^{n P}}   \right\}^{-1}.
$$
Note that $\gamma>1$. Thus there exists  $N\in \ZZ_{\geq 0}$ and $\ve_0\in [0,1)$ such that
$$
\gamma^{1/\mathcal{C}}=q^{N+\ve_0}.
$$
We must now differentiate according to whether or not $N$ is positive.

\subsection*{The case $N\geq 1$}

Suppose first that $N\geq 1$ and
let
$
\Delta=\ve_0\mathcal{C}/N\geq 0$.
Then
 $$
\gamma^{1/(\mathcal{C}+\Delta)}=q^{N},
$$
and we define $\eta_0=N/P$.
Suppose for the moment that  $\eta_0 > 1$, so that $N>P$. Then
\begin{align*}
\min \left\{   \frac{|S( \bal;P)| }{  q^{n P} },  \frac{|S( \bal + \bbe;P)| }{   q^{n P} }  \right\}  &=
C_0^{\frac{1}{2^d   } } \gamma^{-1}
\\
&=
C_0^{\frac{1}{2^d   } } q^{-N(\mathcal{C}+\Delta)}
\\
&\leq
C_0^{\frac{1}{2^d   } } q^{-P\mathcal{C}}.
\end{align*}
But
$
q^{-P}\leq
\max \{ q^{-d P+d-1} \| m^d\bbe\|^{-1}, \| m^d\bbe \|^{\frac{1}{d-1}}q^{-1} \},
$
as  checked by considering the two cases $q^{-d P+d-1} \| m^d\bbe\|^{-1} \leq \| m^d\bbe \|^{\frac{1}{d-1}}q^{-1}$
and $q^{-d P+d-1} \| m^d\bbe\|^{-1} > \|m^d \bbe \|^{\frac{1}{d-1}}q^{-1}$ separately. In particular it follows that
\begin{align*}
\min \left\{   \frac{|S( \bal;P)| }{  q^{n P} },  \frac{|S( \bal + \bbe;P)| }{   q^{n P} }   \right\}  &\leq
C_0^{\frac{1}{2^d   } }\max \{ q^{-d P+d-1} \| m^d\bbe\|^{-1}, \| m^d\bbe \|^{\frac{1}{d-1}}q^{-1} \}^{\mathcal{C}},
\end{align*}
which implies Hypothesis \ref{hyp:2.1} with $C=C_0^{\frac{1}{2^d}}$.

We may therefore proceed under the assumption that
$\eta_0P\in \ZZ\cap [1,P]$. In this case
$$
C_0^{\frac{1}{2^d\mathcal{C}}}\min \left\{   \frac{|S( \bal;P)| }{q^{n P}} ,  \frac{|S( \bal + \bbe;P)| }{q^{n P}}   \right\}^{-1/\mathcal{C}}=
\gamma^{1/\mathcal{C}}=
q^{N+\ve_0}.
$$
Since $N+\ve_0\geq N=\eta_0P$, we obtain
$$
C_0^{-\frac{1}{2^d  \mathcal{C} } } \min \left\{   \frac{|S( \bal;P)| }{q^{nP}} ,  \frac{|S( \bal + \bbe;P)| }{q^{nP}}   \right\}^{1/\mathcal{C}}
\leq  q^{-\eta_0P}.
$$
Suppose that alternative (i) holds in
Lemma~\ref{last1}, with $J=\eta_0P$ and $\{m^d\bbe \}$. 
Then
it immediately follows that
Hypothesis \ref{hyp:2.1} holds with $C=C_0^{\frac{1}{2^d}}$.
If, on the other hand, alternative (ii) holds, then
$$
q^{\eta_0 P -dP+d-1}\leq
|\{ m^d\bbe \} |\leq q^{-\eta_0P(d-1)+d-2}
$$
 and
$
N^{(d-1)}( \eta_0 P ; \{m^d \bbe\})=
N^{\textnormal{aux}}( \eta_0 P ;  \{m^d\bbe\} t^{dP-\eta_0 P} ),
$
in the notation of
\eqref{eq:rabbit} and \eqref{def Nv}.
We have
\begin{align*}
q^{n(d-1)\eta_0P} \gamma^{-2^d}
&=
C_0^{-1}
q^{n(d-1)\eta_0P}
 \min \left\{   \frac{|S( \bal;P)| }{q^{n P}},  \frac{|S( \bal + \bbe;P)| }{q^{n P}}  \right\}^{2^d}\\
&\leq
C_0^{-1}
N^{(d-1)}( \eta_0 P ; m^d \bbe)\\
&= 
C_0^{-1}
N^{(d-1)}( \eta_0 P ; \{m^d \bbe\})\\
&= C_0^{-1}
N^{\textnormal{aux}}( \eta_0 P ;  \{m^d\bbe\} t^{dP-\eta_0 P} ),
\end{align*}
by Lemma \ref{last2}.
We would now like to apply Corollary \ref{c:Nbf}  to estimate the right hand side.
Since   we are working under the assumption that
$
|\{m^d\bbe\} t^{dP-\eta_0 P}|\geq
q^{d-1},
$
we obtain
\begin{align*}
N^{\textnormal{aux}}( \eta_0 P ;  \{m^d\bbe\} t^{dP-\eta_0 P} )
 &\leq (d-1)^n  q^{n(d-1)\eta_0P}q^{-2^{d+\nu(d)}(d-1)\mathcal{C} \theta_d(\eta_0P)},
  \end{align*}
  where $\nu(d)$ is given by \eqref{eq:nu-d}.
Note that $\gamma^{-2^d}=q^{-(N+\ve_0)2^d\mathcal{C}}
=q^{-\eta_0P 2^d\mathcal{C}-\ve_02^d\mathcal{C}}.
$
Hence we conclude that
$$
q^{-\ve_02^d\mathcal{C}}\leq C_0^{-1}(d-1)^n
q^{\eta_0P 2^d\mathcal{C}}\cdot
q^{-2^{d+\nu(d)}(d-1)\mathcal{C} \theta_d(\eta_0P)}.
$$
Suppose first that $d\geq 3$, so  that  $\nu(d)=0$. If  $\eta_0P\geq 2$ then $\theta_d(\eta_0P)=\frac{\eta_0 P}{d-1}+1$ and 
we may deduce that
$
q^{-\ve_02^d\mathcal{C}}\leq
q^{-2^d(d-1)\mathcal{C}}.
$
If $\eta_0P=1$  then  $\theta_d(\eta_0P)=1$ and 
we may deduce that
$
q^{-\ve_02^d\mathcal{C}}\leq
q^{2^d\mathcal{C}}\cdot q^{-2^d (d-1)\mathcal{C}}=q^{-2^d(d-2)\mathcal{C}}.
$
Neither of these cases is  possible, since  $\ve_0<1\leq d-2$. 
 If $d=2$, on the other hand, then
$\nu(d)=1$ and
$\theta_d(\eta_0P)=\eta_0P$, whence
$
q^{-4\ve_0\mathcal{C}}\leq
q^{4\mathcal{C}\eta_0P }\cdot
q^{-8\mathcal{C}\eta_0P}=q^{-4\mathcal{C}\eta_0P }.
$
This too is a contradiction.

\subsection*{The case $N= 0$}
Suppose that alternative (i) holds in
Lemma~\ref{last1}, with $J=1$ and $\{m^d\bbe\}$. 
Then
$$
q^{-1} \leq  \max \left\{  q^{- d P + d- 2 } \|m^d\bbe\|^{-1}, \|m^d\bbe\|^{\frac{1}{d-1}} q^{- 1}
\right\}.
$$
We can't have  $q^{-1} \leq  \|m^d\bbe\|^{\frac{1}{d-1}} q^{-1}$, since $\|m^d\bbe \|< 1$. 
Thus 
$
q^{-1} \leq    q^{- d P + d- 2 } \|m^d\bbe \|^{-1},
$
whence
$1\leq    q^{- d P + d- 1 } \|m^d\bbe\ |^{-1}.$ It follows that
\begin{align*}
\min \left\{   \frac{|S( \bal;P)| }{  q^{n P} },  \frac{|S( \bal + \bbe;P)| }{   q^{n P} }   \right\}^{2^d}
\hspace{-0.2cm}
&=C_0q^{-2^d\ve_0 \mathcal{C}}\\
&\leq
C_{0}  \left(  q^{- d P + d - 1 } \|m^d\bbe\|^{-1} \right)^{2^d \mathcal{C}}
\\
&\leq
C_{0}  \max \left\{  q^{- d P + d - 1	 } \|m^d\bbe\|^{-1}, \|m^d\bbe\|^{\frac{1}{d-1}} q^{-1} \right\}^{2^d \mathcal{C}}.
\end{align*}
This  therefore implies
Hypothesis \ref{hyp:2.1} with $C = C_{0}^{\frac{1}{2^d } }$.

It remains to deal with alternative (ii)
of Lemma~\ref{last1}, with $J=1$ and $\{m^d\bbe\}$.
Then
$q^{- dP + d }   \leq  \|m^d\bbe\|  \leq q^{- 1}$
and
$N^{(d-1)}(1;  \{m^d\bbe\}) =
N^{\textnormal{aux}}( 1;  \{m^d\bbe\} t^{dP-1} ).
$
Taking $J=1$ in
Lemma \ref{last2}, we find that
\begin{align*}
C_{0}  q^{ -  2^d \varepsilon_0 \mathcal{C}}
&=
\min \left\{   \frac{|S( \bal;P)| }{  q^{n P} },  \frac{|S( \bal + \bbe;P)| }{   q^{n P} }  \right\}^{2^d}
\\
&\leq 
q^{- n (d-1) } N^{(d-1)} (1; m^d \bbe)\\
\\
&=
q^{- n (d-1) } N^{(d-1)} (1; \{m^d \bbe\})\\
&=
q^{- n (d-1) } N^{\textnormal{aux}}( 1;  \{m^d\bbe\} t^{dP-1} ).
\end{align*}
Since $\nu(d)\geq 0$, 
an application of Corollary \ref{c:Nbf} now yields
\begin{align*}
C_{0}  q^{ -  2^d \varepsilon_0 \mathcal{C}}
&\leq
q^{- n (d-1) }\cdot
 (d-1)^n  q^{n(d-1)}\cdot q^{-2^d(d-1)\mathcal{C} \theta_d(1)}\\
&=
 (d-1)^n
 q^{-2^d (d-1)\mathcal{C}\theta_d(1)}.
\end{align*}
We have $\theta_d(1)=1$ for $d\geq 2$ and we see that this
is a  contradiction, since $C_0=(d-1)^n$ and $\ve_0<1\leq d-1$.

\begin{proof}[Proof of Theorem \ref{t:simon}]

Let $\FF_q$ be a fixed finite field of characteristic $>d$.
Under the hypotheses of the theorem we have $q\geq (d-1)^n$.
Let $f_1, \ldots, f_R \in \FF_q[x_1, \ldots, x_n]$ be degree $d$ forms,  cutting
out a smooth complete intersection in $\PP^{n-1}$.
In particular $\sigma_\f\leq R-1$ in
Theorem \ref{p2.1}, by Remark \ref{rem:sigma-f}.

The proof of
Theorem \ref{t:simon} is
deduced from Theorem \ref{p2.1} in precisely the same way that
Theorem~4.1.2 is deduced from Equation (4.9.1) in
\cite{lee}, the details of which will not be  repeated here.
We may suppose that we are given $m\in \FF_q[t]\setminus\{0\}$ and $\b\in \FF_q[t]^n$
with $\deg b_i<\deg m$
and $\gcd(\b,m)=1$, and
such that  the system of equations
$$
f_1(m\x+\b)=\dots=f_R(m\x+\b)=0
$$
is everywhere locally soluble.
Under these assumptions it will then suffice to prove that
$N(\f;P,m,\b)>0$, for sufficiently large values of $P$, in the notation of \eqref{eq:defn-NP}.

We fix the choice of
$\mathcal{C}$ in \eqref{setcalC}.
Since $q\geq (d-1)^n$, it follows that
$$
\left( 1 -  \log_q (d-1) \right)n\geq n-1.
$$
Hence $\mathcal{C}>dR$
holds under the assumption
$n> d(d-1)2^{d+\nu(d)}R+R$, where
$\nu(d)$ is given by \eqref{eq:nu-d}.
In the proof of  Theorem \ref{t:simon},  all implied constants are allowed to depend on
$q$, on $d,n,R$, and on $m$ and $\b$.
In particular, it follows that
Hypothesis~\ref{hyp:2.1} holds with $C=O(1)$.
Thus  Theorem
\ref{p2.1} yields  $\delta>0$ such that
$$
N(\f;P,m,\b)= \mathfrak{S} \mathfrak{I} q^{(n - dR)P-dR\deg m}+
O\left(
q^{ (n - dR -  \delta) P }\right).
$$
Moreover, Lemmas \ref{lem2.5iv} and  \ref{si} yield $ \mathfrak{S} \mathfrak{I}>0$.
\end{proof}

\section{Geometric deductions}

We will use  the argument deployed in \cite{BV'} to establish the irreducibility
 and dimension of the space $\mathrm{M}_{e,b}$ that is the object of Theorem \ref{t:BV2}.
 This is based on a counting argument over a finite field $k=\FF_q$
 whose characteristic is greater than the degrees $d$ of the  forms $f_1,\dots,f_R\in \FF_q[x_1,\dots,x_n]$ that define $X$. We may further assume that $q$
 is arbitrarily large.
 We fix a sufficiently small parameter $\ve>0$ and
we  henceforth suppose that $q>
  (d-1)^{1/\ve}$, so that
\begin{equation}\label{eq:0,ep}
0< \frac{ \log (d-1)}{\log q} <\ve.
\end{equation}
We may further assume that
$p_1,\dots,p_b\in \PP^1(\FF_q)$ and $y_1,\dots,y_b\in X(\FF_q)$ in
the definition
\eqref{eq:def_M} of $\mathrm{M}_{e,b}$.
After a coordinate change we may also suppose without loss of generality that none of $p_1,\dots,p_b$ lie at infinity.  In particular, there exist
 $c_1,\dots,c_b\in \FF_q$ and  $\a_1,\dots,\a_b\in \FF_q^n$ such that
 $p_j=(c_j:1)$ and $y_j=(a_{1,j}:\dots:a_{n,j})$, for $1\leq j\leq b$.

\begin{defi}\label{def:1}
 Let $N(q,e;\a_1,\dots,\a_b)$ be the number of tuples  $\mathbf{g}\in \FF_q[t]^n$, of degree at most $e$, at least one of degree exactly $e$, with no common zero, such that $f_1(\g)=\dots=f_R(\g)=0$
 and $\mathbf{g}(c_j)=\a_j$, for $1\leq j\leq b$. We simply write 
$N(q,e)$ for this quantity when $b=0$.
 \end{defi}

\begin{prop}\label{unfree-bound}
We have   
$$
(q-1)\#\mathrm{M}_{e,b}(\mathbb F_q)=\sum_{\lambda_1,\dots,\lambda_{b}\in \FF_q^*}N(q,e;
\lambda_1 \a_1,\dots \lambda_b \a_b).
$$
\end{prop}

\begin{proof}
 Let $\mathrm N_{e,b}$ be the set of tuples  $\mathbf{g}=(g_1,\dots,g_n)\in \FF_q[t]^n$, of degree at most $e$, at least one of degree exactly $e$, with no common zero, such that $f_1(\g)=\dots=f_R(\g)=0$
 and $\mathbf{g}(c_j)\in \mathbb{G}_m(\FF_q)\a_j$, for $1\leq j\leq b$.
Taking
projective coordinates
$(g_1: \dots : g_n)$,
any $\bfg\in \mathrm N_{e,b}$  defines
a degree $e$ map $g:\mathbb P^1 \to X$ over $\FF_q$,
such that $g(p_j)=y_j$ for $1\leq j\leq b$.
 All such maps arise this way, and two tuples define the same map if and only if
 one is a scalar multiple of the other.
 Hence
 $$
\#\mathrm N_{e,b}=  \#\mathbb{G}_m(\FF_q) \#\mathrm{M}_{e,b}(\mathbb F_q)
=(q-1)\#\mathrm{M}_{e,b}(\mathbb F_q).
 $$
Finally, it is clear from the definition that  $\#\mathrm N_{e,b}$  is equal to the right hand side of the expression in the statement of the proposition. 
\end{proof}

\begin{prop}\label{unfree-bound2}
We have
$$
\#\mathcal M_{0,0}(X,e)(\FF_q)= \frac{N(q,e)}{(q-1)
(q^{3}-q)} .
$$
\end{prop}

\begin{proof}
Each point of $\mathcal M_{0,0}(X,e)$ corresponds to $\#\mathrm{PGL}_2(\mathbb F_q) = q^3-q$
points of $\Mor_e(\PP^1,X)$. The result now follows on  taking $b=0$ in Proposition \ref{unfree-bound}.
\end{proof}

We shall prove an explicit upper bound for the counting function defined in Definition \ref{def:1}. The   following result will be obtained by adapting the proof of Theorem~\ref{p2.1}.

\begin{theorem}\label{t:goat}
Let $d\geq 2, $ and $b\geq 0$ be integers, and recall the definition \eqref{eq:nu-d} of $\nu(d)$.
Let
$f_1,\dots,f_R\in \FF_q[x_1,\dots,x_n]$ be as above, let 
$c_1,\dots,c_b\in \FF_q$ and let $\a_1,\dots,\a_b\in \FF_q^n$.
Assume
 that $
n\geq  (d(d-1)2^{d+\nu(d)+1}+1)R$
and
$$
e\geq (d+1)b +
\max\left\{1,
bR\right\}.
$$
Then
$$
 \lim_{q\to \infty}  q^{
 -\left(e(n-dR) +(n-R)(1-b)\right)}
 N(q,e;\a_1,\dots,\a_b)\leq 1.
$$
\end{theorem}

\begin{proof}[Proof of Theorems \ref{t:BV} and \ref{t:BV2}]
Assume that
$n\geq  (d(d-1)2^{d+\nu(d)+1}+1)R$ and $e\geq 1$.
Once combined with Proposition
\ref{unfree-bound2},
 it follows from \cite[Eq.~(3.3)]{BV'} that the case $b=0$ of Theorem~\ref{t:goat} implies that  $\mathcal{M}_{0,0}(X,e)$ is irreducible and of the expected dimension, as
 required for Theorem \ref{t:BV}.
Appealing to the  same method used in \cite[p.~2]{HRS}, moreover, it also follows that   $\mathcal M_{0,0}(X,e)$ is locally a complete intersection, which thereby completes the proof
of Theorem \ref{t:BV}.
Arguing along the same lines, we may   deduce Theorem~\ref{t:BV2} through the union of
Proposition~\ref{unfree-bound} and Theorem \ref{t:goat}.
\end{proof}

It remains to prove Theorem \ref{t:goat}.
For now, let $e\geq 1$ be an arbitrary integer.
In all the estimates that follow, the implied constants are only allowed to depend on $d,n$ and $R$, unless indicated otherwise.
We shall drop the coprimality condition from the definition of $N(q,e,b)$, together with the constraint
that at least one of the polynomials in the tuple $g_1,\dots,g_n$ has exact degree $e$.
Moreover, for each index $1\leq j\leq b$,
the condition
$\mathbf{g}(c_j)=\a_j$ is equivalent to the congruence condition
$
\g\equiv \a_j \bmod{(t-c_j)}.
$
In this way, we see that
$$
N(q,e;\a_1,\dots,\a_b)\leq \hat N(q,e,b),
$$
where
$$
\hat N(q,e,b)=
\#\left\{
\mathbf{g}\in \FF_q[t]^n :
\begin{array}{l}
\deg g_1,\dots,\deg g_n \leq e\\
f_1(\g)=\dots=f_R(\g)=0\\
\text{$\mathbf{g}\equiv \a_j \bmod{(t-c_j)}$ for $1\leq j\leq b$}
 \end{array}
 \right\}.
$$
Here we recall that $c_1,\dots,c_b\in \FF_q$ and
$\a_1,\dots,\a_b\in \FF_q^n$
 are given, with $c_1,\dots,c_b$ all distinct.
 Since $\a_1,\dots,\a_b$ are vectors representing $\FF_q$-points on $X$, we also have  $f_k(\a_j)=0$ for $1\leq k\leq R$ and $1\leq j\leq b$.
  Let
 $$
 m=(t-c_1)\cdots(t-c_b).
 $$
This is a square-free monic  polynomial in $\FF_q[t]$ of degree $b$.
 It follows from the Chinese remainder theorem that there exists $\mathbf{s}\in \FF_q[t]^n$, with
$\gcd(\mathbf{s},m)=1$ and
 $\deg s_i<b$ for $1\leq i\leq n$, such that $f_i(\mathbf{s})\equiv 0 \bmod{m}$ for $1\leq i\leq R$ and 
 $$
\hat N(q,e,b)=
\#\left\{
\mathbf{g}\in \FF_q[t]^n :
\begin{array}{l}
\deg g_1,\dots,\deg g_n \leq e\\
f_1(\g)=\dots=f_R(\g)=0\\
\mathbf{g}\equiv \mathbf{s}  \bmod{m}
 \end{array}
 \right\}.
$$
It is now clear that  $\hat N(q,e,m)=N(\f,e+1-b,m,\mathbf{s})$, in the notation of \eqref{eq:defn-NP}. It follows from
\eqref{orthog} that
$$
\hat N(q,e,m)=
\int_{\TT^R}
S(\bal)\d \bal,
$$
where $S(\bal)=S(\bal;e+1-b,m,\mathbf{s})$ is given by \eqref{eq:sum-S}.

Recall the definition of the major arcs of level $J$ that were introduced in \eqref{defn major arc}.
Taking  $P=e+1-b$, these  take the shape
\begin{equation}\label{eq:MM-J}
\mathfrak{M}(J)=\bigcup_{\substack{g\in \FF_q[t] \text{ monic}\\
0<|g|\leq q^J
}}
\bigcup_{\substack{\a\in \FF_q[t]^R\\
|\a|<|gm^d|\\ \gcd(\a,g)=1}} \left\{\bal\in \TT^R: \left|gm^d\bal-\a\right|<q^{J-de+db}\right\}.
\end{equation}
Note that $\mathfrak{M}(-1)=\emptyset$. Moreover,  the function field
version of Dirichlet's theorem in \cite[Lemma 4.5.1]{lee}
shows that for any $\bal\in \TT^R$  there exist
monic $g\in \FF_q[t]$ and a vector $\a\in \FF_q[t]^R$,
with $|\a|< |gm^d|$ and $\gcd(g,\a)=1$,
 such that
$|g|\leq q^J$ and
$\left|gm^d\bal-\a\right|^R<q^{-J}$. Hence it follows that  $\mathfrak{M}(J)=\TT^R$ if $J\geq Rd(e-b)/(R+1)$.
We proceed by considering the contribution from the major arc of level $J=0$.

\begin{lemma}
We have
$$
\int_{\mathfrak{M}(0)}
S(\bal)\d \bal
= q^{e(n-dR) +(n-R)(1-b)}\left(1+O(
q^{-(n-R-1)/2}
)\right).
$$
\end{lemma}

\begin{proof}
We begin by observing that
\begin{align*}
\int_{\mathfrak{M}(0)}
S(\bal)\d \bal
&=
\sum_{\substack{\a\in \FF_q[t]^R\\
|\a|<|m^d|}}
\sum_{\substack{
\mathbf{g}\in \FF_q[t]^n \\
|\g|<q^{e+1}\\
\mathbf{g}\equiv \mathbf{s}  \bmod{m}}}
\psi\left(\frac{\a .\f(\g)}{m^d}\right)
\int_{|\bbe|<q^{-de} } \psi\left(\bbe. \f(\g)\right) \d \bbe\\
&=
\sum_{\substack{\a\in \FF_q[t]^R\\
|\a|<|m^d|}}
\sum_{\substack{
\mathbf{h}\in \FF_q[t]^n \\
|\h|<|m^d|\\
\mathbf{h}\equiv \mathbf{s}  \bmod{m}}}
\psi\left(\frac{\a .\f(\h)}{m^d}\right)
\sum_{\substack{
\mathbf{g}\in \FF_q[t]^n \\
|\g|<q^{e+1}\\
\mathbf{g}\equiv \mathbf{h}  \bmod{m^d}}}
\int_{|\bbe|<q^{-de} } \psi\left(\bbe. \f(\g)\right) \d \bbe,
\end{align*}
on breaking into residue classes modulo $m^d$.
The inner sum over $\g$ becomes
\begin{align*}
q^{-dRe}
\#\left\{
\mathbf{g}\in \FF_q[t]^n:
\begin{array}{l}
|\g|<q^{e+1}, ~ \mathbf{g}\equiv \mathbf{h}  \bmod{m^d}\\
|f_k(\g)|<q^{de} \text{ for $1\leq k\leq R$}
\end{array}
\right\},
\end{align*}
by \cite[Lemma 1(f)]{kubota}. Any $\g$ in the cardinality to be estimated can be written
as $\g=\g_0t^e+\k$ where $f_k(\g_0)=0$, for $1\leq k\leq R$, and $\k\in \FF_q[t]^n$ has norm
$|\k|< q^e$. Recall our assumption that $e\geq db+1$ in the statement of Theorem \ref{t:goat}. For
 fixed $\g_0$, we must count the number of $\k\in \FF_q[t]^n$ with
$|\k|< q^e$,  such that
 $\k$ lies in a fixed residue class modulo $m^d$. The number of such $\k$ is
 $(q^{e}/|m^d|)^n=q^{n(e-db)}$, since $\deg m=b$. Finally, the number of $\g_0$ is
 $q^{n-R}+O(q^{(n-R+1)/2})$ by Deligne's resolution of the Weil conjectures, as presented in
 Hooley \cite[Thm.~2 and Sec.~5]{hooley}, for example.
 Moreover, the implied constant in this estimate only depends on $d$ and $n$.
It therefore follows that
\begin{align*}
\sum_{\substack{
\mathbf{g}\in \FF_q[t]^n \\
|\g|<q^{e+1}\\
\mathbf{g}\equiv \mathbf{h}  \bmod{m^d}}}
\int_{|\bbe|<q^{-de} } \psi\left(\bbe. \f(\g)\right) \d \bbe
&=q^{-dRe} \cdot q^{n(e-db)} \cdot \left(q^{n-R}+O(q^{(n-R+1)/2})\right)\\
&= q^{e(n-dR) +n-dbn-R}\left(1+O(q^{-(n-R-1)/2}\right).
\end{align*}
Finally, on applying orthogonality of characters, we see that
$$
\sum_{\substack{\a\in \FF_q[t]^R\\
|\a|<|m^d|}}
\sum_{\substack{
\mathbf{h}\in \FF_q[t]^n \\
|\h|<|m^d|\\
\mathbf{h}\equiv \mathbf{s}  \bmod{m}}}
\psi\left(\frac{\a .\f(\h)}{m^d}\right)=q^{dbR} N^\dagger(m^d),
$$
in the notation of \eqref{eq:late}. But it  follows from
\eqref{eq:seminar} that
$$
N^\dagger(m^d)=q^{(d-1)(n-R)b} N^\dagger(m)=q^{(d-1)(n-R)b}.
$$
Combining everything together readily leads to the statement of the lemma.
\end{proof}

 In order to complete the proof of Theorem
\ref{t:goat}, it therefore suffices to show that
\begin{equation}\label{eq:unicorn}
\lim_{q\to \infty} q^{-\left(e(n-dR) +(n-R)(1-b)\right)}\sum_{J=0}^{\lceil Rd(e-b)/(R+1)\rceil-1}\left| \int_{\mathfrak{M}(J+1)\setminus \mathfrak{M}(J)} S(\bal)\d\bal\right| <1.
\end{equation}

Let $\ve>0$ and let
$\mathcal{C}$ be given by
\eqref{setcalC}, where $\nu(d)$ is given by \eqref{eq:nu-d}.
The work in Section \ref{sec4}
shows that
Hypothesis~\ref{hyp:2.1} holds with $C=(d-1)^{\frac{n}{2^d}}$.
It follows from \eqref{eq:0,ep}  that
\begin{equation}\label{eq:egg}
\frac{n-R+1}{2^{d+\nu(d)} (d-1)}>
\mathcal{C} >  \frac{n-R+1}{2^{d+\nu(d)} (d-1)}-\frac{n \ve}{2^d(d-1)}.
\end{equation}
Hence,
on assuming that
$\ve>0$ is chosen to be sufficiently small,
 it follows that $\mathcal{C}>dR$
 if
$n\geq   (d(d-1)2^{d+\nu(d)+1}+1)R$, which we henceforth assume.
(Although it actually suffices to take $n\geq   (d(d-1)2^{d+\nu(d)}+1)R$ here,
this would lead to a significantly  worse lower bound for $e$ in Theorem \ref{t:goat}.)

We need to study the integral
\begin{equation}\label{eq:def_IJ}
I_J=
\int_{\mathfrak{M}(J+1)\setminus \mathfrak{M}(J)} S(\bal)\d\bal,
\end{equation}
where $\mathfrak{M}(J)$ is given by  \eqref{eq:MM-J}.
We shall proceed differently, according to the size of $J$.  Our treatment for larger $J$ is summarised in the following result and is based on the treatment of the minor arcs in Section \ref{s:apply}.

\begin{lemma}\label{lem:IJ-large}
Assume that
$
n\geq
(d(d-1)2^{d+\nu(d)+1}+1)R$, where $\nu(d)$ is given by \eqref{eq:nu-d}.
Then
\begin{align*}
I_J\ll~&
q^{\left(n-\frac{n+1-R}{2^{d+\nu(d)}(d-1)}\right)(e+1-b)+O(\ve)}
+
  q^{1+dRb
+(n - dR)(e+1-b)
-\frac{n-R+1}{2^{d}}\left(1+\left\lfloor \frac{J}{R(d-1)}\right\rfloor\right)
+O(\ve)},
\end{align*}
where  the implied constant in the $O(\ve)$ term depends only on $d,e, n,R$ and $b$.
\end{lemma}

\begin{proof}
The plan is to  recycle our treatment of the minor arcs from Section
\ref{s:apply} with $P=e+1-b$.
We begin by applying Lemma \ref{lem2.5}. Suppose that
$\bal\in \TT^R\setminus \mathfrak{M}(J)$ and that  alternative (ii) holds for some $J'\leq e+1-b$. Then there exist $g\in\FF_q[t]$ and $\a\in \FF_q[t]^R$, with $g$ monic and  $\gcd(g,\a)=1$, such that
$$
0<  |g| \leq q^{R(d-1)(J'-1)} \quad \textnormal{and} \quad
|g m^d\bal- \a| <  q^{- d (e+1-b) + R(d-1) J'-(R-1)(d-1)}.
$$
We shall apply this with
\begin{equation}\label{eq:J'=}
J'=1+\left\lfloor \frac{J}{R(d-1)}\right\rfloor,
\end{equation}
which would
imply  that $\bal \in \mathfrak{M}(J)$, which is a contradiction.
We also need to check that $J'\leq e+1-b$, which certainly holds if
$$
\frac{J}{R(d-1)}\leq e-b.
$$
Since $J\leq
\lceil Rd(e-b)/(R+1)\rceil-1$, we see that it suffices to have
$$
\frac{Rd(e-b)}{R+1}
\leq (e-b)R(d-1),
$$
which is equivalent to demanding that
$d\leq (d-1)(R+1)$, an inequality that always holds for $d\geq 2$ and $R\geq 1$.
Hence we find ourselves  in alternative (i) of
Lemma \ref{lem2.5} and
it follows from Remark \ref{rem:sigma-f} that
$$
S(\bal)\ll q^{n(e+1-b)-\frac{(n-R+1)J'}{2^{d-1}}} =
q^{(n-\delta)(e+1-b)},
$$
for any $\bal\in \mathfrak{B}=\TT^R\setminus \mathfrak{M}(J)$,
with
$$
\delta=\frac{(n-R+1)J'}{(e+1-b)2^{d-1}}>0.
$$

We would now like to apply the second part of
Lemma \ref{lem2.2} with $P=e+1-b$, for which we note that the  assumed bound holds with
$
C_0\ll 1$. It remains to check that $\delta (e+1-b)\geq \mathcal{C}\left(1-\frac{db}{d-1}\right)$.
If $b\geq 1$ then this is trivial, since the
right hand side is negative. Assuming that $b=0$, on the other hand, then it follows from  \eqref{eq:egg} that
it suffices to have
$$
\frac{(n-R+1)J'}{2^{d-1}}\geq
\frac{n-R+1}{2^{d+\nu(d)} (d-1)},
$$
which is valid, since $J'\geq 1$.
We finally deduce from the
 second part of
Lemma \ref{lem2.2} that
\begin{equation}\label{eq:IJ-1}
I_J
\ll
q^{(n - \mathcal{C}) (e+1-b)}
+  q^{1+dRb+(n - dR - \delta(1 - \frac{ d R}{\mathcal{C}}) ) (e+1-b)}.
\end{equation}
Now
it follows from  \eqref{eq:egg} that
$$
q^{(n - \mathcal{C}) (e+1-b)} \leq
q^{\left(n-\frac{n+1-R}{2^{d+\nu(d)}(d-1)}+O(\ve)\right)(e+1-b)}.
$$
This is satisfactory for the lemma.
Next,
  \eqref{eq:egg} yields
$$
\delta \left(1-\frac{dR}{\mathcal{C}}\right)(e+1-b)\geq \frac{(n-R+1)J'\left(1-\frac{dR}{\mathcal{C}}\right)}{2^{d-1}},
$$
where
\begin{equation}\label{eq:juice}
1-\frac{dR}{\mathcal{C}}
=1-\frac{dR2^{d+\nu(d)}(d-1)}{n+1-R}+O(\ve)
\geq \frac{1}{2}+O(\ve),
\end{equation}
under our assumption on $n$.
The lemma follows on recalling the expression \eqref{eq:J'=} for $J'$
and putting these estimates together in \eqref{eq:IJ-1}.
\end{proof}

Our second estimate for $I_J$  is more appropriate for smaller values of $J$ and draws on the major arc analysis in Section \ref{s:apply}.

\begin{lemma}\label{lem:IJ-small}
Assume that $n\geq  dR$ and
 $$
0\leq
J\leq e-db-b-1.
$$
Then
\begin{align*}
I_J \ll_\ve~&
q^{(n-dR)(e+1-b)+
dR-R(d-1)b
-(J+1)\left((1-\frac{1}{d})(\frac{n}{d}-R)-1\right)
+O(\ve)}\\
&\quad +
q^{(n-dR)(e+1-b)-R(d-1)b+(1-\frac{2}{d})n+R
-(J+d)\left((1-\frac{1}{d})(\frac{n}{d}-R)-1\right)
+O(\ve)},
\end{align*}
where the implied constant in the $O(\ve)$ term depends only on $d,e, n,R$ and $b$.
\end{lemma}

\begin{proof}
Recall the definition
 \eqref{eq:MM-J} of the major arcs. By  \eqref{eq:overlap}, the major arcs of level $J+1$ are non-overlapping if $J+1\leq \frac{1}{2}d(e-b)$. Our assumption
$ J\leq e-db-b-1$ is sufficient  to ensure this property.
Hence
 \begin{equation}\label{eq:anna}
I_J=
\sum_{\substack{g\in \FF_q[t] \text{ monic}\\
0<|g|\leq q^{J+1}
}}
\sum_{\substack{\a\in \FF_q[t]^R\\
|\a|<|gm^d|\\ \gcd(\a,g)=1}}
\int_{R_J\setminus \mathfrak{M}(J)} S(\bal)\d\bal
\end{equation}
in \eqref{eq:def_IJ},
where
$$
R_J=\left\{\bal\in \TT^R: \left|\bal-\frac{\a}{gm^d}\right|<\frac{q^{J+1-de}}{|g|}\right\}.
$$
Since $\bal\not \in \mathfrak{M}(J)$, we may assume in \eqref{eq:anna} that
\begin{equation}\label{eq:on_major}
|g|=q^{J+1} \quad \text{ or } \quad
 \left|\bbe\right|=\frac{q^{J-de}}{|g|},
\end{equation}
where $\bbe= \bal-\frac{\a}{gm^d}$.
We proceed by applying Lemma \ref{lem2.4}, which yields
$$
S(\bal)=q^{n(e+1-b)} S_{gm^d}(\a)S_\infty (m^dt^{d(e+1-b)} \bbe),
$$
provided that the hypotheses of the lemma are met. The first hypothesis is the upper bound $|gm^d|\leq q^{e+1-b}$,  which is
clearly met if $J\leq  e-db-b$. The second hypothesis is the bound
$
|\bbe|<q^{-(d-1)(e-b)}|g|^{-1}|m|^{-2d},
$
for which it suffices to have
$$
q^{J+1-de}
\leq q^{-(d-1)(e-b)}|m|^{-2d}
= q^{-de-db+e-b}.
$$
Thus we see that both  hypotheses hold if $J\leq e-db-b-1$, which is what we assumed in the statement of the lemma.

We have therefore shown that
$$
I_J=
q^{n(e+1-b)}
\sum_{\substack{g\in \FF_q[t] \text{ monic}\\
0<|g|\leq q^{J+1}
}}
\sum_{\substack{\a\in \FF_q[t]^R\\
|\a|<|gm^d|\\ \gcd(\a,g)=1}}
S_{gm^d}(\a)
\int_{|\bbe|<q^{J+1-de}/|g|}
S_\infty (m^dt^{d(e+1-b)} \bbe)
\d\bbe,
$$
where $g$ and $\bbe$ are constrained to satisfy \eqref{eq:on_major}.
Making a  change of variables, we see that
$$
\int_{|\bbe|<q^{J+1-de}/|g|}
S_\infty (m^dt^{d(e+1-b)} \bbe)
\d\bbe
=
q^{-dRb-dR(e+1-b)}
\int_{|\bga|<q^{J+1+d}/|g|}
S_\infty ( \bga)
\d\bga.
$$
Hence
$$
I_J=
q^{(n-dR)(e+1-b)-dRb}
\sum_{\substack{g\in \FF_q[t] \text{ monic}\\
0<|g|\leq q^{J+1}
}}
\sum_{\substack{\a\in \FF_q[t]^R\\
|\a|<|gm^d|\\ \gcd(\a,g)=1}}
S_{gm^d}(\a)
\int_{|\bga|<q^{J+1+d}/|g|}
S_\infty (\bga)
\d\bga,
$$
where $g$ and $\bga$ are constrained to satisfy
$$
|g|=q^{J+1} \quad \text{ or } \quad
 \left|\bga\right|=\frac{q^{J+d}}{|g|}.
$$

Recalling the notation  \eqref{eq:def A} for $A(g)$,
it follows from the proof of \eqref{eq:jerk} that
$$
\sum_{ \substack{
\a\in \FF_q[t]^R\\
 |\a| <|gm^d|  \\ \gcd(\a, g) = 1 }}
S_{gm^d}(\a)
= \mathbf{1}_{\gcd(g,m)=1} |m|^R
A(g).
$$
Hence, on recalling the notation
  \eqref{eq:ST} and \eqref{eq:IT}, we may now write
$$
|I_J|\leq
q^{(n-dR)(e+1-b)-R(d-1)b}
\sum_{(i,j)\in \mathcal{E}_1\cup \mathcal{E}_2}
\mathcal{S}(i) \cdot |\mathcal{I}(j)|,
$$
where
\begin{align*}
\mathcal{E}_1&=\left\{(J+1,j)\in \ZZ^2: j\leq d-1\right\},\\
\mathcal{E}_2&=\left\{(i,J+d-i)\in \ZZ^2: 0\leq i\leq J+1\right\}.
\end{align*}
Lemma \ref{lem:toad}  implies
that
$$
\mathcal{S}(i)
\ll_\ve
q^{i\left(1-(1-\frac{1}{d})(\frac{n}{d}-R)+O(\ve)\right)},
 $$
 where the implied constant in $O(\ve)$ term depends only on $d,e,n,R$ and $b$.
Likewise,
Lemma \ref{lem:est2} yields
$$
 \mathcal{I}(j)\ll
 \begin{cases}
 q^{R+jR} &\text{ if $j\leq d-1$,}\\
  q^{(1-\frac{2}{d})n+R-j(1-\frac{1}{d})(\frac{n}{d}-R)} &\text{ if $j\geq d-1$.}
  \end{cases}
 $$

Let $I_J^{(k)}$ denote the overall contribution to $I_J$ from indices $(i,j)\in \mathcal{E}_k$, for $k\in \{1,2\}$.
On  noting that $j\leq d-1$ in $\mathcal{E}_1$, we find that
\begin{align*}
I_J^{(1)}
&\ll_\ve
q^{(n-dR)(e+1-b)+R-R(d-1)b}
\sum_{\substack{(i,j)\in \mathcal{E}_1}}
q^{-i\left((1-\frac{1}{d})(\frac{n}{d}-R)-1+O(\ve)\right)}\cdot q^{jR}\\
&\ll_\ve q^{(n-dR)(e+1-b)+
dR-R(d-1)b
-(J+1)\left((1-\frac{1}{d})(\frac{n}{d}-R)-1\right)
+O(\ve)}.
\end{align*}
Similarly,
\begin{align*}
I_J^{(2)}
&\ll_\ve
q^{(n-dR)(e+1-b)
-R(d-1)b+
(1-\frac{2}{d})n+R} \\
&\quad\quad \times
\sum_{0\leq i\leq J+1}
q^{-i\left((1-\frac{1}{d})(\frac{n}{d}-R)-1+O(\ve)\right)}
\cdot q^{
-(J+d-i)(1-\frac{1}{d})(\frac{n}{d}-R)}\\
&\ll_\ve
q^{(n-dR)(e+1-b)
-R(d-1)b
+(1-\frac{2}{d})n+R-(J+d)(1-\frac{1}{d})(\frac{n}{d}-R)
+O(\ve)}
\sum_{0\leq i\leq J+1}
q^{i}\\
&\ll_\ve
q^{(n-dR)(e+1-b)-R(d-1)b+(1-\frac{2}{d})n+R
-(J+d)\left((1-\frac{1}{d})(\frac{n}{d}-R)-1\right)
+O(\ve)}.
\end{align*}
Combining these estimates, we
 therefore complete the proof of the lemma.
\end{proof}

It is now time to use this pair of results to prove  that   \eqref{eq:unicorn} holds under the assumptions of Theorem \ref{t:goat}.  In particular, we have  $d\geq 2$ and $n\geq  (d(d-1)2^{d+\nu(d)+1}+1)R$, where $\nu(d)$ is given by \eqref{eq:nu-d}. Moreover,
we begin by assuming that
$e\geq (d+1)b+1$.
Write
$$
J_0=e-db-b-1 \quad \text{ and } \quad
J_1=\lceil Rd(e-b)/(R+1)\rceil-1,
$$
and note that
$0\leq J_0\leq J_1 \ll e$.

It follows from  Lemma \ref{lem:IJ-large} that
\begin{align*}
\sum_{J=J_0+1}^{J_1} I_J
\ll_\ve~&
e q^{\left(n-\frac{n+1-R}{2^{d+\nu(d)}(d-1)}\right)(e+1-b)+O(\ve)}\\
&\quad +
  eq^{1+dRb
+(n - dR)(e+1-b)+O(\ve)}
q^{
-\frac{n-R+1}{2^{d}}\left(1+\left\lfloor \frac{J_0+1}{R(d-1)}\right\rfloor\right)}.
\end{align*}
Likewise, Lemma \ref{lem:IJ-small} implies that
\begin{align*}
\sum_{J=0}^{J_0} I_J
\ll_\ve~&
q^{(n-dR)(e+1-b)+
dR-R(d-1)b+1
-(1-\frac{1}{d})(\frac{n}{d}-R)
+O(\ve)}\\
& +
q^{(n-dR)(e+1-b)-R(d-1)b+(1-\frac{2}{d})n+R+d
-(d-1)(\frac{n}{d}-R)
+O(\ve)}.
\end{align*}

Putting these together,
we may conclude that  \eqref{eq:unicorn} holds, as desired,
provided that our parameters $d,n,m,R$ and $e$ are chosen in such a way that
the exponents of $q$ are all  strictly less than
$$
e(n-dR)+(n-R)(1-b)
=(n-dR)(e+1-b) +(d-1)R(1-b).
$$
But, on choosing $\ve>0$ sufficiently small, this is equivalent to demanding the veracity of the four  inequalities
\begin{equation}\label{eq:FIN1}
\left(n-\frac{n+1-R}{2^{d+\nu(d)}(d-1)}\right)(e+1-b)<(n-dR)(e+1-b)
+(d-1)R(1-b),
\end{equation}
\begin{equation}\label{eq:FIN2}
1+dRb
-\frac{n-R+1}{2^{d}}\left(1+\left\lfloor \frac{J_0+1}{R(d-1)}\right\rfloor\right)
< (d-1)R(1-b),
\end{equation}
\begin{equation}\label{eq:FIN4}
1-\left(1-\frac{1}{d}\right)\left(\frac{n}{d}-R\right)
<
 -R,
\end{equation}
and
\begin{equation}\label{eq:FIN5}
\left(1-\frac{2}{d}\right)n+R+d-(d-1)\left(\frac{n}{d}-R\right)<
 (d-1)R.
\end{equation}
Our  assumptions on $d$ and $n$ plainly ensure that both
 \eqref{eq:FIN4} and
  \eqref{eq:FIN5} hold.
Moreover,
 \eqref{eq:FIN1} holds if and only if
$$
(d-1)R(b-1)<(e+1-b)\left(\frac{n+1-R}{2^{d+\nu(d)}(d-1)}-dR\right).
$$
Since
$e\geq (d+1)b+1\geq 2b+1$ and
$n\geq  (d(d-1)2^{d+\nu(d)+1}+1)R$, the right hand side is
$>(d-1)R(b-1)$, as required.
Thus \eqref{eq:FIN1} holds.

It remains to check \eqref{eq:FIN2}, for which we
note that
$$
1+\left\lfloor \frac{J_0+1}{R(d-1)}\right\rfloor\geq \frac{J_0+2}{R(d-1)}.
$$
Recalling that $J_0= e-db-b-1$
and $n\geq  (d(d-1)2^{d+1}+1)R$,
we see that \eqref{eq:FIN2} holds if
$$
2d(e-db-b+1)\geq 1-R(d-1)+b(2d-1)R.
$$
Since $d\geq 2$, it  suffices to have
$$
e\geq \left(d+1+R\right)b.
$$
The proof of Theorem \ref{t:goat} is now complete.

\section{Low degree  rational curves}

The main goal of this section is to prove Theorem \ref{thm:low}.
Let $X\subset \PP^{n-1}$ be a smooth complete intersection over $\FF_q$,
cut out by $R$ hypersurfaces of  degree $d$.
 We shall always assume that
$\mathrm{char}(\FF_q)>d$ and that $n$ satisfies the bound in   \eqref{eq:n-bound}.
As usual, we denote by
$f_1, \ldots, f_R \in \FF_q[x_1, \ldots, x_n]$ the relevant  forms of degree $d\geq 2$.

Recall
$N(q,e)$ from
Definition  \ref{def:1}, which counts the number of tuples  $\mathbf{g}\in \FF_q[t]^n$, of degree at most $e$, at least one of degree exactly $e$, with no common zero, such that $f_1(\g)=\dots=f_R(\g)=0$. It follows from Proposition \ref{unfree-bound2}  that
\begin{equation}\label{eq:straw}
\#\mathcal M_{0,0}(X,e)(\FF_q)= \frac{N(q,e)}{(q-1)
(q^{3}-q)} .
\end{equation}
In order to prove
Theorem \ref{thm:low} it will suffice to show that the right hand side is positive if $q$
is sufficiently large in terms of $d$ and $n$, and $e\geq 2(d-1)R+3d$.

\subsection{Counting in the affine cone}

Let $P\in \NN$. For now we simply assume that $q\geq (d-1)^n$ and we fix the same choice of
$\mathcal{C}$ in  \eqref{setcalC}.  In particular, we have $\mathcal{C}>dR$ under the assumption    \eqref{eq:n-bound},
and we saw at the close of Section \ref{sec4} that
Hypothesis \ref{hyp:2.1} holds with $C=O(1)$.

We begin by focusing on the counting function
$N(\mathbf{f};P)
=
N(\mathbf{f};P, 1,\0)$ defined in \eqref{eq:defn-NP}, with $m=1$ and $\b=\0$.
Assume that
\begin{equation}\label{eq:P-assume}
P>(d-1)R.
\end{equation}
Then
it follows from Theorem \ref{p2.1} that
\begin{equation}\label{eq:p21}
N(\mathbf{f};P)= \mathfrak{S} \mathfrak{I} q^{(n - dR)P}+
O_\ve\left(E(P)\right),
\end{equation}
for any $\ve>0$,
with
\begin{align*}
E(P)=~& q^{(n-\mathcal{C})P}
+
q^{1+dR +\delta_0(1-\frac{dR}{\mathcal{C}})+ (n - dR - \delta_1) P}
+q^{(n - d R-\delta_2+\ve)P +n +\delta_2(1+R(d-1))},
\end{align*}
and where
\begin{equation}\label{eq:glass}
\delta_0 =
\frac{n - \sigma_\f}{ (d-1)2^{d-1}R  }, \quad
\delta_1= \delta_0 \left( 1 - \frac{dR}{\mathcal{C}}  \right), \quad
\delta_2=\left(1-\frac{1}{d}\right)\left(\frac{n}{d}-R\right)-1.
\end{equation}
Here, the implicit constant depends only on
 $d, n,R$ and $\ve$.

The singular series $\mathfrak{S}$ is defined in \eqref{def ss} and the singular integral $\mathfrak{I}$ in \eqref{def si}.
We  shall need to develop precise estimates for these quantities.

\begin{lemma}
\label{ssssssssss}
Assume $n> 2(d-1)R$ and  that $q$ is sufficiently large with respect to $d$ and $n$. Then
$
\mathfrak{S} = 1 +  O( q^{ 
  - \frac{n-R-3 }{2} 
  } ).
$
\end{lemma}

\begin{proof}
For any  prime $\pi \in \FF_q[t]$,  define
$$
\sigma(\pi) = \sum_{j = 0}^{\infty} \sum_{ \substack{ |\a| < | \pi^j |  \\ \gcd(\a, \pi) = 1 } } S_{ \pi^j }(\a),
$$
in the notation of \eqref{eq:D1}. In the light of the notation \eqref{eq:def A}, and the assumed
lower bound $n> 2(d-1)R$,
it follows from
Lemma \ref{lem:111} that this sum converges. Moreover,
we clearly have
$$
\mathfrak{S} = \prod_{\pi}  \sigma(\pi)
$$
in \eqref{def ss}.
Using Ramanujan sums and recalling that $m=1$, it  is routine to check that
$$
\sigma(\pi) = \lim_{k \to \infty }  |\pi|^{ - k (n - R)} N (\pi^k),
$$
in the notation of \eqref{eq:fan},
where
$$
N(\pi^k)=\#\left\{\z\in \FF_q[t]^n: |\z|<|\pi^k|,~ f_i(\z)\equiv 0\bmod{\pi^k} \text{ for $1\leq i\leq R$}\right\}.
$$
It will also  be convenient to recall that
$$
N^*(\pi^k)=\#\left\{\z\in \FF_q[t]^n: |\z|<|\pi^k|,~ \pi\nmid \z, ~f_i(\z)\equiv 0\bmod{\pi^k} \text{ for $1\leq i\leq R$}\right\}.
$$

We proceed by adapting the proof of Lemma \ref{lem:111} to prove that
\begin{equation}
\label{eq:somelem}
N(\pi^e)
=
\frac{N^*(\pi)}{ |\pi|^{n-R} } C_{\infty}(\pi)
|\pi|^{(n-R)e}
+O\left(|\pi|^{(1-\frac{1}{d})en}\right),
\end{equation}
where
$$
C_{\infty}(\pi ) =   \sum_{ j \geq 0 } |\pi|^{- j ( n - dR) } .
$$
First, we deduce from  \eqref{eq:Hensel}  that
$
N^*(\pi^k) =   |\pi|^{ (n - R) (k-1)  }   N^*(\pi),
$
for any $k \geq 1$.  It now follows from \eqref{eq:spoon}
that 
\begin{align*}
N(\pi^e)
&=
N^*(\pi)  |\pi|^{(n-R)(e-1)}
\sum_{\substack{0\leq j< e/d
}}
|\pi|^{(d-1)jn - (n-R) dj }
+O\left(|\pi|^{(1-\frac{1}{d})en}\right)
\\
&=
N^*(\pi)  |\pi|^{(n-R)(e-1)}
\sum_{\substack{0\leq j< e/d
}}
|\pi|^{- j ( n - dR) }
+O\left(|\pi|^{(1-\frac{1}{d})en}\right),
\end{align*}
for $e>d$.
It is clear  that
$$
\left| C_{\infty}(\pi) -  \sum_{\substack{0\leq j< e/d
}}
|\pi|^{- j ( n - dR) }  \right|
\ll
|\pi|^{- \frac{e}{d} ( n - dR) }.
$$
Putting these together therefore establishes  the claimed bound \eqref{eq:somelem}.

Note that
\begin{equation}\label{eq:pond}
C_{\infty}(\pi ) =1 + O( |\pi|^{- ( n - dR) }).
\end{equation}
Moreover, according to  \eqref{eq:case1}, we have 
$N^*(\pi) =|\pi|^{n-R }(1+O( |\pi|^{ -\frac{n-R-1 }{2}}))$.
It now follows
from
\eqref{eq:somelem} and \eqref{eq:pond} that
\begin{align*}
\sigma(\pi) &=   \lim_{k \to \infty }  |\pi|^{ - k (n - R) }
\left( \frac{N^*(\pi)}{ |\pi|^{n-R} } C_{\infty}(\pi)
|\pi|^{(n-R)k}
+O (|\pi|^{(1-\frac{1}{d})kn  }) \right)
\\
&= \frac{N^*(\pi)}{ |\pi|^{n-R} } C_{\infty}(\pi)\\
&
= 1 + J(\pi),
\end{align*}
where $J(\pi)=|\pi|^{ - \frac{n-R-1 }{2} }$.
Moreover,  the implicit constant depends only on $d$ and $n$.
Note that
\begin{align*}
\sum_{\pi} J(\pi)
\ll
\sum_{\pi}  q^{ - (\frac{n-R-1 }{2} )  \deg \pi }
&\ll
\sum_{j = 1}^{\infty}  q^{ j   - (\frac{n-R-1 }{2} )  j  }
\ll
q^{   - \frac{n-R-3 }{2}   } .
\end{align*}

Let us assume $q$ is sufficiently large with respect to $d$ and $n$, so that we can use Taylor  expansion to deduce that
$
\log ( 1 + J(\pi)  ) = O( |J(\pi)| ),
$
for an absolute implied  constant.
Then it follows that
\begin{align*}
\mathfrak{S} &=
\prod_\pi \left(1 + J(\pi) \right)
\\
&= \exp \left(   \sum_{  \pi }  \log ( 1 + J(\pi) ) \right)
\\
&=
\exp \left(  O  \left(  \sum_{  \pi } |J(\pi)|   \right)   \right)
\\
&=
\exp \left(  O  (  
q^{   - \frac{n-R-3 }{2}   } )   \right)
\\
&=
1 +    O  (  q^{   - \frac{n-R-3 }{2}  }   ),
\end{align*}
where we again used the fact that $q$ is sufficiently large to make use of the Taylor series expansion of $\exp$.
\end{proof}

\begin{lemma}
\label{intttt}
Assume $n\geq 2(d-1)R$. Then
$
\mathfrak{I} = q^{R( d - 1 )} (1+ O(  q^{ -    \frac{n-R-1 }{2} }  )).
$
\end{lemma}

\begin{proof}
Let $T  >  2d - 1$. Firstly, it follows from  Lemma \ref{si} that
$$
| \mathfrak{I} - \mathfrak{I}(T)| \ll q^{n-(1-\frac{1}{d})(\frac{n}{d}-R)T}.
$$
It also follows from combining the proof of Lemma \ref{lem:est2} with
\eqref{eq:somelem} that
\begin{align*}
\mathfrak{I}(T) &= q^{RT} m (T)
\\
&= q^{(d - 1 - T)n +  RT} N( t^{T + 1 - d} )
\\
&= q^{(d - 1 - T)n +  RT} \left( \frac{N^*(t)}{ q^{n-R} } C_{\infty}(t)
q^{(n-R)(T + 1 - d)}
+O\left(q^{(1-\frac{1}{d})(T + 1 - d)n}\right) \right)
\\
&=   \frac{N^*(t)}{ q^{n-R} } C_{\infty}(t) q^{R( d - 1 )}
+O\left(
q^{-\frac{T}{d} (n -d R )
-\frac{1}{d}(1-d)n}
\right).
\end{align*}
But  (\ref{eq:case1}) and \eqref{eq:pond} yield
$$
\frac{N^*(t)}{ q^{n-R} } C_{\infty}(t)
= 1 + O(  q^{  - \frac{n-R-1 }{2} }   ),
$$
under the assumption on $n$ in the lemma.
Therefore, on choosing $T$ to be sufficiently large, we easily arrive at the desired estimate.
\end{proof}

\subsection{M\"obius inversion}

Now that we have an estimate for $N(\mathbf{f};P)$ in
\eqref{eq:p21}, we may go on to deduce an  estimate for the
quantity
$$
N^*(\mathbf{f}; P) = \#\{ \mathbf{g} \in \FF_q[t]^n_{\textnormal{prim}}: 1 \leq  |\mathbf{g}| < q^P, f_1(\mathbf{g}) = \cdots = f_R(\mathbf{g}) = 0 \},
$$
in which $\mathbf{g} \in \FF_q[t]^n_{\textnormal{prim}}$ means that
$\gcd(g_1,\dots,g_n)=1$.
As usual we shall assume that $q$ is sufficiently large in terms of $d$ and $n$.
We shall also need to  assume  that $P$ satisfies the inequality
\eqref{eq:P-assume}.

Let $\mu:\FF_q[t]\to \{0,\pm 1\}$ denote the function field version of the M\"{o}bius function.
Then  M\"{o}bius inversion yields
\begin{align*}
N^*(\mathbf{f}; P)
&=
\sum_{ \substack{ k \in \mathbb{F}_q[t]  \\   k \, \textnormal{monic} \\ |k| < q^P   }  }
\mu(k)
\#\{ \mathbf{g} \in \FF_q[t]^n: 1 \leq  |\mathbf{g}| < q^{P - \deg k}, f_1(\mathbf{g}) = \cdots = f_R(\mathbf{g}) = 0 \}
\\
&=
\sum_{ \substack{ k \in \mathbb{F}_q[t]  \\   k \, \textnormal{monic} \\ |k| < q^P }  }
\mu(k) ( N(\f; P - \deg k)  -  1).
\end{align*}
We will need to truncate the sum over $k$, so that $P-\deg k$ satisfies the constraint
\eqref{eq:P-assume}. Thus
\begin{align*}
N^*(\mathbf{f}; P)
&=
\sum_{ \substack{ k \in \mathbb{F}_q[t]  \\   k \, \textnormal{monic} \\ |k| < q^{P-(d-1)R} }  }
\mu(k)  N(\f; P - \deg k)  +O\left( q^P +L(P)\right),
\end{align*}
where
$$
L(P)=
\sum_{ \substack{ k \in \mathbb{F}_q[t]  \\  k \, \textnormal{monic} \\ |k| \geq  q^{P-(d-1)R} }  }
\hspace{-0.3cm}
N(\f; P - \deg k)
\leq
\sum_{ \substack{ k \in \mathbb{F}_q[t]  \\  k \, \textnormal{monic} \\ |k| \geq  q^{P-(d-1)R} }  }
\hspace{-0.3cm}
\left(\frac{q^{P}}{|k|}\right)^n
\ll q^{(d-1)R(n-1)+P}.
$$
On appealing to \eqref{eq:p21}, we therefore deduce that
$$
N^*(\mathbf{f}; P)
=
 \mathfrak{S} \mathfrak{I} q^{(n - dR) P  }  \sum_{ \substack{ k \in \mathbb{F}_q[t]  \\   k \, \textnormal{monic} \\ |k| < q^{P-(d-1)R} }  }
\mu(k)q^{- (n - dR)\deg k }
+ O \left( q^{(d-1)R(n-1)+P} +  E(P)    \right),
$$
since
$\sum_{k}E (P - \deg k)\ll E(P)$.

We see that
$$
\sum_{ \substack{ k \in \mathbb{F}_q[t]  \\   k \, \textnormal{monic} \\ |k| < q^{P-(d-1)R} }  }
\mu(k)q^{- (n - dR)\deg k }
= 1 +
O \left( \sum_{1 \leq j < P} q^{j - (n - dR) j }   \right)
=
1 +
O (  q^{ 1 - (n - dR)  } ).
$$
Our assumption \eqref{eq:n-bound} implies that  $n> 2(d-1)R$. Hence, on combining these estimates with Lemmas \ref{ssssssssss} and \ref{intttt}, we obtain
\begin{align*}
N^*(\mathbf{f}; P)
&=
\mathfrak{S} \mathfrak{I} q^{(n - dR) P  } (1 +
O (  q^{ 1 - (n - dR)  } )) + O ( q^{(d-1)R(n-1)+P}+E(P) )
\\
&=
q^{(n - dR) P  +( d - 1 )R}
(
1
 +
O( q^{ - \frac{n-R-3 }{2}   }  ))
 + O ( q^{(d-1)R(n-1)+P}+E(P) )
\\
&=
q^{(n - dR) P+( d - 1 ) R } \left(1 +O ( \widetilde{E} (P) )\right),
\end{align*}
where
\begin{align*}
\widetilde{E}(P)=~&
q^{(d-1)R(n-2)
-(n-dR-1)P}
+
q^{     - \frac{n-R-3 }{2}  } + q^{(dR-\mathcal{C})P-(d-1)R}\\
&+
q^{1+R +\delta_0(1-\frac{dR}{\mathcal{C}})- \delta_1 P}
+q^{-\delta_2(P-1-(d-1)R)+\ve P+n-(d-1)R}.
\end{align*}

\subsection{Completion of the proof}

We are now ready to conclude the proof
of Theorem~\ref{thm:low}, based on the expression \eqref{eq:straw}.
Assuming that $q$ is sufficiently large with respect to $d$ and $n$,
it is clear from our work above that
\begin{align*}
N(q,e)
&=N^*(\mathbf{f}; e+1)-N^*(\mathbf{f}; e)\\ &=
q^{(n - dR) (e+1)+( d - 1 )R  }\left(1-q^{-(n-dR)}\right)
+O ( q^{(n - dR) (e+1)+( d - 1 )R}  \widetilde{E} (e+1) ),
\end{align*}
under the assumption that $e>(d-1)R$.

Clearly $n-dR>0$, $\mathcal{C}>dR$  and
$\frac{n-R-3 }{2} >0$ under our assumption on $n$ in \eqref{eq:n-bound}.
Thus, in the light of our remarks  on the size of $q$,
we may conclude that
$\#\mathcal M_{0,0}(X,e)(\FF_q)>0$ if  the  three inequalities
\begin{equation}\label{eq:karikal}
\begin{split}
(n-dR-1)(e+1)&>(d-1)R
(n-2),\\
\delta_1 (e+1)&>
1+R +\delta_0\left(1-\frac{dR}{\mathcal{C}}\right),\\
\delta_2(e-(d-1)R)+(d-1)R&>\ve(e+1)+n,
\end{split}
\end{equation}
all hold, for any $\ve>0$.  Recalling \eqref{eq:glass}, the second inequality is equivalent to
$$
e>
\frac{1+R}{\delta_1}.
$$
Now it follows from our assumption \eqref{eq:n-bound} that
$n\geq d(d-1)2^{d+\nu(d)+1}R+R$, where $\nu(d)$ is given by \eqref{eq:nu-d}.
In particular $1-\frac{dR}{\mathcal{C}}\geq \frac{1}{2}+O(\ve)$, for any $\ve>0$,  by \eqref{eq:juice}.
But then, since
Remark \ref{rem:sigma-f} yields
 $\sigma_\f\leq R-1$,
 it follows from \eqref{eq:glass} that
$
\delta_1\geq d2^{\nu(d)+1} +O(\ve) \geq d,
$
on choosing $\ve$ sufficiently small. We conclude that the second inequality in \eqref{eq:karikal} is satisfied if  $e\geq R+1$.
This is implied by our assumption $e>(d-1)R$.
We now make the more stringent assumption
$e\geq 2(d-1)R$, under which the first inequality is  obvious.
Finally, we
make the even more stringent assumption
that
$e\geq  2(d-1)R+3d$ and we check the third inequality in 
 \eqref{eq:karikal}
with $\ve=\delta_2\ve'$, for any $\ve'>0$, which is then  equivalent to 
$$
\delta_2\left(e(1-\ve')-(d-1)R-\ve'\right)+(d-1)R>n.
$$
It follows from \eqref{eq:glass} that
$$
\delta_2=\left(1-\frac{1}{d}\right)\left(\frac{n}{d}-R\right)-1\geq 
\frac{n-dR-2d}{2d},
$$
since  $1-\frac{1}{d}\geq \frac{1}{2}$
for $d\geq 2$. Hence, 
since $e\geq  2(d-1)R+3d$, we obtain
\begin{align*}
\delta_2\left(e(1-\ve')-(d-1)R-\ve'\right)
&\geq 
\frac{n-dR-2d}{2d}\cdot 
\left((1-2\ve')(d-1)R+(1-\ve')3d-\ve'\right).
\end{align*}
Taking $R\geq 1$
and 
$\ve'$ sufficiently small, 
it therefore  follows that  
\begin{align*}
\delta_2\left(e(1-\ve')-(d-1)R-\ve'\right)+(d-1)R
&\geq \frac{(4d-1)(n-dR-2d)}{2d} -O(\ve').
\end{align*}
This is  greater than $n$ under our assumptions on $n$.
This therefore completes the proof of Theorem \ref{thm:low}.

\end{document}